\documentclass{Class-ICCM}
\usepackage{times}
\usepackage{amsmath,amssymb,amsthm,mathrsfs,epsf,a4,color,graphicx,eucal}

\theoremstyle{plain}

\theoremstyle{definition}
\newtheorem{definition}{Definition}
\newtheorem{proposition}{Proposition}
\newtheorem{lemma}{Lemma}
\newtheorem{remark}{Remark}

\usepackage{psfrag}
\usepackage{epsfig}

\newcounter{myfigure}
\newenvironment{my-picture}[3]{\refstepcounter{myfigure}\label{#3}\setlength{\unitlength}{\textwidth}\begin{picture}(#1,#2)}{\end{picture}}

\newcommand{\ITEM}[2]{\parbox[t]{.05\textwidth}{#1}\hfill\parbox[t]{.95\textwidth}{#2}\vspace*{.8mm}}
\newcommand{\DDD}[3]{\begin{array}[t]{c}#1\vspace*{-.5em}\\_{#2}\vspace*{-.4em}\\_{#3}\end{array}}

\newcommand\R{\mathbb R}
\newcommand\N{\mathbb N}

\newcommand\DT[1]{\mathchoice
                 {{\buildrel{\hspace*{.1em}\text{\LARGE.}}\over{#1}}}
                 {{\buildrel{\hspace*{.1em}\text{\Large.}}\over{#1}}}
                 {{\buildrel{\hspace*{.1em}\text{\large.}}\over{#1}}}
                 {{\buildrel{\hspace*{.1em}\text{\large.}}\over{#1}}}}

\newcommand\bbC{\mathbb C}
\newcommand\bbK{\mathbb K}
\newcommand\calE{\mathscr E}
\newcommand\calR{\mathscr R}
\newcommand\GC{\Gamma_{\mbox{\tiny\rm C}}}
\newcommand\SC{\Sigma_{\mbox{\tiny\rm C}}}

\newcommand{\SDir}{\Sigma_{\mbox{\tiny\rm D}}}
\newcommand{\GDir}{\Gamma_{\mbox{\tiny\rm D}}}
\newcommand{\GNeu}{\Gamma_{\mbox{\tiny\rm N}}}

\newcommand{\Dir}{{\mbox{\tiny\rm D}}}

\newcommand{\tto}{\rightrightarrows}

\newcommand{\weak}{\rightharpoonup}
\newcommand{\weaks}{\stackrel{*}{\rightharpoonup}}

\renewcommand{\d}{\mathrm{d}}

\begin{document}
\begin{frontmatter}
\author{T.~Roub\'\i\v cek\corref{cor1}\fnref{label1}}
\ead{roubicek@karlin.mff.cuni.cz}
 \address{Mathematical Institute, Charles University, Sokolovsk\'a 83, CZ-186~75~Praha~8, Czech Republic}
 \address{Institute of Thermomechanics of the ASCR, Dolej\v skova 5, CZ-182 00 Praha 8, Czech Republic}

 \author{C.G.~Panagiotopoulos\corref{cor2}\fnref{label2}}
 \ead{cpanagiotopoulos@us.es}
 
 \author{V.~Manti\v c\corref{cor3}\fnref{label2}}
 \ead{mantic@etsi.us.es}
 \address{Group of Elasticity and Strength of Materials, Department of Continuum Mechanics, School of Engineering, University of Seville, Camino de los Descubrimientos s/n \\ES-410 92 Sevilla, Spain}

\fntext[label1]{Support from Universidad de Sevilla and the grants 201/09/0917 and 
201/10/0357 (GA \v CR), together with the institutional support RVO:\,61388998
(\v CR) is acknowledged.}
\fntext[label2]{Support from the Junta de Andaluc\'{\i}a and Fondo Social Europeo 
(Proyecto de Excelencia TEP-4051) is acknowledged. Also the support 
from Ministerio de Ciencia e Innovaci\'on (Proyecto MAT2009-14022) is 
acknowledged.}

\title{Quasistatic adhesive contact of visco-elastic bodies
and its numerical treatment for very small viscosity}

\begin{abstract}
An adhesive unilateral contact of elastic bodies with a small viscosity 
in the linear Kelvin-Voigt rheology at small strains
is scrutinized. The flow-rule for debonding the adhesive is
considered rate-independent and unidirectional,
and inertia is neglected. The asymptotics for the viscosity approaching 
zero towards purely elastic material
involves a certain defect-like measure recording in some
sense natural additional energy dissipated in the bulk
due to (vanishing) viscosity,
which is demonstrated on particular 2-dimensional computational 
simulations based on a semi-implicit time discretisation and 
a spacial discretisation implemented by boundary-element method.
\end{abstract}


\end{frontmatter}

\section{Introduction, quasistatic delamination problem}\label{sec-intro}

Quasistatic inelastic processes on surfaces of (or interfaces between) solid 
elastic bodies like  fracture or {\it delamination} (or {\it debonding}) of 
adhesive contacts have received intensive engineering and mathematical scrutiny 
during past decades. Often, the time scale 
of such processes is much faster than the external loading time scale, 
and such processes are then modelled as rate independent, 
which may bring theoretical and computational advantages. Yet,
the above mentioned inelastic phenomena typically lead to sudden jumps
during evolution, which is related with the attribute of nonconvexity of 
the governing stored energy (cf.\ here the non-convex term 
$\int_{\GC}\frac12z\bbK u{\cdot}u\,\d S$ in \eqref{engr-ineq-E} below), 
and then it is not entirely clear which concept of solutions suits well 
for the desired specific application. 

The ``physically'' safe way to coup with this problem is to reduce 
rate-independency 
on only such variables with respect to which the stored energy is convex,
the resting ones being subjected to certain viscosity (or possibly 
also inertia). Here we neglect inertia from the beginning, which is 
addressed as a quasistatic problem; cf.\ \cite[Sect.\,5]{Roub??ACVE} 
for the dynamical case. Moreover, such viscosities can be (and in most 
materials also are) very small, and in engineering literature are 
almost always completely neglected. However, 
although arbitrarily small, such viscosities are critically important
to keep energetics valid. It therefore makes a sense to investigate 
the asymptotics towards purely elastic materials when these viscosities 
vanish. In the limit, we thus get some solutions of the underlying 
rate-independent system which, however, might (and, in specific 
applications, intentionally should) be different from solutions 
arising when viscosity are directly zero and global-energy-minimization 
principle is in play, cf.\ also Remark~\ref{rem-stress-vs-energy} below.

In this article, we will confine ourselves to visco-elastic bodies 
{\it at small strains} and we consider the viscosity in the 
{\it Kelvin-Voigt rheology}, which is the simplest rheology 
which makes the desired effect of natural prevention of the too-early 
delamination, cf.\ \cite{Roub??ACVE}. It should also be emphasized that our 
viscosity is in the bulk while the inelastic delamination itself is considered 
fully rate-independent, in contrast to a usual vanishing-viscosity approach as 
e.g.\ in 
\cite{DDMM08VVAQ,EfeMie06RILS,Fias??VVAQ,KnMiZa08ILMC,LazToa??MCPB,MiRoSa09?BVSV,ToaZan09AVAQ}.
A certain bulk viscosity but acting on displacement itself rather than
on the strain was considered in \cite{Cagn08VVAF}.

For notational simplicity, we consider a single visco-elastic body occupying 
a bounded Lipschitz domain $\Omega\subset\R^d$ and the adhesive contact on 
a part $\GC$ of the boundary $\partial\Omega$, so that 
we consider $\partial\Omega=\GC\cup\GDir\cup\GNeu\cup N$ with
disjoint relatively open $\GC$, $\GDir$, and $\GNeu$ subsets of 
$\partial\Omega$ and with $N$ having a zero $(d{-}1)$-dimensional measure. 
All results are, however, valid equally for delamination on boundaries inside 
$\Omega$, i.e.\ an adhesive contact between several visco-elastic bodies.
For readers' convenience, let us summarize the notation used below:\medskip

\medskip

\hspace*{2.4em}\fbox{
\begin{minipage}[t]{0.48\linewidth}\small
$d$ dimension of the problem ($d=2,3$),

$u$ displacement (defined on $\Omega$),

$z$ delamination parameter (defined on $\GC$),

$e(u)=\frac12(\nabla u)^\top\!+\frac12\nabla u\ $ small-strain tensor,

$\bbC$ tensor of elastic moduli of the 4$^{\rm th}$-order,

$\chi$ a ``Kelvin-Voigt'' relaxation time,

$\chi\bbC$ viscous-moduli tensor,

$\epsilon=e(\chi\DT{u}+u)$

$\bbK$ the matrix of elastic moduli of the adhesive,

\end{minipage} 
\begin{minipage}[t]{0.35\linewidth}\small

$E$ Young modulus,

$\nu$ Poisson ratio,

$\alpha$ fracture toughness,

$\mathfrak d$ driving energy for delamination,

$\mathfrak t$ traction stress vector (acting on $\GNeu\cup\GC$),  

$\mathfrak t_{\rm n}$, $\mathfrak t_{\rm t}$ normal or tangential 
component of $\mathfrak t$,  

$f$ bulk load (acting on $\Omega$),

$g$ surface load (acting on $\GNeu$),

$w_\Dir$ surface displacement loading  (on $\GDir$).

\end{minipage}\medskip
}

\vspace*{-.4em}

\begin{center}
{\small\sl Table\,1.\ }
{\small Summary of the basic notation used thorough the paper. 
}
\end{center}

\medskip

\noindent
We consider the standard model of a {\it unilateral frictionless Signorini 
contact}. The quasistatic boundary-value problem for the displacement $u$ on 
$\Omega$ and the so-called delamination parameter $z$ on $\GC$ valued in 
$[0,1]$, representing Fr\'emond's concept \cite{Frem85DAS} of delamination, 
considered in this paper is: 
\begin{subequations}\label{plain}\begin{align}
\label{plain1}
&\mathrm{div}\,\bbC\epsilon+f=0\qquad\text{ with }\qquad 
\epsilon=\epsilon(u,\DT{u})=\chi e(\DT{u})+e(u)
&&\text{on }\Omega,
\\\label{plain2}  
&u=w_\Dir&&\text{on }\GDir,
\\\label{plain3}  
&\mathfrak{t}(\epsilon)=g&&\text{on }\GNeu,
\\\label{class-form-d}
   &\left.\begin{array}{ll}
&\hspace{-1.7em}
\mathfrak{t}_{\rm t}(\epsilon)+
z\big(\bbK u{-}\big((\bbK u){\cdot}\vec{n}\big)\vec{n}\big)=0,
\\[.3em]
   &\hspace{-1.7em}u{\cdot}\vec{n}\ge0,
\quad\
\mathfrak{t}_{\rm n}(\epsilon){+}z(\bbK u){\cdot}\vec{n}\ge0,
\quad\ \big(\mathfrak{t}_{\rm n}(\epsilon){+}z(\bbK u){\cdot}\vec{n}\big)
\big(u{\cdot}\vec{n}\big)=0,
\\[.3em]
   &\hspace{-1.7em}\DT{z}\le0,\qquad\qquad\mathfrak d\le \alpha,
\qquad\qquad\qquad\quad\ \DT{z}(\mathfrak d{-}\alpha)=0,
 \\[.3em] &\hspace{-1.7em}
\mathfrak d\in\frac12\bbK u{\cdot}u+N_{[0,1]}(z)
   \end{array}\ \ \right\}
   &&\text{on }\GC,&&
\end{align}
where we use the usual ``dot-notation'' for the time derivative, i.e.\ 
$\displaystyle{(\cdot)\!\DT{^{}}}$=$\frac\partial{\partial t}$, where
the set-valued mapping $N_{[0,1]}:\R\tto\R$ assigns $z\in\R$ the normal cone 
$N_{[0,1]}(z)$ to the
convex set $[0,1]\subset\R$ at $z\in\R$, and where the traction stress
and its normal and tangential components are defined on $\GC\cup\GNeu$ 
respectively by the formulas 
\begin{align}\label{def-of-T}
\mathfrak{t}(\epsilon)=\big(\bbC\epsilon\big)\big|_{\Gamma}\vec{n},
\qquad\mathfrak{t}_{\rm n}(\epsilon)=\vec{n}\big(\bbC\epsilon\big)\big|_{\Gamma}\vec{n},
\qquad\mathfrak{t}_{\rm t}(\epsilon)=\big(\bbC\epsilon\big)\big|_{\Gamma}\vec{n}
-\mathfrak{t}_{\rm n}(\epsilon),
\end{align}
where $\vec{n}=\vec{n}(x)$ is the unit outward normal to 
$\Gamma:=\partial\Omega$. We further consider the initial-value problem
for (\ref{plain}a-e) by prescribing the initial conditions
\begin{align}
\label{IC}
u(0)=u_0\qquad\text{and}\qquad z(0)=z_0.
\end{align}\end{subequations}
Of course, the loading $f$, $g$, and $w_\Dir$ in (\ref{plain}a-c)
depend on time $t$. The parameter $\alpha>0$ in \eqref{class-form-d} is a given
phenomenological number quantity (possibly as a function of $x\in\GC$)
with a physical dimension J/m$^2$ with the meaning of a specific 
energy needed (and thus deposited in the newly created surface) 
to delaminate 1m$^2$ of the surface under adhesion
or, equally, the energy dissipated by this delamination process;
in fact, \eqref{def-of-T-R} below reflects the latter interpretation.
In engineering, $\alpha$ is also called {\it fracture toughness} 
(or fracture energy).
 
As already mentioned, in conventional materials, the viscosity and 
thus {\it relaxation time} $\chi>0$ is mostly very 
small in comparison with external force loading time-scale, and it is worth 
studying the asymptotics for $\chi\to0$.
Formally, the inviscid limit problem arising for $\chi\to0$ is 
a quasistatic problem for purely elastic material 
which consists in replacing (\ref{plain1}) by
\begin{align}
\label{plain-hyp1}
&\mathrm{div}\,\bbC e(u)+f=0
\qquad\qquad\text{on }\Omega,
\end{align}
and  in replacing $\mathfrak{t}(\epsilon)$ by $\mathfrak{t}(e(u))$
in \eqref{plain3} and similarly
$\mathfrak{t}_{\rm n}(\epsilon)$ and $\mathfrak{t}_{\rm t}(\epsilon)$
by $\mathfrak{t}_{\rm n}(e(u))$ and $\mathfrak{t}_{\rm t}(e(u))$
in \eqref{class-form-d} 
with $\mathfrak{t}(\cdot)$, $\mathfrak{t}_{\rm n}(\cdot)$,
and $\mathfrak{t}_{\rm t}(\cdot)$ again from \eqref{def-of-T}. 
This limit {\it rate-independent problem} itself, however, does not
record any trace of energy dissipated by viscosity in the bulk
during rupture of the delaminating surface, but there are
explicit examples, cf.\ \cite{Roub??ACVE}, showing that 
this energy is not negligible no matter how the viscosity coefficient $\chi>0$ 
is small, which leads to a notion of {\it Kelvin-Voigt approximable solution}
to this limit rate-independent problem involving a certain, so-called
{\it defect measure} recording the ``memory'' of this dissipated energy
which somehow remains even if viscosity coefficient $\chi$ vanishes
(i.e.\ is passed to 0).

The plan of the paper is as follows: First, in Section~\ref{sec-invisc}, we 
formulate the above initial-boundary-value problem \eqref{plain} weakly
and briefly present the main results about a-priori estimates and convergence
for $\chi\to0$ to the inviscid quasistatic rate-independent problem,
leading to the above mentioned approximable solutions and defect measures,
mainly taken from \cite{Roub??ACVE}. In Sect.~\ref{sec-disc}, we perform time 
discretisation by a semi-implicit scheme and present some convergence results 
again from \cite{Roub??ACVE}, and prove that the residuum in the discrete 
energy balance converges to zero if the time step goes to 0. Merging 
Sections~\ref{sec-invisc} and \ref{sec-disc}, this energy-residuum convergence
serves as an important ingredient for controlling convergence of the 
discretisation with dependence on convergence of viscosity. This 
is eventually used in Section~\ref{sec-num} where, making still
a spacial discretisation by boundary-element method (BEM),
we perform computational experiments both with a one-dimensional
example from \cite{Roub??ACVE} with a known solution to tune
parameters of the algorithm and eventually with a nontrivial two-dimensional example.
In this last example, we (to our best knowledge historically for the first time) 
present numerical study of a nontrivial, spatially non-homogeneous defect 
measure.

\section{Inviscid problem as a vanishing-viscosity limit}\label{sec-invisc}

The weak formulation of the initial-boundary value problem 
\eqref{plain} is a bit delicate 
due to the doubly-nonlinear structure of the flow rule for $z$ on $\GC$
without any compactness (i.e.\ without any gradient theory for $z$)
and with both involved nonlinearities unbounded due to the constraints 
$\DT z\le0$ and $z\ge0$ (while the third constraint $z\le1$ is 
essentially redundant if the initial condition satisfies it).
This would make serious difficulties in proving the existence of conventional 
weak solutions.
Benefiting from rate-independency of the evolution rule for $z$, we can
cast a suitable definition by combining the conventional weak solution
concept for $u$ and the so-called energetic-solution concept 
\cite{Miel05ERIS,MieThe04RIHM} of $z$ as in \cite{Roub09RIPV}. 

Considering a fixed time horizon $T>0$, we use the shorthand notation 
$I=(0,T)$, $\bar I=[0,T]$, $Q=I\times\Omega$, $\bar Q=\bar I\times\bar\Omega$ 
with $\bar\Omega$ the closure of $\Omega$,  $\SDir=I\times\GDir$, and $\SC=I\times\GC$. We will 
assume, without substantial restriction of generality of geometry of 
the problem, that  
\begin{align}\label{ass-geom}
\text{dist}(\GC,\GDir)>0,\qquad \text{meas}_{d-1}(\GDir)>0,
\qquad \text{meas}_{d-1}(\GC)>0.
\end{align}

We first make a transformation of the problem to get time constant Dirichlet 
condition. To this goal, we 
first consider a suitable prolongation $u_\Dir$ of $w_\Dir$ defined
on $Q$, i.e.\ $u_\Dir|_{\SDir}=w_\Dir$. Then we shift $u$ to $u+u_\Dir$,
and rewrite \eqref{plain} for such a shifted $u$. 
Thanks to the first condition in \eqref{ass-geom}, we can assume that 
$u_\Dir|_{\SC}=0$ so that \eqref{class-form-d}
remains unchanged under this shift. The equations (\ref{plain}a-c) transform
in such a way that the original loading $f$, $g$, and $w_\Dir$ as well as the 
original initial data $u_0$ are respectively modified as follows:
\begin{subequations}\label{rhs-new}\begin{align}\label{rhs-new-f}
&f\ \ \ \ \ \text{ replaced by }\ \ \ \ \ f+
\mathrm{div}\,\bbC\epsilon_\Dir
\ \ \ \ \ \text{ with }\ \ \ \epsilon_\Dir=e(\chi\DT u_\Dir{+}u_\Dir),
\\\label{rhs-new-g}&
g\ \ \ \ \ \text{ replaced by }\ \ \ \ \ 
g+\big(\bbC\epsilon_\Dir)|_{\GNeu}\vec{n},
\\\label{rhs-new-wD}
&w_\Dir\ \ \ \text{ replaced by }\ \ \ \ \,0,
\\\label{rhs-new-u0}
&u_0\ \ \ \;\text{ replaced by }\ \ \ \ \,u_0-u_\Dir(0).
\end{align}\end{subequations}

We will use the standard notation $W^{1,p}(\Omega)$ for the Sobolev space 
of functions having the gradient in $L^p(\Omega;\R^d)$. If valued in 
$\R^n$ with $n\ge2$, we will write $W^{1,p}(\Omega;\R^n)$, and furthermore we 
use the shorthand notation $H^1(\Omega;\R^n)=W^{1,2}(\Omega;\R^n)$.
We also use the notation of ``$\,\cdot\,$'' and ``$\,:\,$''
for a scalar product of vectors and 2nd-order tensors, respectively. Later, 
$\mathrm{Meas}(\bar Q)\cong
C(\bar Q)^*$ will denote the space of measures on the compact set $\bar Q$.
For a Banach space $X$, $L^p(I;X)$ will denote the Bochner space of 
$X$-valued Bochner measurable functions $u:I\to X$ with its norm 
$\|u(\cdot)\|$ in $L^p(I)$, here $\|\cdot\|$ stands for 
the norm in $X$. Further,
$BV(\bar I;X)$ will denote the space of mappings $u:\bar I\to X$ 
with a bounded variations, i.e.\ $
\sup_{0\le t_0<t_1<...<t_{n-1}<t_n\le T}\sum_{i=1}^n\|u(t_i){-}u(t_{i-1})\|<\infty$ 
where the supremum
is taken over all finite partitions of the interval $[0,T]$. Also, we will use
$H^1(I;X)$ for the Sobolev space of $X$-valued functions with distributional
derivatives in $L^2(I;X)$.
To accommodate the transformation \eqref{rhs-new} into the weak formulation,
we introduce the functional $\mathfrak f(t)\in H^1(\Omega;\R^d)^*$ by
\begin{align}\label{rhs-of-f}
&\big\langle \mathfrak f(t),v\big\rangle
:=\int_\Omega\!f(t){\cdot}v-\bbC e\big(\chi\DT{u}_\Dir(t){+}u_\Dir(t)\big){:}e(v)
+\int_{\GNeu}\!\!g(t){\cdot}v\,\d S.
\end{align}

\begin{definition}\label{def-ES} 
The couple $(u_\chi,z_\chi)$ with $u_\chi\!\in\!H^1(I;H^1(\Omega;\R^d))$ and 
$z_\chi\!\in\!BV(\bar I;L^1(\GC))\cap L^\infty(\SC)$ is called 
an energetic solution to the initial-boundary-value problem \eqref{plain}
under the transformation \eqref{rhs-new} if\\
\ITEM{(i)}{the momentum equilibrium (together with Signoring boundary conditions) in the weak form}\vspace*{-1.5em}
\begin{subequations}\label{weak}\begin{align}
&\hspace*{-2em}\int_Q\bbC e(\chi\DT{u}_\chi{+}u_\chi){:}e(v{-}u_\chi)\,\d x\d t
\label{momentum-equilibrium-visc}
+\int_{\SC}\!\!z_\chi\bbK u_\chi{\cdot}(v{-}u_\chi)\,\d S\d t
\ge\int_0^T\big\langle \mathfrak f(t),v{-}u_\chi\big\rangle\,\d t
\end{align}
\ITEM{}{with $\mathfrak f$ defined 
in \eqref{rhs-of-f} holds for any $v\in H^1(I;H^1(\Omega;\R^d))$ with 
$v|_{\SC}{\cdot}\vec{n}\ge0$ and $v|_{\SDir}=0$,}
\ITEM{(ii)}{the so-called semi-stability of the delamination holds for any 
$t\in[0,T]$:}
\begin{align}\label{semi-stab-visc}
\hspace*{-2em}\bbK u_\chi(t,x){\cdot}u_\chi(t,x)\le2\alpha(x)
\ \ \ \ \text{ or }\ \ \ \ z_\chi(t,x)=0\ \ \ \ \text{ for a.a.\ $x\in\GC$,}
\end{align}
\ITEM{(iii)}{and the energy equality}
\begin{align}
&\hspace*{-2em}\calE(t,u_\chi(t),z_\chi(t))
+\int_0^t\!\int_{\Omega}\!\chi\bbC e(\DT u_\chi){:}e(\DT u_\chi)\,\d x\d t
\label{engr-ineq}
+\int_{\GC}\!\!\alpha\big(z_0{-}z_\chi(t)\big)\,\d S=\calE(0,u_0,z_0)
+\int_0^t\!\big\langle\DT{\mathfrak f},u_\chi\big\rangle\,\d t
\end{align}
\ITEM{}{holds for any $t\in[0,T]$ with $\mathfrak f$ defined again 
by \eqref{rhs-of-f}, and with}
\begin{align}
\label{engr-ineq-E}
&\hspace*{-2em}\calE(t,u,z):=\begin{cases}
\displaystyle{\int_{\Omega}\!\frac12\bbC e(u){:}e(u)\d x
-\big\langle\mathfrak f(t),u\big\rangle
+\int_{\GC}\!\frac12z\bbK u{\cdot}u\d S}
&\text{if }\ u{\cdot}\vec{n}\ge0,\ \ 0\le z\le1\ \text{ on }\GC,
\\[-.5em]&\text{and if }\ u=0\ \text{ on }\GDir,
\\+\infty&\text{else,}\end{cases}\end{align}
\end{subequations}
\ITEM{(iv)}{the initial conditions \eqref{IC} understood transformed as 
{\rm(\ref{rhs-new-u0})} hold.}

This definition is indeed well selective in the sense that any 
smooth energetic solution solves also \eqref{plain} in the classical
sense. Due to \eqref{engr-ineq}, $\calE(t,u_\chi(t),z_\chi(t))<\infty$
so that it holds $u_\chi|_{\SC}{\cdot}\vec{n}\ge0$, $u_\chi|_{\SDir}=0$, and 
$0\le z_\chi\le1$ if the initial conditions satisfies these constraints so that 
$\calE(0,u_0,z_0)<\infty$. Note also that \eqref{plain} has an abstract 
structure of the initial-value problem for the triply nonlinear system of two 
evolution inclusions:
\begin{subequations}\label{abstract}\begin{align}
\big[\calR_\chi\big]_{\DT u}'\DT u+\partial_u\calE(t,u,z)\ni0\,,&&u(0)=u_0,&&&&&&&&
\\
\partial_{\DT z}\calR_\chi(\DT z)+\,\partial_z\calE(t,u,z)\ni0\,,&&z(0)=z_0,&&&&&&&&
\end{align}\end{subequations}
with $[\cdot]'$  denoting the (partial) G\^ateaux differentials and $\partial$ 
denoting the partial subdifferentials in the sense of convex analysis, with
$\calE$ from \eqref{engr-ineq-E}, and with the $\chi$-dependent 
(pseudo)potential of dissipative forces $\calR_\chi$ given by
\begin{align}\label{def-of-T-R}
\calR_\chi(\DT u,\DT z)=\begin{cases}
\displaystyle{\int_{\Omega}\frac\chi2\bbC e(\DT u){:}e(\DT u)
\,\d x+\!
\int_{\GC}\!\!\!\alpha|\DT z|\,\d S}\!\!&\text{if }\DT z\le0\text{ a.e. on }\GC,
\\+\infty&\text{otherwise.}\end{cases}
\end{align}
Also note that \eqref{semi-stab-visc} is equivalent to the integrated
form of the abstract semistability 
$\calE(t,u_\chi(t),z_\chi(t))\le\calE(t,u_\chi(t),\tilde z)+
\calR_0(\tilde z{-}z_\chi(t))$ holding for any $\tilde z\ge0$, 
where we wrote briefly 
$\calR_0(\DT u,\DT z)=\calR_\chi(0,\DT z)=:\calR_0(\DT z)$. This means here:
\begin{align}
\forall \tilde z\!\in\! L^\infty(\GC),\ \ 0\!\le\!\tilde z\!\le\!z_\chi(t):\quad
&\int_{\GC}(z_\chi(t){-}\tilde z)\big(\bbK u_\chi(t){\cdot}u_\chi(t)
-2\alpha\big)\,\d S\le0.
\label{semi-stabil}
\end{align}

We will generally assume the following data qualification:
\begin{subequations}\label{ass}\begin{align}\label{ass0}
&\bbC>0\ \text{(=\,positive definiteness)},
\\&\label{ass1}
\mathfrak f\in W^{1,1}(I;H^1(\Omega;\R^d)^*),
\ \ \ u_0\in H^1(\Omega;\R^d),
\ \ \ z_0\in L^\infty(\GC),
\\&\label{ass2}
 u_0|_{\GC}{\cdot}\vec{n}\ge0,\qquad 
0\le z_0\le1\ \ \text{ a.e.\ on }\GC,\ \ \text{ and }\ \ 
\\&\label{ass3}
\bbK u_0(x){\cdot}u_0(x)\le2\alpha
\ \ \text{ or }\ \ z_0(x)=0\ \ \text{ for a.a.\ }x\in\GC.
\end{align}\end{subequations}
Note that the qualification of $\mathfrak f$ in \eqref{ass1} represents, in 
fact, assumptions on $f$, $g$,
and $w_\Dir$ in the original problem \eqref{plain}, and that (\ref{ass}c,d) 
means semi-stability of the initial condition $(u_0,z_0)$.
Under \eqref{ass}, existence of the solutions due to Definition~\ref{def-ES}
can, in fact, be proved by limiting the discrete solutions 
\eqref{semi-impl}, cf.~Lemma~\ref{lem-disc} and details in 
\cite{Roub09RIPV,Roub??ACVE}.
\end{definition}
\begin{proposition}[Vanishing viscosity limit, \cite{Roub??ACVE}]\label{ch6:prop-slow-load-I}
Let \eqref{ass-geom} and \eqref{ass} hold, and let $\chi>0$.
Then:\\
\ITEM{(i)}{Any solution $(u_\chi,z_\chi)$
according to Definition~\ref{def-ES} satisfies the 
a-priori estimates:}\vspace*{-1.5em}
\begin{subequations}\label{ch6:delam-slow-apriori}
\begin{align}
\label{ch6:delam-slow-apriori-2}
&\big\|\DT u_\chi\big\|_{L^2(I;H^1(\Omega;\R^d))}
\le C/\sqrt\chi,
\\\label{ch6:delam-slow-apriori-3}
&\big\|u_\chi\big\|_{L^\infty(I;H^1(\Omega;\R^d))}\le C,
\\\label{ch6:delam-slow-apriori-4}
&\big\|z_\chi\big\|_{L^\infty(\SC)\,\cap\,BV(\bar I;L^1(\GC))}\le C
\end{align}\end{subequations}
\ITEM{}{with $C$ independent of $\chi$.}
\ITEM{(ii)}{There are $u\!\in\!L^\infty(I;H^1(\Omega;\R^d))$,
$z\!\in\!BV(\bar I;L^1(\GC))$, and $\mu\!\in\!\mathrm{Meas}(\bar Q)$, 
and a subsequence such that, for $\chi\to0$, }\\[-1em]
\begin{subequations}\label{ch6:delam-slow-conv}
\begin{align}\label{ch6:delam-slow-conv-a}
&\!
u_\chi(t)\to u(t)&&
\text{in }H^1(\Omega;\R^d)\ \ \text{for a.a. }  t\!\in\![0,T],\!\!\!\!\!\!\!\!&&&&
\\\label{ch6:delam-slow-conv-b}
&\!
z_\chi(t)\weaks z(t)&&\text{in }L^\infty(\GC)\ \ \ \ \ \;
\text{for all }\ \ t\!\in\![0,T],\!\!\!\!\!\!\!\!&&&&
\\\label{ch6:delam-slow-conv-c}
&\!
\chi\bbC e(\DT u_\chi){:}e(\DT u_\chi)\weaks\mu&&
\text{in }\mathrm{Meas}(\bar Q).&&&&
\end{align}\end{subequations}
\ITEM{(iii)}{Any triple 
$(u,z,\mu)$ obtained by this way fulfills, for a.a. $t\in[0,T]$, the momentum equilibrium 
in the weak form, i.e.}\\[-1em]
\begin{subequations}\label{semi-ES}\begin{align}
&\int_\Omega\bbC e(u(t)){:}e(v{-}u(t))\,\d x
\label{momentum-eq}
+\int_{\GC}\!\!z(t)\bbK u(t){\cdot}(v{-}u(t))\,\d S
\ge\big\langle\mathfrak f(t),v{-}u(t)\big\rangle
\end{align}
\ITEM{}{for all $v\in H^1(\Omega;\R^d)$ with $v|_{\GC}{\cdot}\vec{n}\ge0\,$, 
furthermore the semi-stability }
\begin{align}&\label{semi-stab}
\bbK u(t,x){\cdot}u(t,x)\le2\alpha(x)
\ \ \ \ \text{ or }\ \ \ \ z(t,x)=0\ \ \ \ \text{ for a.a.\ $x\!\in\!\GC$}
\end{align}
\ITEM{}{and eventually the energy equality}
\begin{align}&\calE(t,u(t),z(t))
+\int_{\GC}\!\!\!\alpha\big(z_0{-}z(t)\big)\,\d S
+\int_0^t\!\!\int_{\bar\Omega}\mu(\d x\d t)
=\calE(0,u_0,z_0)
+\int_0^t\!\big\langle \DT{\mathfrak f},u\big\rangle\,\d t.
\label{engr-eq1}
\end{align}\end{subequations}
\end{proposition}

The above assertion suggests the following:

\begin{definition}\label{def-ES1} 
A triple $(u,z,\mu)$ with $u\!\in\!L^\infty(I;H^1(\Omega;\R^d))$,
$z\!\in\!BV(\bar I;L^1(\GC))$, and $\mu\!\in\!\mathrm{Meas}(\bar Q)$,  
$\mu\ge0$, is called a Kelvin-Voigt-approximable solution to the quasistatic 
rate-independent delamination problem \eqref{plain} with $\chi=0$
transformed by \eqref{rhs-new} if \eqref{semi-ES} holds for a.a.\ 
$t\!\in\!I$, and $z(0)=z_0$, and if $(u,z,\mu)$ is attainable by a sequence 
of viscous solutions $\{(u_\chi,z_\chi)\}_{\chi>0}$ in the sense 
\eqref{ch6:delam-slow-conv}.
\end{definition}

The measure $\mu\in\mathrm{Meas}(\bar Q)$, invented in \cite{Roub??ACVE}, 
occurring in Proposition~\ref{ch6:prop-slow-load-I} represents a certain
additional energy distributed over $\bar Q$ specified rather implicitly 
by \eqref{ch6:delam-slow-conv-c} but anyhow with a certain physical 
justification. Similar concept has been invented in various other problems 
in continuum 
mechanics (particularly of fluids) under the name of {\it defect measures}
to reflect a possible additional energy dissipation of solutions lacking 
regularity and exhibiting various concentration effects in contrast to 
regular weak solutions where the defect measure vanishes, cf.\ 
\cite{Gera91MDM,Feir03DCF,Naum06ETWS}.
Here, $\mu$ reflects the possible additional dissipated energy of 
Kelvin-Voigt-approximable solutions comparing to the so-called energetic 
solutions, cf.\ also Remark~\ref{rem-stress-vs-energy} below.
A nonvanishing $\mu$ is vitally important and rather desirable in 
the context of fracture mechanics in contrast to 
the mentioned fluid-mechanical applications where the phenomenon of 
nonvanishing $\mu$ is related ``only'' to a possible lack of regularity
of weak solutions and is not entirely clear whether it has some physical
justification and supported experimental evidence.

\section{Time discretisation, convergence}\label{sec-disc}

Some solutions to the initial-boundary value problem \eqref{plain} 
in accord to Definition~\ref{def-ES} can be obtained rather 
constructively by a {\it semi-implicit time discretisation}.
To facilitate the a-priori estimates, we again consider the transformation
\eqref{rhs-new}. Using an equidistant partition of the time interval $[0,T]$
with a time step $\tau>0$ such that $T/\tau\in\N$, we consider:
\begin{subequations}\label{semi-impl}\begin{align}
\label{semi-impl-1}
&\mathrm{div}\,\bbC\epsilon_\tau^k+f_\tau^k=0
\qquad\text{ with }\ 
\epsilon_\tau^k=\chi e\Big(\frac{u_\tau^k{-}u_\tau^{k-1}}\tau\Big)+e(u_\tau^k)
&&\text{on }\Omega,
\\\label{semi-impl-2}  
&u_\tau^k=0&&\text{on }\GDir,
\\\label{semi-impl-3}  
&\mathfrak{t}(\epsilon_\tau^k)=g_\tau^k&&\text{on }\GNeu,
\\\label{semi-impl-4}
   &\left.\begin{array}{ll}
&\hspace{-1.7em}
\mathfrak{t}_{\rm t}(\epsilon_\tau^k)+
z_\tau^{k-1}\big(\bbK u_\tau^k{-}\big((\bbK u_\tau^k){\cdot}\vec{n}\big)\vec{n}\big)=0,
\\[.3em]
&\hspace{-1.7em}
u_\tau^k{\cdot}\vec{n}\ge0,\ \
\mathfrak{t}_{\rm n}(\epsilon_\tau^k)
{+}z_\tau^{k-1}(\bbK u_\tau^k){\cdot}\vec{n}\ge0,
\ \ \big(\mathfrak{t}_{\rm n}(\epsilon_\tau^k)
{+}z_\tau^{k-1}(\bbK u_\tau^k){\cdot}\vec{n}\big)
\big(u_\tau^k{\cdot}\vec{n}\big)=0,\hspace{-1em}
\\[.3em]
&\hspace{-1.7em}z_\tau^k\le z_\tau^{k-1},\qquad\qquad\mathfrak{d}_\tau^k\le\alpha,
\qquad\qquad\qquad(z_\tau^k-z_\tau^{k-1})(\mathfrak{d}_\tau^k-\alpha)=0,
 \\[.3em] &\hspace{-1.7em}
\mathfrak{d}_\tau^k\in \frac12\bbK u_\tau^k{\cdot}u_\tau^k+N_{[0,1]}(z_\tau^k)
   \end{array}\ \ \right\}\!\!\!\!
   &&\text{on }\GC,
\end{align}\end{subequations}
with $\mathfrak{t}(\cdot)$, $\mathfrak{t}_{\rm n}(\cdot)$, and 
$\mathfrak{t}_{\rm t}(\cdot)$ from \eqref{def-of-T} and with 
$f_\tau^k=f(k\tau)$ and $g_\tau^k=g(k\tau)$ with $f$ and $g$ from 
\eqref{rhs-new}, and proceeding recursively for $k=1,...,T/\tau$ with 
starting for $k=1$ from 
\begin{align}
u_\tau^0=u_0\qquad\text{and}\qquad z_\tau^0=z_0.
\end{align}
The adjective ``semi-implicit'' is related with usage of $z_\tau^{k-1}$ 
in the first complementarity problem in \eqref{semi-impl-4},
instead of $z_\tau^k$ which would lead to a fully implicit formula.
Such usage of $z_\tau^{k-1}$ leads to the decoupling of the problem:
first we can solve (\ref{semi-impl}a-c) with the first complementarity 
problem in \eqref{semi-impl-4} for $u_\tau^k$ and only after 
the rest of \eqref{semi-impl-4} for $z_\tau^k$; in fact, this
can be understood as a fractional-step method, cf.\ also 
\cite[Remark~8.25]{Roub12NPDE}. In addition, we can 
employ the variational structure of both decoupled problems. 
We thus obtain two convex minimization problems: first, we are to solve
\begin{subequations}\label{semi-impl+}\begin{align}\label{semi-impl+1}
&\left.\begin{array}{ll}
\text{minimize}&\displaystyle{\calE(k\tau,u,z_\tau^{k-1})
+\tau\calR_\chi\Big(\frac{u{-}u_\tau^{k-1}}{\tau},0\Big)}
\\[.3em]
\text{subject to}&u\in H^1(\Omega;\R^d),\ \,u|_{\GDir}=0,\ \,
u|_{\GC}{\cdot}\vec{n}\ge0
\end{array}\right\}
\intertext{and, denoting its unique solution by $u_\tau^k$, then we solve}
\label{semi-impl+2}
&\left.\begin{array}{ll}
\text{minimize}&\displaystyle{\calE\big(k\tau,u_\tau^k,z\big)
+\calR_\chi\big(0,z{-}z_\tau^{k-1}\big)}
\\[.3em]
\text{subject to}&z\in L^\infty(\GC),\ \ \ 0\le z\le z_\tau^{k-1}\end{array}
\hspace{3.3em}\right\}
\end{align}\end{subequations}
with the stored energy $\calE$ and the dissipation (pseudo)potential 
$\calR_\chi$ defined here by
\begin{subequations}\label{E-R}\begin{align}\label{M}
&\calE(t,u,z)=\int_{\Omega}\!\frac12\bbC e(u){:}e(u)\d x
+\int_{\GC}\!\frac12z\bbK u{\cdot}u\d S
-\big\langle\mathfrak f(t),u\big\rangle,
\\&\calR_\chi(\DT u,\DT z)=\int_\Omega\frac\chi2\bbC e(\DT u){:} e(\DT u)\,\d x
-\int_{\GC}\!\!\alpha\DT z\,\d S.
\end{align}\end{subequations}
Note that the constraints $u|_{\GC}{\cdot}\vec{n}\ge0$, $0\le z\le 1$, 
and $\DT z\le0$, originally contained in $\calE$ and $\calR_\chi$ in
\eqref{engr-ineq-E} and \eqref{def-of-T-R}, are now included in 
\eqref{semi-impl+} so that
we can equivalently use the smooth functionals $\calE(t,\cdot,\cdot)$
and $\calR_\chi$ in \eqref{E-R}. Also 
note that $\calR_\chi(\DT u,\cdot)$ is  degree-1 homogeneous
so that the factor $\tau$ does not show up in the functional 
in \eqref{semi-impl+2}, in contrast to the degree-2 homogeneous functional
$\calR_\chi(\cdot,\DT z)$ in \eqref{semi-impl+1}.

The discrete analog of \eqref{momentum-equilibrium-visc} is 
by summation (for $k=1,...,T/\tau$) of the optimality conditions for 
\eqref{semi-impl+1} written at $u=u_\tau^k$, i.e.
\begin{align}
\int_\Omega\bbC\epsilon_\tau^k{:}e(v{-}u_\tau^k)\,\d x
+\int_{\GC}\!\!z_\tau^{k-1}\bbK u_\tau^k{\cdot}(v{-}u_\tau^k)\,\d S
\ge\big\langle \mathfrak f_\tau^k,v{-}u_\tau^k\big\rangle
\label{momentum-disc}
\end{align}
with $\epsilon_\tau^k$ from \eqref{semi-impl-1} and 
$\mathfrak f_\tau^k=\mathfrak f(k\tau)$, and
tested by an arbitrary test-function $v=v_\tau^k$.
By comparison of values of \eqref{semi-impl+2} at $z_\tau^k$ 
and an arbitrary $\tilde z$, we get 
$\calE(k\tau,u_\tau^k,z_\tau^k)+\calR_0(z_\tau^k{-}z_\tau^{k-1})\le
\calE(k\tau,u_\tau^k,\tilde z)+\calR_0(\tilde z{-}z_\tau^{k-1})$.
By the degree-1 homogeneity and the convexity of $\calR_0(\cdot)$,
we further get the triangle inequality 
$\calR_0(\tilde z{-}z_\tau^{k-1})\le \calR_0(z_\tau^k{-}z_\tau^{k-1})
+\calR_0(\tilde z{-}z_\tau^k)$. Re-organizing the first 
estimate and merging it with the second one, we obtain
the discrete analog of the semistability \eqref{semi-stab-visc}, namely:
\begin{align}
\calE(k\tau,u_\tau^k,z_\tau^k)&\le \calE(k\tau,u_\tau^k,\tilde z)
+\calR_0(\tilde z{-}z_\tau^{k-1})-\calR_0(z_\tau^k{-}z_\tau^{k-1})
\label{disc-stab}
\le\calE(k\tau,u_\tau^k,\tilde z)+\calR_0(\tilde z{-}z_\tau^k),
\end{align}
A discrete analog of \eqref{engr-ineq} as an inequality ``$\le$''
can be obtained by testing the optimality conditions for \eqref{semi-impl+1}
and \eqref{semi-impl+2} respectively 
by $u_\tau^k{-}u_\tau^{k-1}$ and $z_\tau^k{-}z_\tau^{k-1}$ (which, in fact,
means plugging $v=u_\tau^{k-1}$ into \eqref{momentum-disc} for 
the former test), and by adding it, benefiting from the cancellation of the 
terms $\pm\calE(k\tau,u_\tau^k,z_\tau^{k-1})$ and by the separate
convexity of $\calE(t,\cdot,\cdot)$. This gives the estimate
\begin{align}\label{engr-ineq-disc}
\calE(k\tau,u_\tau^k,z_\tau^k)
+\tau\sum_{l=1}^k\calR_\chi\Big(\frac{u_\tau^l{-}u_\tau^{l-1}}\tau,
\frac{z_\tau^l{-}z_\tau^{l-1}}\tau\Big)
\le\calE(0,u_0,z_0)+\tau\sum_{l=1}^k
\Big\langle\frac{f_\tau^l{-}f_\tau^{l-1}}\tau,u_\tau^{l-1}\Big\rangle.
\end{align}

Let us by $u_{\chi,\tau}$ denote the continuous piecewise affine interpolant of 
the values $(u_\tau^k)_{k=0}^{T/\tau}$, and by $\bar u_{\chi,\tau}$ the piecewise 
constant ``backward''  interpolant, while $\underline u_{\chi,\tau}$  the 
piecewise constant ``forward'' interpolant. Analogously, we introduce
$z_{\chi,\tau}$ and $\underline z_{\chi,\tau}$ 
interpolating values $(z_\tau^k)_{k=0}^{T/\tau}$, and also
$\bar{\mathfrak{f}}_\tau$ and ${\mathfrak{f}}_\tau$ interpolating values 
$(\mathfrak{f}_\tau^k)_{k=0}^{T/\tau}$. In terms of these interpolants, we can 
write \eqref{momentum-disc}, \eqref{disc-stab}, and \eqref{engr-ineq-disc} 
more ``compactly'' as
\begin{subequations}\label{semi-engr-disc}
\begin{align}\label{semi-engr-disc-1}
&\int_Q\bbC e(\chi\DT{u}_{\chi,\tau}{+}\bar u_{\chi,\tau})
{:}e(v{-}\bar u_{\chi,\tau})\,\d x\d t
+\int_{\SC}\!\!\underline z_{\chi,\tau}
\bbK\bar u_{\chi,\tau}{\cdot}(v{-}\bar u_{\chi,\tau})\,\d S\d t
\ge\int_0^T\!\big\langle \bar{\mathfrak{f}}_\tau,v{-}\bar u_{\chi,\tau}\big\rangle
\d t
\intertext{for any $v\in L^2(I;H^1(\Omega;\R^d))$ with
$v|_{\SC}{\cdot}\vec{n}\ge0$, and}
\label{semi-engr-disc-2}&\calE(t,u_{\chi,\tau}(t),z_{\chi,\tau}(t))\le
\calE(t,u_{\chi,\tau}(t),\tilde z)+\calR_0(\tilde z{-}z_{\chi,\tau}(t)),
\intertext{for any $\tilde z\in L^\infty(\GC)$ with 
$0\le\tilde z\le z_{\chi,\tau}(t)$ on $\GC$, and}
&\nonumber
\int_0^t\!\!\bigg(\int_\Omega\chi\bbC e(\DT u_{\chi,\tau}){:} e(\DT u_{\chi,\tau})\,\d x
-\langle\DT{\mathfrak{f}}_{\tau},\underline u_{\chi,\tau}\rangle
\bigg)\,\d t+\int_{\GC}\!\!\alpha\big( z_0{-} z_{\chi,\tau}(t)\big)\,\d S
\\&\label{semi-engr-disc-3}
\hspace{12.7em}
+\calE(t,u_{\chi,\tau}(t),z_{\chi,\tau}(t))-\calE(0,u_0,z_0)
=:\mathfrak E_{\chi,\tau}(t)\le0,
\end{align}\end{subequations}
for any $t=k\tau$, $k=1,...,T/\tau$.

Existence of $(u_\tau^k,z_\tau^k)$ solving \eqref{semi-impl} is
simply by a direct method applied to the underlined variational
problems \eqref{semi-impl+}. Fixing $\chi>0$, we can investigate the 
convergence for $\tau\to0$. By a-priori estimates we have at disposal 
from \eqref{semi-engr-disc-3}, using Banach's and Helly's selection 
principles, we have immediately:

\begin{lemma}\label{lem-disc} Assuming \eqref{ass}, $(u_{\chi,\tau},z_{\chi,\tau})$ 
constructed recursively by \eqref{semi-impl} exists and, for $\chi>0$ fixed,
there is a subsequence (indexed by $\tau$'s converging to 0) 
and $u_\chi\in H^1(I;H^1(\Omega;\R^d))$ and 
$z_\chi\in L^\infty(\SC)\cap BV(\bar I;\mathrm{Meas}(\GC))$ such that
\begin{subequations}\label{delam-conv}
\begin{align}\label{delam-conv-a}
&u_{\chi,\tau}\weak u_\chi&&
\text{in }H^1(I;H^1(\Omega;\R^d)),&&&&
\\\label{delam-conv-b}
&z_{\chi,\tau}\weaks z_\chi&&
\text{in }L^\infty(\SC)\cap BV(\bar I;\mathrm{Meas}(\GC)),\ \text{ and}&&&&&&
\\\label{delam-conv-c}
&z_{\chi,\tau}(t)\weaks z_\chi(t)&&
\text{in }L^\infty(\GC)\ \ \ \text{ for any }t\in[0,T].
\end{align}
\end{subequations}
Any $(u_\chi,z_\chi)$ obtained by this way is an energetic solution
to \eqref{plain} due to Definition~\ref{def-ES}.
\end{lemma}

In fact, the last claim required a limit passage in \eqref{semi-engr-disc}
and then, a-posteriori, the proof of energy equality, which is rather 
technical and for details we refer to \cite{Roub09RIPV,Roub??ACVE}.

For further numerical study, cf.\ also Figures~\ref{fig:1-D-residuum} and 
\ref{fig:1-D-residuum+} below, the important feature is that the residuum 
$\mathfrak E_{\chi,\tau}\in L^\infty(I)$ in the discrete energy (im)balance 
\eqref{semi-engr-disc-3} can be controlled by making the time step $\tau>0$ 
sufficiently small:

\begin{proposition}\label{prop-Err}
Assuming again \eqref{ass} and $\chi>0$ fixed, it holds (even without 
any need of selection subsequences as in Lemma~\ref{lem-disc}) that  
\begin{align}\label{Err}
\lim_{\tau\to0}\|\mathfrak E_{\chi,\tau}\|_{L^p(I)}=0\qquad
\text{ for any }1\le p<\infty.
\end{align}
Moreover, for a selected subsequence satisfying \eqref{delam-conv},
the weak convergence \eqref{delam-conv-a} is, in fact, strong, 
i.e.\ in particular
\begin{align}\label{strong-e}
e(\DT u_{\chi,\tau})\to e(\DT u_\chi)\qquad\text{ in }\ L^2(Q;\R^{d\times d}).
\end{align}
\end{proposition}

\begin{proof}
Essentially, for \eqref{Err}, the only important point is to prove 
\eqref{strong-e}. Using 
\eqref{semi-engr-disc-3} for $t=T$, this can be seen from 
\begin{align}
&\nonumber
\hspace{-2.5em}\int_Q\chi\bbC e(\DT u_\chi){:} e(\DT u_\chi)\,\d x\d t
\le\liminf_{\tau\to0}
\int_Q\chi\bbC e(\DT u_{\chi,\tau}){:} e(\DT u_{\chi,\tau})\,\d x\d t
\le\limsup_{\tau\to0}
\int_Q\chi\bbC e(\DT u_{\chi,\tau}){:} e(\DT u_{\chi,\tau})\,\d x\d t
\\&\nonumber
\ \le\calE(0,u_0,z_0)+\limsup_{\tau\to0}
\bigg(\int_0^T\!\!\langle\DT{\mathfrak{f}}_\tau,\underline u_{\chi,\tau}\rangle\,\d t
-\int_{\GC}\!\!\!\alpha\big( z_0{-} z_{\chi,\tau}(T)\big)\,\d S-\calE(T,u_{\chi,\tau}(T),z_{\chi,\tau}(T))\bigg)
\\&
\nonumber
\ \le\calE(0,u_0,z_0)
+\int_0^T\langle\DT{\mathfrak{f}},u_{\chi}\rangle\,\d t-\int_{\GC}\!\!\alpha\big( z_0{-} z_{\chi}(T)\big)\,\d S
-\calE(T,u_{\chi}(T),z_{\chi}(T))
\\
\label{strong-conv-dissip}
&\ 
=\int_Q\chi\bbC e(\DT u_\chi){:} e(\DT u_\chi)\,\d x\d t.
\end{align}
Note that, by \eqref{delam-conv-b}, we have at disposal 
$u_{\chi,\tau}(T)\weak u_{\chi}(T)$ in $H^1(\Omega;\R^d)$ if
$\chi>0$ is fixed, which we used together with \eqref{delam-conv-c} for $t=T$
to estimate $\limsup_{\tau\to0}-\calE(T,u_{\chi,\tau}(T),z_{\chi,\tau}(T))
\le-\calE(T,u_{\chi}(T),z_{\chi}(T))$ in \eqref{strong-conv-dissip}.
Eventually, the last equality in \eqref{strong-conv-dissip}
can be proved by limiting a regularization of the Signorini condition,
cf.\ \cite[Step 4 in Sects.\,8-9]{RosRou11TARI}. As a result,
\eqref{strong-conv-dissip} proves 
\begin{align}\label{delam-conv-a++}
\lim_{\tau\to0}
\int_Q\chi\bbC e(\DT u_{\chi,\tau}){:} e(\DT u_{\chi,\tau})\,\d x\d t
=\int_Q\chi\bbC e(\DT u_\chi){:} e(\DT u_\chi)\,\d x\d t.
\end{align}
Using uniform convexity of the space $L^2(Q;\R^{d\times d})$ 
equipped with the norm $\|e\|:=(\int_Q\bbC e{:}e\,\d x\d t)^{1/2}$,
\eqref{delam-conv-a} with \eqref{delam-conv-a++} allows
for improvement of \eqref{delam-conv-a} to the strong convergence
\begin{align}\label{delam-conv-a+}
u_{\chi,\tau}\to u_\chi&&
\text{in }H^1(I;H^1(\Omega;\R^d)),&&&&
\end{align}
hence also \eqref{strong-e} is proved.

Therefore, $\int_0^t\!\int_\Omega\chi\bbC e(\DT u_{\chi,\tau}){:}e(\DT u_{\chi,\tau})\,\d x\d t$
in \eqref{semi-engr-disc-3} converges to 
$\int_0^t\!\int_\Omega\chi\bbC e(\DT u_{\chi}){:} e(\DT u_{\chi})\,\d x\d t$
even uniformly in $t$. From \eqref{delam-conv-a+},
we have certainly $u_{\chi,\tau}(t)\to u_\chi(t)$ in $H^1(\Omega;\R^d)$
for any $t$, and using also \eqref{delam-conv-c},
we have $\lim_{\tau\to0}\calE(t,u_{\chi,\tau}(t),z_{\chi,\tau}(t))=
\calE(t,u_{\chi}(t),z_{\chi}(t))$ for any $t\in[0,T]$.
Thus we have the convergence in  \eqref{semi-engr-disc-3} 
for any $t\in[0,T]$. As the sequence 
$\{\mathfrak E_{\chi,\tau}\}_{\tau>0}$ does not alternate sign
and is bounded in $L^\infty(I)$, by Lebesgue theorem 
it converges to some $\mathfrak E_{\chi}\in L^\infty(I)$ 
strongly in $L^p(I)$ for any $1\le p<\infty$.
Thus we showed that, in the limit for $\tau\to0$,
\begin{align}
&\hspace*{-4em}\calE(t,u_{\chi}(t),z_{\chi}(t))
+\int_0^t\!\!\int_\Omega\chi\bbC e(\DT u_{\chi}){:} e(\DT u_{\chi})\,\d x
-\langle\DT{\mathfrak{f}},u_{\chi}\rangle\,\d t
\label{semi-engr-disc-3+}
-\calE(0,u_0,z_0)+\int_{\GC}\!\!\alpha\big(z_0{-} z_{\chi}(t)\big)\,\d S
=\mathfrak E_{\chi}(t).
\end{align}
We used already in \eqref{strong-conv-dissip} that the 
left-hand side of \eqref{semi-engr-disc-3+} is zero, i.e.\ here also
$\mathfrak E_{\chi}=0$, so that \eqref{Err} is proved.

Eventually, we can realize that, in contrast to \eqref{delam-conv}
and \eqref{strong-e}, the convergence \eqref{Err} holds even 
for the whole sequence (indexed by a-priori chosen countable
number of $\tau$'s), which can be seen by a standard (contradiction)
arguments based on uniqueness of the limit (here just 0). 
\end{proof}

It should be remarked that, however, we did not prove validity of \eqref{Err}
for $p=\infty$. Anyhow, even a coarser mode of convergence \eqref{Err}
can ensure that the inequality \eqref{semi-engr-disc-3} yields 
eventually the energy equality \eqref{engr-eq1}, as we will pursue
in what follows.

Having the time-discrete viscous scheme, one may think about 
a convergence for both $\chi\to0$ and $\tau\to0$ simultaneously 
to obtain the Kelvin-Voigt-approximable solution to the 
quasistatic rate-independent problem. Here a certain circumspection
has to be taken: obviously, 
$\lim_{\chi\to0}\chi\bbC e(\DT u_{\chi,\tau}){:}e(\DT u_{\chi,\tau})=0$ 
for any $\tau>0$ fixed, cf.\ also Remark~\ref{rem-inviscid} below;
in fact, this convergence is even strong in $W^{1,\infty}(I;L^1(\Omega))$.
Therefore clearly,
\begin{equation*}
\lim_{\tau\to0}\lim_{\chi\to0}\chi\bbC e(\DT u_{\chi,\tau}){:}e(\DT u_{\chi,\tau})=0
\end{equation*}
and the energy balance \eqref{engr-eq1} would be obtained with $\mu=0$
as an inequality only; cf. also Figure~\ref{fig:1-D-residuum+} below.
It is thus obvious that only some conditional convergence will lead to the 
desired $\mu$ and the energy equality  \eqref{engr-eq1} as before. 
Obviously, we 
must consider rather $\lim_{\chi\to0}\lim_{\tau\to0}$.
Linking \eqref{delam-conv-a+} with \eqref{ch6:delam-slow-conv-c},
we have 
\begin{align}\label{double-1}
\text{w*-}\lim_{\!\!\!\!\!\!\!\!\chi\to0}\lim_{\tau\to0}
\chi\bbC e(\DT u_{\chi,\tau}){:}e(\DT u_{\chi,\tau})
=\text{w*-}\lim_{\!\!\!\!\!\!\!\!\chi\to0}\chi\bbC e(\DT u_{\chi}){:}e(\DT u_{\chi})=\mu,
\end{align}
meant in $\text{Meas}(\bar Q)$, and by \eqref{Err} we have obviously also
\begin{align}\label{double-2}
\lim_{\chi\to0}\lim_{\tau\to0}\mathfrak E_{\chi,\tau}=0
\end{align}
meant in $L^p(I)$, $1\le p<\infty$. These two double limits 
can be merged under an implicit stability criterion 
$\mathscr{T}:\R^+\to\R^+$ so that both
\begin{align}\label{joint-conv}
\DDD{\text{w*-lim}}{\tau\le\mathscr{T}(\chi)}{\tau\to0\,,\ \chi\to0}
\chi\bbC e(\DT u_{\chi,\tau}){:}e(\DT u_{\chi,\tau})
=\mu\ \qquad\text{ and}\qquad
\DDD{\text{lim}}{\tau\le\mathscr{T}(\chi)}{\tau\to0\,,\ \chi\to0}\!\!\!
\mathfrak E_{\chi,\tau}=0;
\end{align}
cf.\ the arguments in the proof of \cite[Cor.\ 4.8(ii)]{BarRou11TVER}.
In this way, we obtain a Kelvin-Voigt approximable solution
according Definition~\ref{def-ES1}.
However, the stability criterion $\tau\le\mathscr{T}(\chi)$ is not explicit
and thus not of a direct usage in general.

\bigskip

\section{Numerical implementation and computational experiments}\label{sec-num}

A general observation is that \eqref{semi-impl+} represents
two recursive alternating linear-quadratic minimization problems which,
after another spatial discretisation leads to linear-quadratic programming.
On top of it, as no gradient of $z$ is involved in $\calE$, \eqref{semi-impl+2}
has a local character and allows, after a suitable discretisation of $\GC$,
decoupling on particular boundary elements. Therefore, conceptually the 
proposed scheme leads to a very efficient numerical strategy for fixed 
$\chi>0$ and $\tau>0$. 

The essential difficulty 
is realization of the convergence \eqref{joint-conv}. As the defect 
measure $\mu$ is typically not known (and, on top of it, is not unique),
we can hardly control the former convergence in \eqref{joint-conv}.
Yet, we can at least control the latter one. To this goal, we devise
the following conceptual algorithm relying on the convergence \eqref{Err}
considered with $p=1$:

\begin{center}\fbox{
\begin{minipage}{.52\textwidth}
\smallskip
(1) Set $\chi=\chi_0>0$ and $\tau=\tau_0>0$, and choose $\gamma>0$ fixed.\\
(2) Compute $(u_{\chi,\tau},z_{\chi,\tau})$ and 
$\|\mathfrak E_{\chi,\tau}\|_{L^1(I)}$.\\
(3) If not $\|\mathfrak E_{\chi,\tau}\|_{L^1(I)}\le C\chi^\gamma$,
then put $\tau:=\tau/2$ and go to (2).\\
(4) Put $\chi:=\chi/2$ and $\tau:=\tau/2$.\\
(5) If not $\chi\le\chi_{\rm final}$,  go to (2).\\
(6) The end.
\end{minipage}
}
\\[.3em]
\begin{minipage}[t]{.54\textwidth}\baselineskip=8pt
 {\small\sl Table\,2.\ } {\sl Conceptual strategy to converge with the 
viscosity $\chi$ and the time step $\tau$ 
to approximate the correct energy balance.}
\end{minipage}
\end{center}

\noindent
In this way, we have at least the energetics in the 
limit under control if $\chi_{\rm final}$ would be pushed to zero.

\begin{proposition}\label{prop-algorithm}
The procedure from Table~2 is an algorithm in the sense that, for 
the parameters $\chi_0>\chi_{\rm final}>0$, $\tau_0>0$, $C$, and $\gamma>0$
given, it ends after a finite number of loops, giving a solution
with a viscosity  $\chi$ smaller than the a-priori chosen $\chi_{\rm final}$. 
\end{proposition}

\begin{proof}
The only notable point is that, after finite number of 
refinement of the time discretisation, the condition
$\|\mathfrak E_{\chi,\tau}\|_{L^1(I)}\le C\chi^\gamma$ in Step (3) can
be fulfilled. This follows from \eqref{Err} and the fact, proved in
Proposition~\ref{prop-Err}, that this holds for the whole sequence
of the time partitions.
\end{proof}

\subsection{Spatial discretisation by BEM}

To launch computational experiments, one naturally needs to perform still a 
spatial discretisation. As $\calE(t,\cdot,z)$ and $\calR_\chi(\cdot,\DT z)$ 
are quadratic functionals, \eqref{semi-impl+1} is a quadratic problem with 
the only constraint on $\GC$ which are linear. This allocation of all nonlinear
effects exclusively on the boundary $\GC$ allows for using efficiently 
the boundary-element method (BEM) combined with linear-quadratic programming
treating the variables on $\GC$. 

BEM standardly uses so-called Poincar\'e-Steklov operators which are known in 
specific static cases, here in particular for the homogeneous isotropic 
elastic material which we consider in what follows. Yet, we have to calculate 
the visco-elastic modification and here we benefit from choosing 
the ansatz of the tensor of viscous moduli as simply proportional 
to the elastic moduli, i.e.\ $\chi\bbC$. Therefore we can use BEM with 
the same Poincar\'e-Steklov operators as in the static case 
only for a new variable $v_\tau^k:=u_\tau^k+\chi(u_\tau^k{-}u_\tau^{k-1})/\tau$;
then, in terms of this new variable, one obviously has the Kelvin-Voigt strain 
$\epsilon_\tau^k=e(v_\tau^k)$, the velocity $(u_\tau^k{-}u_\tau^{k-1})/\tau=
(v_\tau^k{-}u_\tau^{k-1})/(\tau{+}\chi)$, and the displacement
$u_\tau^k=(\tau v_\tau^k{+}\chi u_\tau^{k-1})/(\tau{+}\chi)$,
which is to be used in \eqref{semi-impl}, leading to the problem 
\begin{subequations}\label{semi-impl++}\begin{align}
\label{semi-impl-1++}
&\mathrm{div}\,\bbC e(v_\tau^k)+\mathfrak{f}_\tau^k=0
&&\text{on }\Omega,
\\\label{semi-impl-2++}  
&v_\tau^k=0&&\text{on }\GDir,
\\\label{semi-impl-3++}  
&\mathfrak{t}(e(v_\tau^k))=g_\tau^k&&\text{on }\GNeu,
\\\label{semi-impl-4++}
   &\left.\begin{array}{ll}
&\hspace{-1.7em}
\displaystyle{\mathfrak{t}_{\rm t}(e(v_\tau^k))+
z_\tau^{k-1}\Big(\bbK\frac{\tau v_\tau^k{+}\chi u_\tau^{k-1}}{\tau{+}\chi}
{-}\Big(\bbK 
\frac{\tau v_\tau^k{+}\chi u_\tau^{k-1}}{\tau{+}\chi}{\cdot}\vec{n}\Big)
\vec{n}\Big)=0,}
\\[.3em]
&\hspace{-1.7em}
\displaystyle{v_\tau^k{\cdot}\vec{n}\ge -\frac\chi\tau u_\tau^{k-1}
{\cdot}\vec{n},\ \ \ \ \,
\mathfrak{t}_{\rm n}(e(v_\tau^k))
{+}z_\tau^{k-1}\Big(\bbK 
\frac{\tau v_\tau^k{+}\chi u_\tau^{k-1}}{\tau{+}\chi}\Big){\cdot}\vec{n}\ge0,}
\\&\hspace{-1.7em}
\displaystyle{\Big(\mathfrak{t}_{\rm n}(e(v_\tau^k))
{+}z_\tau^{k-1}\Big(\bbK 
\frac{\tau v_\tau^k{+}\chi u_\tau^{k-1}}{\tau{+}\chi}
\Big){\cdot}\vec{n}\Big)
\big((\tau v_\tau^k{+}\chi u_\tau^{k-1}){\cdot}\vec{n}\big)=0,}\hspace{-1em}
\\[.3em]
&\hspace{-1.7em}z_\tau^k\le z_\tau^{k-1},\qquad\quad\ \
\mathfrak{d}_\tau^k\le\alpha,
\qquad\quad\ \ (z_\tau^k-z_\tau^{k-1})(\mathfrak{d}_\tau^k-\alpha)=0,
 \\[.3em] &\hspace{-1.7em}
\displaystyle{\mathfrak{d}_\tau^k\in \frac1{2(\tau{+}\chi)^2}\bbK 
\big(\tau v_\tau^k{+}\chi u_\tau^{k-1}\big){\cdot}
\big(\tau v_\tau^k{+}\chi u_\tau^{k-1}\big)
+N_{[0,1]}(z_\tau^k)}
   \end{array}\ \ \right\}\!\!\!\!
   &&\text{on }\GC,&&
\end{align}\end{subequations}
with $u_\tau^{k-1}=(\tau v_\tau^{k-1}{+}\chi u_\tau^{k-2})/(\tau{+}\chi)$
proceeding recursively for $k=1,... T/\tau\in\N$. Then, like 
\eqref{semi-impl+}, one constructs the corresponding 
minimization problems in terms of $(v_\tau^k,z_\tau^k)$. Also,
evaluation of the energy balance \eqref{semi-engr-disc-3} in terms of $v$ 
is possible at least approximately. More specifically, 
by using the Poincar\'e-Steklov operator for the auxiliary variable $v_\tau^k$
which gives the equilibrium stress (in contrast to $u_\tau^k$),
we calculate the test of the traction stress $\mathfrak{t}(e(v_\tau^k))$
by velocity, i.e.\ the boundary integral 
\begin{align}\nonumber
\int_\Gamma\mathfrak{t}(e(v_\tau^k)){\cdot}
\Big(\frac{u_\tau^k-u_\tau^{k-1}}\tau\Big)\,\d S&=
\int_\Omega\bbC e(v_\tau^k){:}e\Big(\frac{u_\tau^k-u_\tau^{k-1}}\tau\Big)\,\d x
-\Big\langle\mathfrak{f}_\tau^k,\frac{u_\tau^k-u_\tau^{k-1}}\tau\Big\rangle
\\\nonumber
&=\int_\Omega\Big(\chi{+}\frac\tau2\Big)
\bbC e\Big(\frac{u_\tau^k{-}u_\tau^{k-1}}\tau\Big)
{:}e\Big(\frac{u_\tau^k{-}u_\tau^{k-1}}\tau\Big)
+\frac1{2\tau}\bbC e(u_\tau^k){:}e(u_\tau^k)
\\\label{approx-energetics}
&\ \ -\frac1{2\tau}\bbC e(u_\tau^{k-1}){:}e(u_\tau^{k-1})
-\Big\langle\mathfrak{f}_\tau^k,\frac{u_\tau^k-u_\tau^{k-1}}\tau\Big\rangle
\end{align}
where again $\Gamma:=\partial\Omega$, we obtain 
approximately the rate of stored energy and dissipation together, 
which can be used to express the overall energy balance as
in \eqref{semi-engr-disc-3} at least approximately by using
boundary values only except the bulk contribution of $f$ from 
\eqref{rhs-of-f}, namely
\begin{align}\nonumber
&\hspace*{-4em}\int_0^t\!\bigg(\int_\Gamma\mathfrak{t}(e(\bar v_{\chi,\tau}))
{\cdot}\DT u_{\chi,\tau}\,\d S
-\langle\DT{\mathfrak{f}}_{\tau},\underline u_{\chi,\tau}\rangle\bigg)\,\d t
+\int_{\GC}\!\!\frac12z_{\chi,\tau}(t)\bbK u_{\chi,\tau}(t){\cdot}u_{\chi,\tau}(t)
-\frac12z_0\bbK u_0{\cdot}u_0+
\alpha\big( z_0{-} z_{\chi,\tau}(t)\big)\,\d S
\\&\qquad\qquad\qquad
-\langle\mathfrak{f}_{\tau}(t),u_{\chi,\tau}(t)\rangle
+\langle\mathfrak{f}_{\tau}(0),u_0\rangle
=\mathfrak E_{\chi,\tau}(t)+
\tau\int_0^t\!\int_\Omega\bbC e(\DT u{\chi,\tau}){:}e(\DT u_{\chi,\tau})\,\d x\d t,
\end{align}
cf.\ \eqref{M} and \eqref{semi-engr-disc-3}. The coefficient 
$\chi{+}\tau/2$ in \eqref{approx-energetics} makes 
this expression only an estimate of the actual energetics 
\eqref{semi-engr-disc-3} which would need rather $\chi$. This additional term 
$\frac12\tau\bbC e(\DT u_{\chi,\tau}){:}e(\DT u_{\chi,\tau})$
will vanish if $\tau\to0$, being of the order $\mathscr{O}(\tau/\chi)$. 
Yet, it may not be entirely negligible for small $\chi$, which is a certain drawback of
the BEM implementation. 

BEM also allows for avoiding transformation \eqref{rhs-new} of the 
Dirichlet condition if $u$ (or here rather $v$) is considered 
only on $\GC$ which is, due to \eqref{ass-geom}, far from $\GDir$
and thus the partial derivative $\calE_t'(\cdot,u,z)$ have a good
sense. In fact, the calculations presented below have been obtained by BEM
implemented by a so-called collocation method. 

In what follows, we use this implementation in a two-dimensional geometry 
and, as already mentioned, {\it isotropic material}. In this situation, the 
Poincar\'e-Steklov operator involved in BEM is well known; cf.\ 
\cite[Sect.2.2]{HsiWen08BIE}. As for the material, more specifically we use 
\begin{subequations}\label{bulk-mater}\begin{align}
&E=70\,\text{GPa}&&\text{(Young modulus)},\\
&\nu=\begin{cases}0& \text{(in Sect.\,\ref{sec-0D})},\\[-.3em]
0.35&\text{(in Sect.\,\ref{sec-2D})},\end{cases} 
&&\text{(Poisson ratio)};&&
\end{align}\end{subequations}
thus, with  $\delta_{ij}$ standing for the Kronecker symbol,
the elastic moduli tensor used in the previous sections takes the form 
\begin{align*}
\bbC_{ijkl}=\frac{\nu E}{(1{+}\nu)(1{-}2\nu)}\delta_{ij}\delta_{kl}
+\frac E{2{+}2\nu}(\delta_{ik}\delta_{jl}{+}\delta_{il}\delta_{jk}).
\end{align*}
The viscosity of material describe by the relaxation time $\chi$ will be varied 
and adjusted in particular cases below.

\subsection{Computational experiments: a simple test geometry}\label{sec-0D}

In this section, we test the two-dimensional algorithm on a 0-dimensional example 
from \cite{Roub??ACVE} where the viscous solutions as well as the limit for $\chi\to0$ are 
explicitly known, together with a resulting nontrivial defect measure $\mu$. 
We choose a rectangular specimen glued on one side and pulled on the 
opposite one by gradually increasing Dirichlet load in the normal direction, 
cf.\ Figure~\ref{fig6:delam-1D}. Choosing the Poisson ratio 0 makes the 
quasistatic problem essential 0-dimensional (i.e.\ the strain, stress,  
dissipation rates, and delamination $z$ are spatially constant). Such sort of
tests are common in building geophysical models where it is 
called a one-degree-of-freedom slider.

\begin{my-picture}{.7}{.12}{fig6:delam-1D}
\psfrag{GN}{\footnotesize $\GNeu$}
\psfrag{GD}{\footnotesize $\GDir$}
\psfrag{GC}{\footnotesize $\GC$}
\psfrag{elastic}{\footnotesize visco-elastic body}
\psfrag{obstacle}{\footnotesize rigid obstacle}
\psfrag{adhesive}{\footnotesize adhesive}
\psfrag{LC}{$L_c$}
\psfrag{L}{\footnotesize $L=\ $100\,mm}
\psfrag{H}{\footnotesize $H=$}
\psfrag{12.5}{\scriptsize 12.5mm}
\psfrag{loading}{\footnotesize loading}
\psfrag{number}{}
\psfrag{numbers}{\hspace*{-2.5em}\footnotesize boundary element numbers}
\hspace*{4em}\includegraphics[clip=true,width=.75\textwidth]{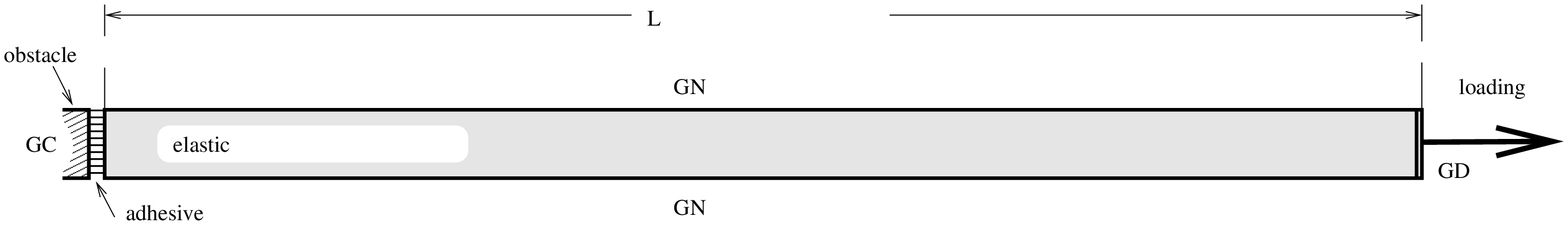}
\end{my-picture}
\begin{center}
{\small Fig.\,\ref{fig6:delam-1D}.\ }
\begin{minipage}[t]{.8\textwidth}\baselineskip=8pt
{\small
Essentially a 0-dimensional experiment (if the material is incompressible) 
with gradually increasing Dirichlet load and explicitly known solution.
}
\end{minipage}
\end{center}

\def\tKVnu{t_{_{^{\mathrm{RUP}}},\chi}}
\def\uKVnu{w_{_{^{\mathrm{RUP}}},\chi}}
\def\eKVnu{\E_{_{\mathrm{KV},}\mbox{$_\chi$}}}
\def\bC{E}
\def\bK{K}

\noindent
Considering still an isotropic adhesive $\bbK_{ij}=K\delta_{ij}$, 
the initial condition $u_0=0$ and $z_0=1$, and gradually increasing
Dirichlet load $w_\Dir(t)=v_\Dir t$, we know analytically the solution
of the viscous problem as well as the Kelvin-Voigt approximable
solution of the inviscid problem (which is probably even unique).
More specifically, there is a time, let us denote it by 
$\tKVnu$ or (for the limit $\chi\to0$) by $t_{_{^{\mathrm{RUP}}}}$,
when the spontaneous and complete rupture 
happens; $\tKVnu$ is determined only rather implicitly 
as a solution $t$ of the transcendental equation 
$(a_0t+b_\chi(1{-}{\rm e}^{-t/t_\chi}))^2=2\alpha/\bK$ 
with the coefficients\vspace*{-.5em}  
\begin{align}\label{coef}
a_0=\frac{\bC }{\bC{+}L\bK}v_\Dir, 
\qquad b_\chi=
-\chi\frac{L\bC\bK}{(\bC{+}L\bK)^2}v_\Dir,\quad\text{ and }\quad
t_\chi=\chi\frac{\bC}{\bC{+}L\bK},
\end{align} 
and for $\chi\to0$ it converges monotonically
to some limit, let us denote it by $t_{_{^{\mathrm{RUP}}}}$; more
specifically,\vspace*{-.5em} 
\begin{align}\label{t-break}
\tKVnu\nearrow 
t_{_{^{\mathrm{RUP}}}}=\frac{\bC{+}L\bK}{v_\Dir\bC}\sqrt{\frac{2\alpha}{\bK}}.
\end{align}
For $\chi>0$, the response is given by\vspace*{-.5em} 
\begin{subequations}\begin{align}
&z_\chi(t,x)=\begin{cases}1&\text{for }t<\tKVnu\,,\\
0&\text{for }t>\tKVnu\,,\end{cases}
\intertext{and the stress $\sigma_\chi(t,\cdot)=\sigma_\chi(t)$ 
and the viscous dissipation rate $\chi\bC e(\DT u_\chi){:}e(\DT u_\chi)$
are constant in space (with values denoted by $\sigma_\chi(t)$ and $r_\chi(t)$,
respectively), while the displacement $u_\chi(t,\cdot)$ is affine with 
$u_\chi(t,L)=w_\Dir(t)$ and $u_\chi(t,0)=w_\chi(t)$ with}
\nonumber\\[-3.3em]&\label{explicit-w}
w_\chi(t)=\begin{cases}(v_\Dir{-}a_0)t
-b_\chi\big(1{-}{\rm e}^{-t/t_\chi}\big)&\text{for }t<\tKVnu\,,\\
\big(\uKVnu{-}v_\Dir\tKVnu\big){\rm e}^{-(t-\tKVnu)/\chi}&\text{for }t>\tKVnu\,,
\end{cases}
\\&
\sigma_\chi(t)=\bK\frac{v_\Dir t{-}w_\chi(t)}L=\chi E\frac{v_\Dir{-}\DT w_\chi(t)}L
+E\frac{v_\Dir t{-}w_\chi(t)}L,
\label{explicit-s}
\\\label{explicit-r}&
r_\chi(t)= \begin{cases}
\displaystyle{\chi\frac{\bC\bK^2 v_\Dir^2}{(\bC{+}L\bK)^2}
\big(1{-}{\rm e}^{-t/t_\chi}\big)^2}
&\text{for }t<\tKVnu\,,
\\\displaystyle{\frac1\chi\bC
\Big(\frac{\uKVnu{-}v_\Dir\tKVnu}L\Big)^2
{\rm e}^{-2(t-\tKVnu)/\chi}}
&\text{for }t>\tKVnu\,,
\end{cases}\end{align}\end{subequations}
where $a_0$, $b_\chi$, and $t_\chi$ are from \eqref{coef}
and $\uKVnu:=(v_\Dir{-}a_0)\tKVnu-b_\chi\big(1{-}{\rm e}^{-\tKVnu/t_\chi}\big)$.
From \eqref{explicit-r}, one can see that
the viscous dissipation rate $\chi\bbC e(\DT u_\chi){:}e(\DT u_\chi)$
concentrates in time 
when $\chi\to0$ and, referring to \eqref{ch6:delam-slow-conv-c},
the resulting {\it defect measure} $\mu$ takes the form\vspace*{-.5em}  
\begin{align}\label{mu-known}
\mu=\frac{\calE_{_{^{\mathrm{RUP}}}}}{{\rm meas}(\Omega)}(\delta_{t_{_{^{\mathrm{RUP}}}}}\!\otimes 
1\hspace{-.3em}{\rm l})\qquad\text{ with }\ \ \ \calE_{_{^{\mathrm{RUP}}}}=\alpha\frac\bK\bC
\end{align}
with $\delta_{t}\!\in\!{\rm Meas}(\bar I)$ denoting the Dirac measure
supported at $t$ and $1\hspace{-.3em}{\rm l}\!\in\!{\rm Meas}(\Omega)$ is 
the spatial constant measure with density 1 (i.e.\ the Lebesgue measure)
on $\Omega$, and $\calE_{_{^{\mathrm{RUP}}}}$ is the energy stored in the bulk at the
time of rupture $t_{_{^{\mathrm{RUP}}}}$ when also the driving force 
$\mathfrak{d}_\chi=\frac12\bbK w_\chi{\cdot}w_\chi$ reaches the activation
threshold $\alpha$; cf.\ \cite{Roub??ACVE} for details about this calculation. 
Although $\mu$ is known and thus we could design the strategy from Table~2 
to control also the difference 
$\chi\bC e(\DT u_{\chi,\tau}){:}e(\DT u_{\chi,\tau})-\mu$
in some norm on ${\rm Meas}(\bar Q)$ which would be weakly* continuous, we 
intentionally do not want it because, in general (as also e.g.\ in 
Sect.~\ref{sec-2D} below), $\mu$ is not known.

In addition to \eqref{bulk-mater}, we consider $K=150\,$GPa/m,
$\alpha=375\,$J/m$^2$, $v_\Dir=267\,\mu$m/s,
and $T=0,375\,$s. The length of the specimen is $L=0.1\,$m, as already
depicted on  Figure~\ref{fig6:delam-1D}, while its
cross-section is not important in this experiment. It is important that
the implementation is able to hold the energetics with a good accuracy
that can be efficiently controlled by making the time step small, 
as proved theoretically in Proposition~\ref{prop-Err} and shown on
Figure~\ref{fig:1-D-residuum} for a moderate selected viscosity $\chi$.
The BEM spatial discretisation was coarse as all the quantities 
are either constant or affine in space in this ``1-dimensional'' example,
so the coarseness of the spatial discretisation is irrelevant.

\begin{my-picture}{.55}{.32}{fig:1-D-residuum}
\psfrag{*********************}{\footnotesize \hspace*{-2em}residuum $-\mathfrak{E}_{\chi,\tau}$ in 
\eqref{semi-engr-disc-3}}
\psfrag{tau=T/25}{\tiny $\tau=T/25$}
\psfrag{tau=T/50}{\tiny $\tau=T/50$}
\psfrag{tau=T/100}{\tiny $\tau=T/100$}
\psfrag{tau=T/200}{\tiny $\tau=T/200$}
\psfrag{tau=T/400}{\tiny $\tau=T/400$}
\psfrag{tau=T/800}{\tiny $\tau=T/800$}
\psfrag{time}{\hspace*{3.3em}\footnotesize time\hspace*{5em}$t$}
\hspace*{11em}\vspace*{0em}\includegraphics[clip=true,width=.53\textwidth,height=.32\textwidth]{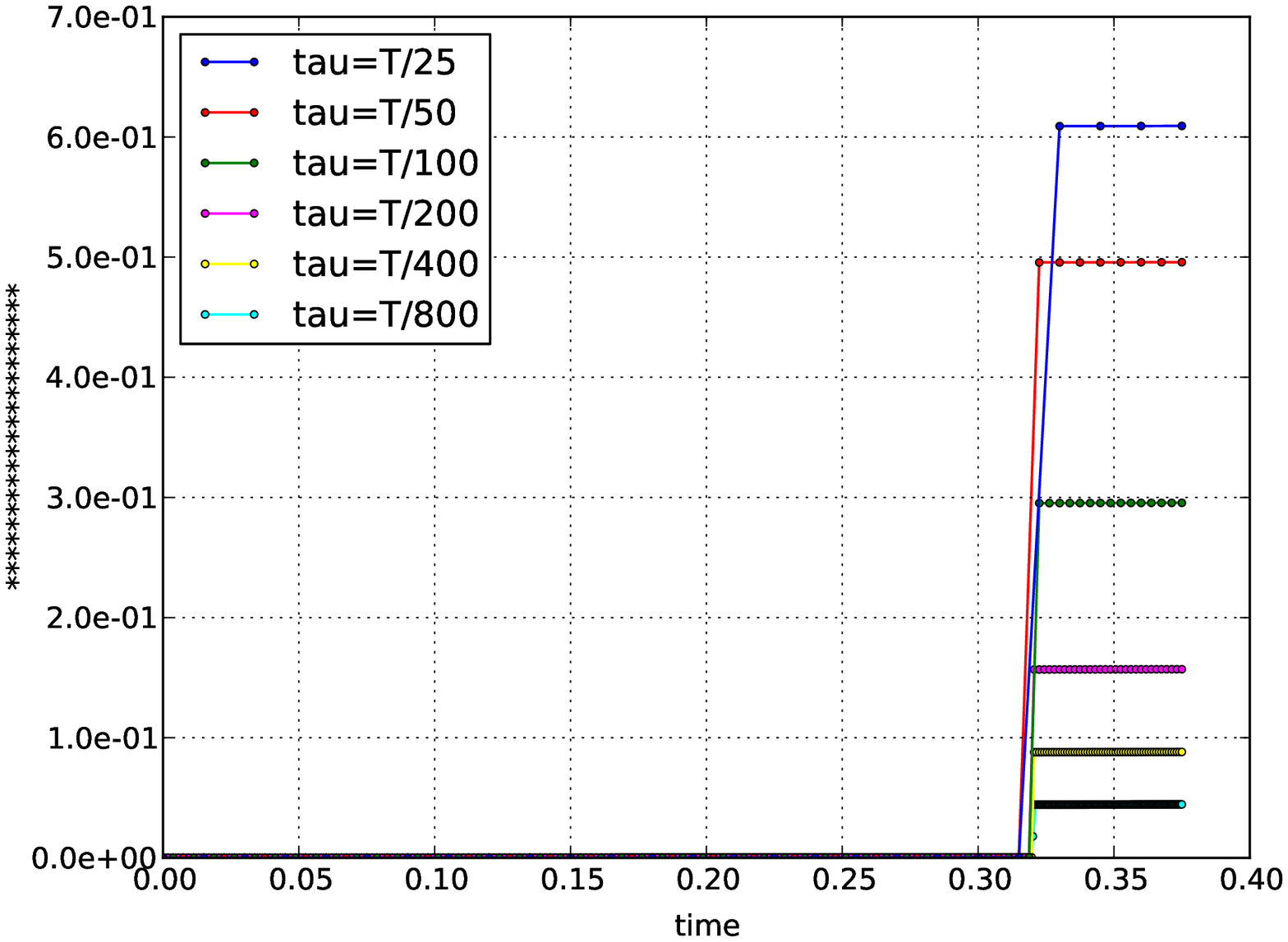}
\end{my-picture}
\\[-2.5em]
\begin{center}
{\small Fig.\,\ref{fig:1-D-residuum}.\ }
\begin{minipage}[t]{.75\textwidth}\baselineskip=8pt
{\small
Illustration of the time-dependent residuum $-\mathfrak{E}_{\chi,\tau}(\cdot)$ 
in the energy balance \eqref{semi-engr-disc-3} for $\tau$ gradually decreasing 
as depicted from up to down,
while $\chi=6.25\!\times\!10^{-3}\,$s is fixed. The numerical error occurs 
especially around sudden rupture but is shown to converge to 0 for $\tau\to0$, 
as also proved in \eqref{Err}.
}
\end{minipage}
\end{center}

The interplay between $\chi$ and $\tau$ and its influence on the energy
balance is depicted on Fig.\,\ref{fig:1-D-residuum+}, clearly showing 
a very slow (resp.\ no) convergence for small $\chi>0$ (resp.\ for $\chi=0$). 
The strategy from Table~2 chooses, in fact, a path decaying sufficiently slow 
from the left-upper corner towards the right-down corner in Fig.\,\ref{fig:1-D-residuum+}(left):

\begin{my-picture}{.8}{.3}{fig:1-D-residuum+}
\psfrag{v=0.0}{\scriptsize $\chi{=}\ 0$}
\psfrag{v=0.05}{\scriptsize $\chi{=}\ 50\ $ms}
\psfrag{v=0.025}{\scriptsize $\chi{=}\ 25\ $ms}
\psfrag{v=0.0125}{\scriptsize $\chi{=}12.5\,$ms}
\psfrag{v=0.00625}{\scriptsize $\chi{=}6.25\,$ms}
\psfrag{v=0.003125}{\scriptsize $\chi{=}3.125\,$ms}
\psfrag{v=0.0015625}{\scriptsize $\chi{=}1.5625\,$ms}
\psfrag{v=0.00078125}{\scriptsize $\chi{=}0.7813\,$ms}
\psfrag{v=0.000390625}{\scriptsize $\chi{=}0.3906\,$ms}
\psfrag{v=0.0001953125}{\scriptsize $\chi{=}0.1953\,$ms}
\psfrag{L1_norm}{\footnotesize $L^1$-norm of $\mathfrak{E}_{\chi,\tau}$}
\psfrag{Linf_norm}{\footnotesize $L^\infty$-norm of $\mathfrak{E}_{\chi,\tau}$}
\psfrag{time_step}{\footnotesize \!\!\!\!number of time steps $T/\tau$}
\hspace*{1em}\vspace*{0em}\includegraphics[clip=true,width=.45\textwidth,height=.3\textwidth]{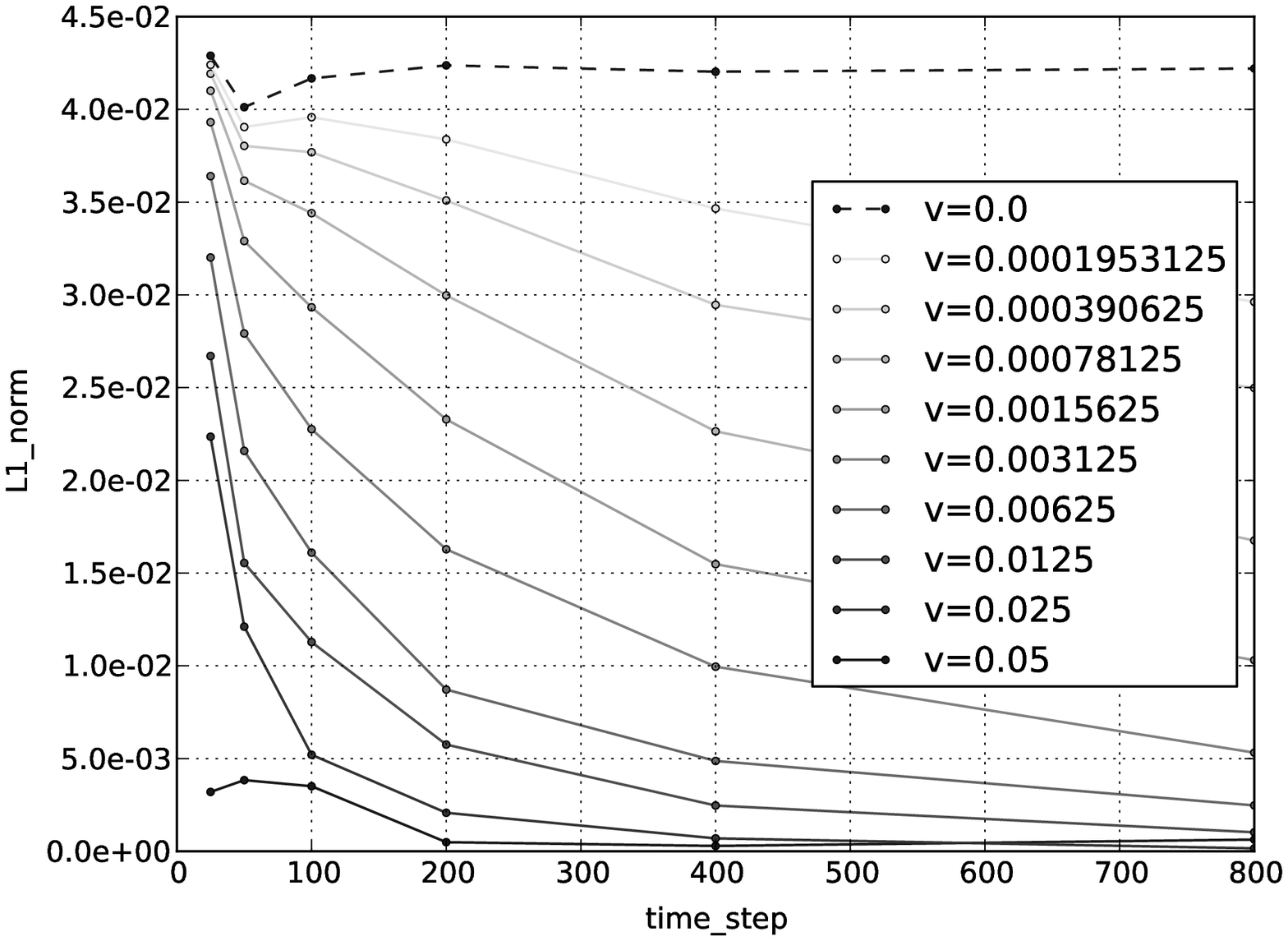}
\hspace*{2em}\vspace*{0em}\includegraphics[clip=true,width=.45\textwidth,height=.3\textwidth]{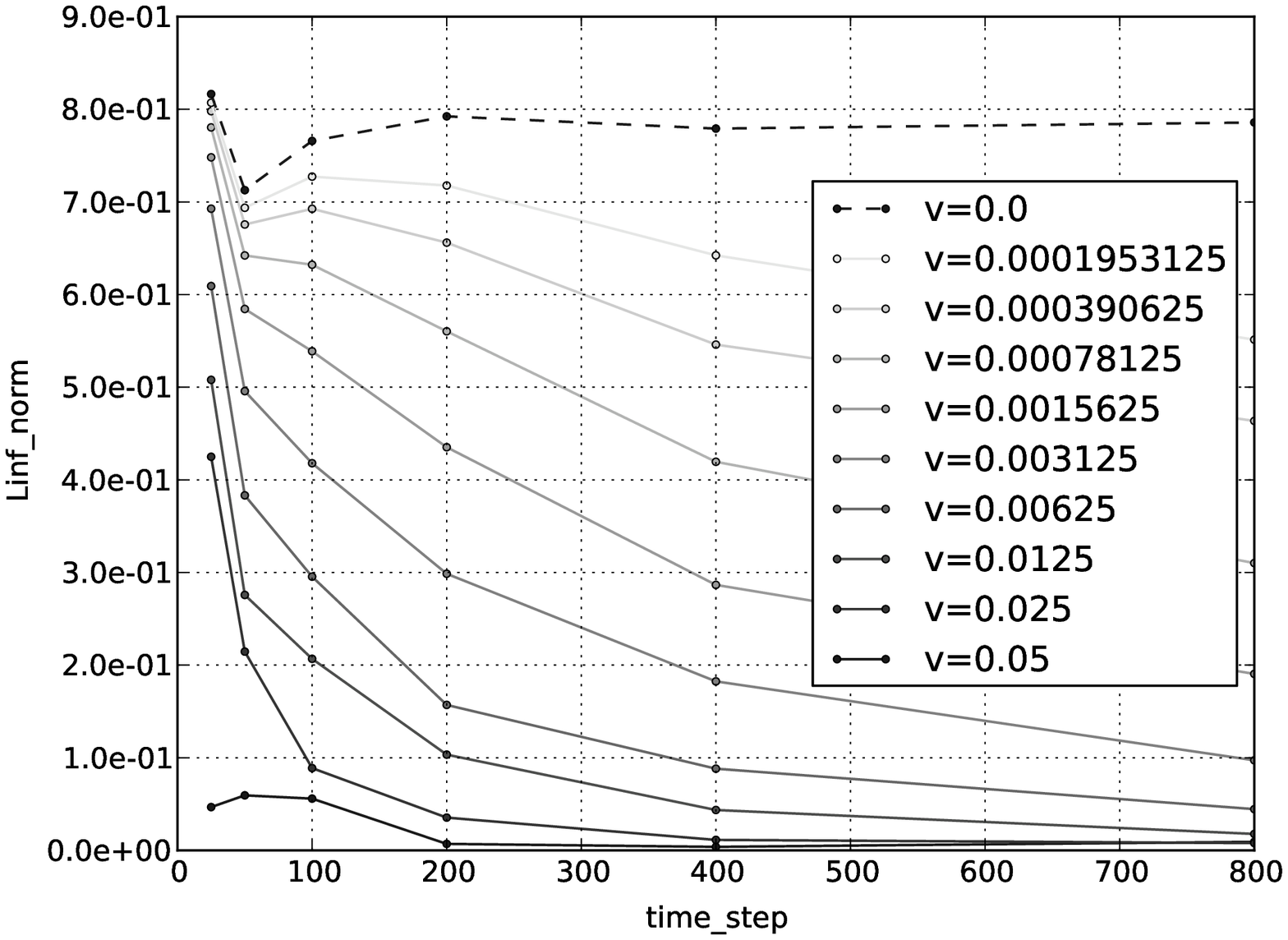}
\end{my-picture}
\\[-2.3em]
\begin{center}
{\small Fig.\,\ref{fig:1-D-residuum+}: Left:\ }
\begin{minipage}[t]{.9\textwidth}\baselineskip=8pt
{\small the convergence of $L^1$-norm of $\mathfrak{E}_{\chi,\tau}$ parametrized 
by $\chi$, documenting the theoretical result from Proposition~\ref{prop-Err}
for $p=1$.  
}
\end{minipage}
\\
{\small \hspace*{2.5em}Right:\ }\begin{minipage}[t]{.89\textwidth}
\baselineskip=8pt
{\small 
$L^\infty$-norm converges similarly in this example 
although this convergence is not theoretically supported by 
Proposition~\ref{prop-Err}.}
\end{minipage}
\end{center}

\noindent
For gradually vanishing viscosity $\chi$, Figure~\ref{fig:1-D-stress-strain} 
displays respectively $w_\chi$ and $\sigma_\chi$ from \eqref{explicit-w} and 
\eqref{explicit-s} calculated numerically by a sufficiently small time step 
$\tau$. 

\begin{my-picture}{.8}{.3}{fig:1-D-stress-strain}
\psfrag{v=0.025}{\scriptsize $\!\chi{=}\ 25\ $ms}
\psfrag{v=0.0125}{\scriptsize $\!\chi{=}12.5\,$ms}
\psfrag{v=0.00625}{\scriptsize $\!\chi{=}6.25\,$ms}
\psfrag{v=0.003125}{\scriptsize $\!\!\!\chi{=}3.125\,$ms}
\psfrag{time}{\hspace*{2em}\footnotesize time\hspace*{6em}$t$}
\psfrag{+++++++++++++++++++++}{\hspace*{-1em}\hspace*{8em}\footnotesize strain}
\psfrag{*********************}{\hspace*{-3em}\hspace*{8.5em}\footnotesize stress}
\hspace*{1em}\vspace*{0em}\includegraphics[clip=true,width=.45\textwidth,height=.3\textwidth]{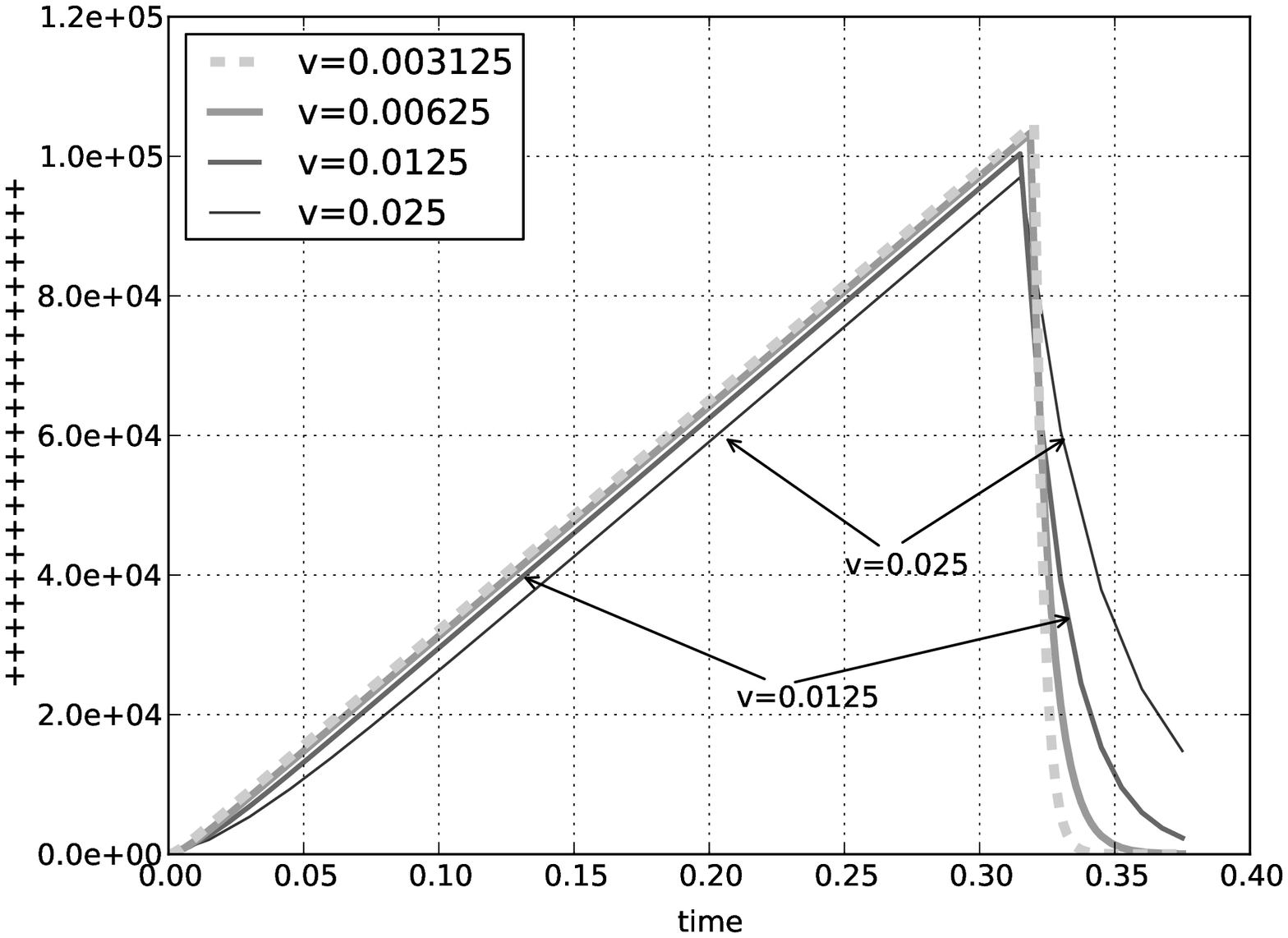}
\hspace*{2em}\vspace*{0em}\includegraphics[clip=true,width=.45\textwidth,height=.3\textwidth]{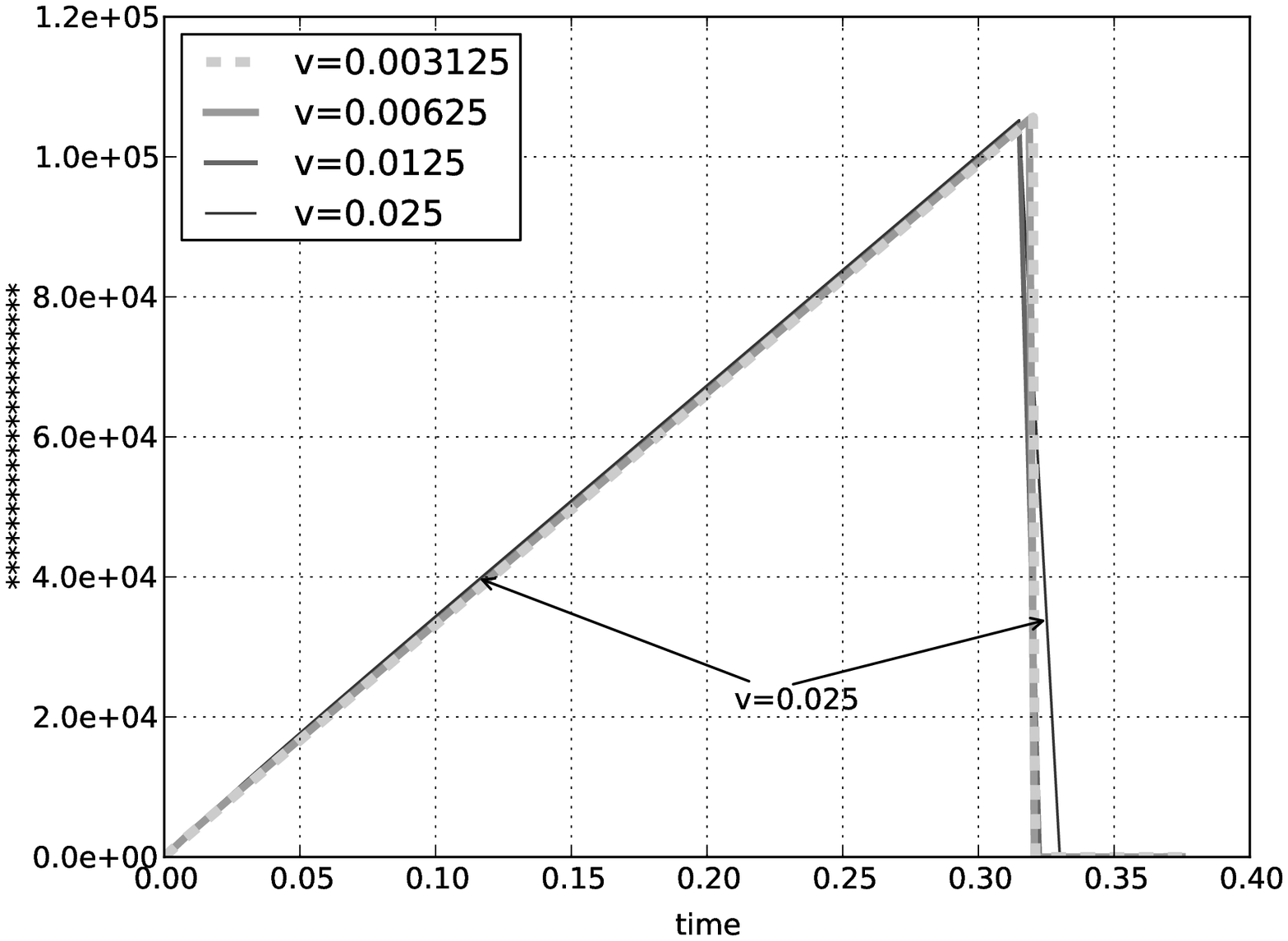}
\end{my-picture}\nopagebreak
\\[-2.5em]
\begin{center}
{\small Fig.\,\ref{fig:1-D-stress-strain}: }\begin{minipage}[t]{.9\textwidth}
\baselineskip=8pt
{\small The strain (left) and stress (right) response; due to the symmetry,
these tensors have only one nonzero component.}
\end{minipage}
\end{center}

\medskip

Here, the defect measure $\mu$ is known from \eqref{mu-known}; now with 
$t_{_{^{\mathrm{RUP}}}}=0.322\,$s and $\calE_{_{^{\mathrm{RUP}}}}=803.75$\,J/m$^3$; 
cf.\ \eqref{t-break} and \eqref{mu-known}. We can thus check the former 
convergence \eqref{joint-conv} at least a-posteriori, which allowed us at 
least to tune the parameters for the algorithm from Table~2. 
Figure~\ref{fig:1-D-visco-mu}-left displays $r_\chi$ from 
\eqref{explicit-r} calculated numerically by a sufficiently small
time step $\tau$. To visualize the weak* convergence to a Dirac measure,
we display rather the overall energy dissipated by viscosity
on the interval $[0,t]$, i.e.\  
$\int_0^t\chi\bbC e(\DT u_{\chi,\tau}){:}e(\DT u_{\chi,\tau})\,\d t$,
which should converge to $\int_0^t\mu\,\d t$ being just a jump 
at time $t_{_{^{\mathrm{RUP}}}}$ of the magnitude $\calE_{_{^{\mathrm{RUP}}}}$,
cf.\ Figure~\ref{fig:1-D-visco-mu}-right; again realize 
that spatial dependence is not interesting here as all
these quantities are constant in space. 

\begin{my-picture}{.8}{.28}{fig:1-D-visco-mu}
\psfrag{time_integral_viscous_dissipation_rate}{\hspace*{-2em}\scriptsize total viscous dissipated energy on $[0,t]$}
\psfrag{*********************}{\hspace*{-3.5em}\scriptsize viscous dissipation rate $r_{\chi,\tau}$ 
\ \ \ [TPa/s]}
\psfrag{3.0e+12}{\scriptsize \ \ \ 3.0}
\psfrag{2.5e+12}{\scriptsize \ \ \ 2.5}
\psfrag{2.0e+12}{\scriptsize \ \ \ 2.0}
\psfrag{1.5e+12}{\scriptsize \ \ \ 1.5}
\psfrag{1.0e+12}{\scriptsize \ \ \ 1.0}
\psfrag{5.0e+11}{\scriptsize \ \ \ 0.5}
\psfrag{0.0e+00}{\scriptsize \ \ \ \ 0}
\psfrag{v=0.025}{\scriptsize $\!\chi{=}\ \,25\ $ms}
\psfrag{v=0.0125}{\scriptsize $\!\chi{=}12.5\,$ms}
\psfrag{v=0.00625}{\scriptsize $\!\chi{=}6.25\,$ms}
\psfrag{v=0.003125}{\scriptsize $\!\!\!\chi{=}3.125\,$ms}
\psfrag{400}{\scriptsize 800}
\psfrag{Time}{\hspace*{2em}\footnotesize time\hspace*{6em}$t$}
\psfrag{time}{\hspace*{2em}\footnotesize time\hspace*{6em}$t$}
\hspace*{1em}\vspace*{0em}\includegraphics[clip=true,width=.45\textwidth,height=.28\textwidth]{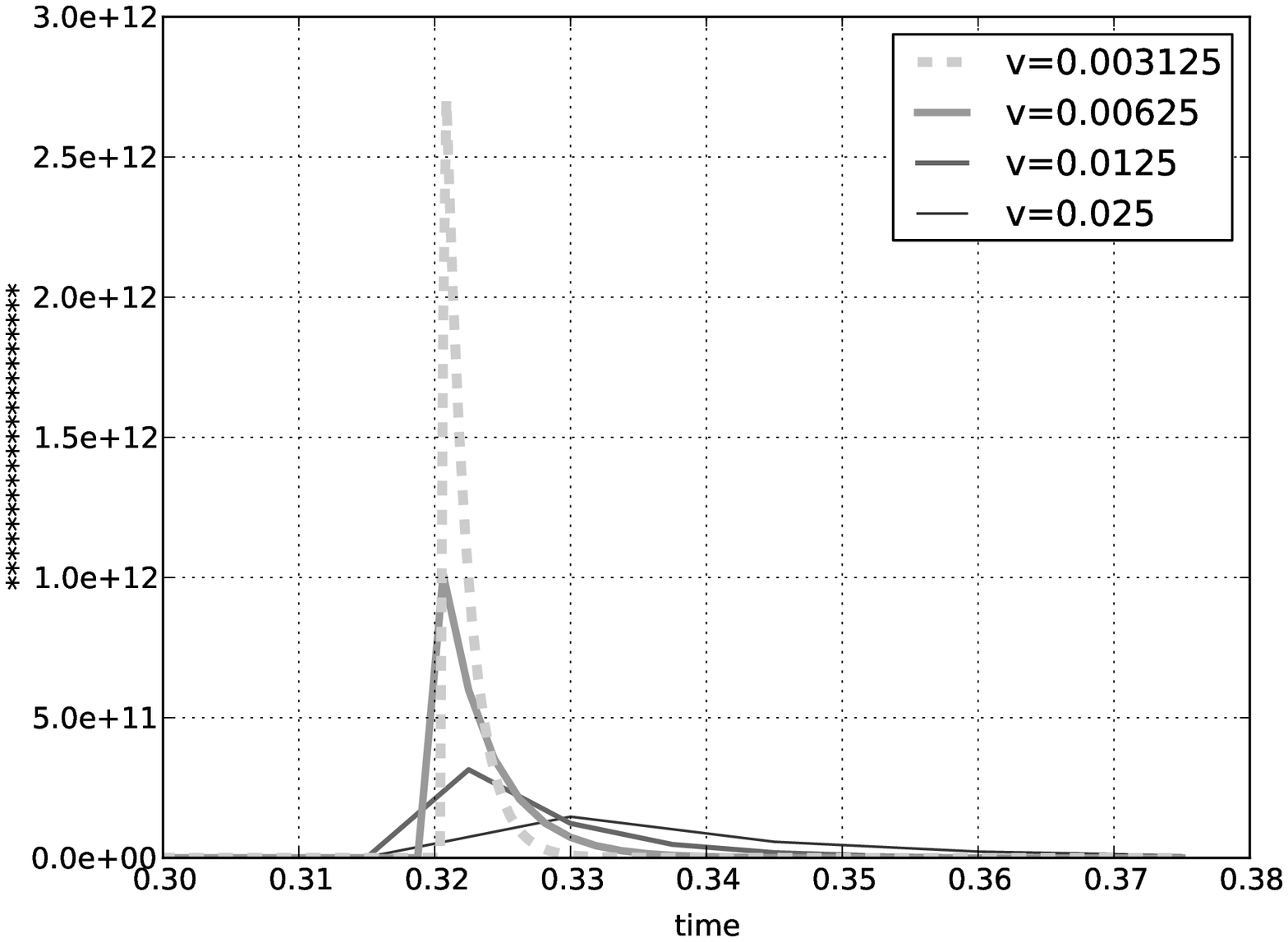}
\hspace*{2em}\vspace*{0em}\includegraphics[clip=true,width=.45\textwidth,height=.28\textwidth]{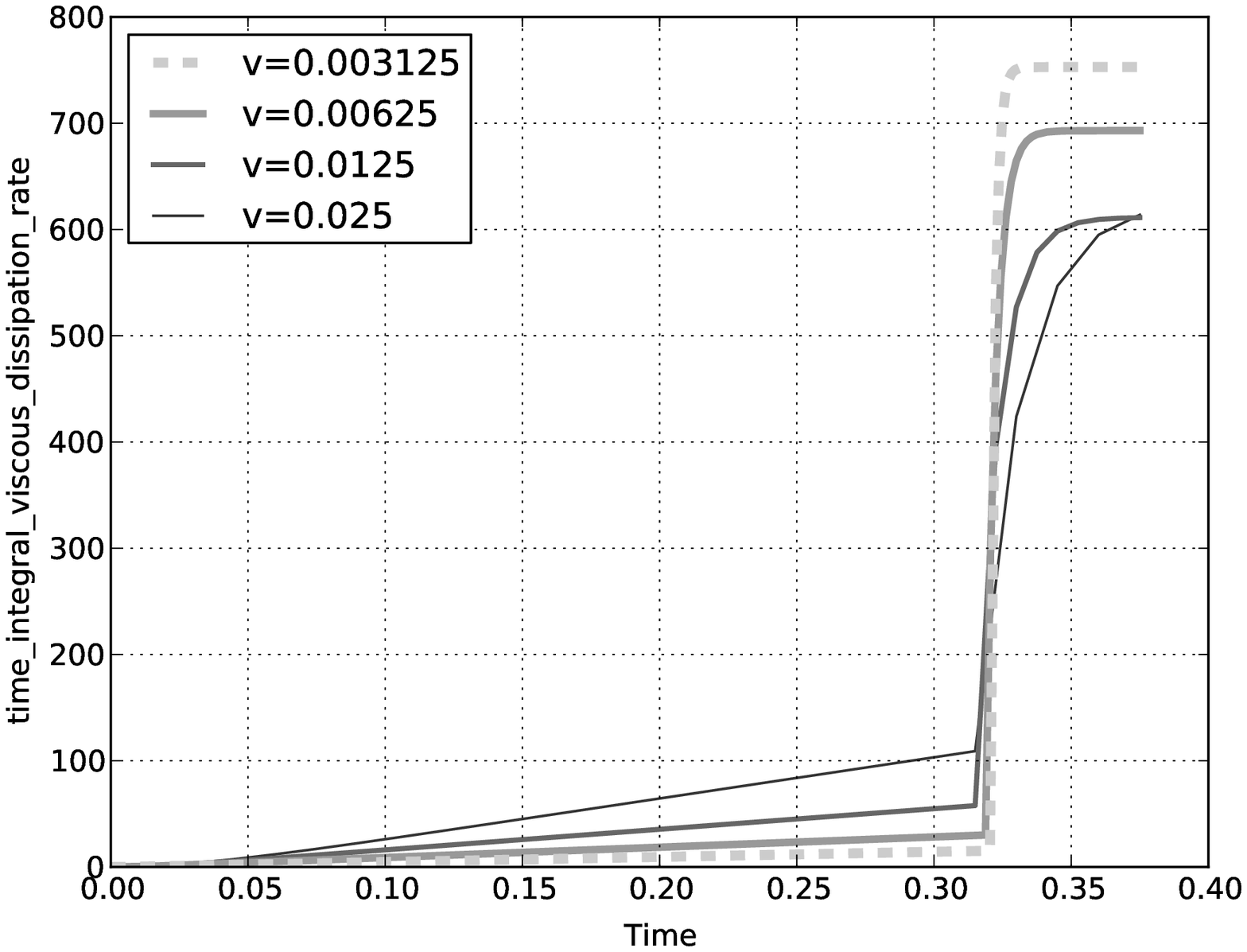}
\end{my-picture}\nopagebreak
\\[-2.5em]
\begin{center}
{\small Fig.\,\ref{fig:1-D-visco-mu}: Left:\ }
\begin{minipage}[t]{.9\textwidth}\baselineskip=8pt
{\small Convergence of the viscous dissipation rate 
$r_{\chi,\tau}=\chi\bbC e(\DT u_{\chi,\tau}){:}e(\DT u_{\chi,\tau})$ 
towards the defect measure $\mu$
from \eqref{mu-known}, i.e.\ here the Dirac at $t_{_{^{\mathrm{RUP}}}}=0.322\,$s for 
$\chi=0.025\!\times\!2^{-k}$ with $k=0,1,2,3$ and decreasing 
$\tau$ chosen according the strategy from Table~2, 
zoomed in and depicted on a selected time subinterval $[0.3\,,\,0.375]$.
}
\end{minipage}
\\
{\small \hspace*{2.5em}Right:\ }\begin{minipage}[t]{.9\textwidth}
\baselineskip=8pt
{\small Energy dissipated by viscosity over $[0,t]$, 
i.e.\ $\int_0^t\chi\bbC e(\DT u_{\chi,\tau}){:}e(\DT u_{\chi,\tau})\,\d t$,
converging to the jump at $t_{_{^{\mathrm{RUP}}}}=0.322\,$s
of the magnitude $\calE_{_{^{\mathrm{RUP}}}}=803.75$\,J. 
Also the convergence $\tKVnu\nearrow t_{_{^{\mathrm{RUP}}}}$ from \eqref{t-break}
is well documented.\\
}
\end{minipage}
\end{center}

\begin{remark}\label{rem-inviscid}(Direct calculation of inviscid problem.)
\upshape
In principle, our semi-implicit time discretisation works for 
$\chi=0$, too. Even, solving directly the inviscid problem is 
algorithmically much simpler. Yet, as pointed out at the end 
of Section~\ref{sec-disc}, we cannot expect reasonable results if $\chi$ will 
converge to 0 too fast with respect to $\tau$, and in particular if straight 
$\chi=0$ would be used. Here, we saw it already on 
Figure~\ref{fig:1-D-residuum+} where, for $\chi=0$, the error in the energy 
balance practically remains constant no matter how the time discretisation 
refines. On Figure~\ref{fig:1-D-visco-mu}, $\chi=0$ would cause all curves 
to degenerate simply to the $t$-axis, which shows {\it fatal non-convergence} 
of the overall viscous dissipation. Thus also the energy 
balance cannot hold. It is surprising 
that $u$-,  $\sigma$- and $z$-responses may still numerically converge
to the correct solutions, as documented on Figure~\ref{fig:1-D-violation},
which may seem to give a very efficient numerical strategy.  
We observed this phenomenon in all our calculations, in particular also 
on Figs.\,\ref{fig-2D-engr} and \ref{fig-2D-engr-2} 
below; cf.\ also the discussion in Sect.~\ref{sect-concl}.

\begin{my-picture}{.85}{.25}{fig:1-D-violation}
\psfrag{time}{\footnotesize time\hspace{7em}$t$}
\psfrag{*********************}{\hspace*{4em}\footnotesize strain}
\psfrag{+++++++++++++++++++++}{\hspace*{5em}\footnotesize stress}
\psfrag{near_zero}{\footnotesize\!\!viscous}
\psfrag{zero}{\footnotesize\!\!inviscid}
\hspace*{3em}\vspace*{0em}\includegraphics[clip=true,width=.4\textwidth,height=.25\textwidth]{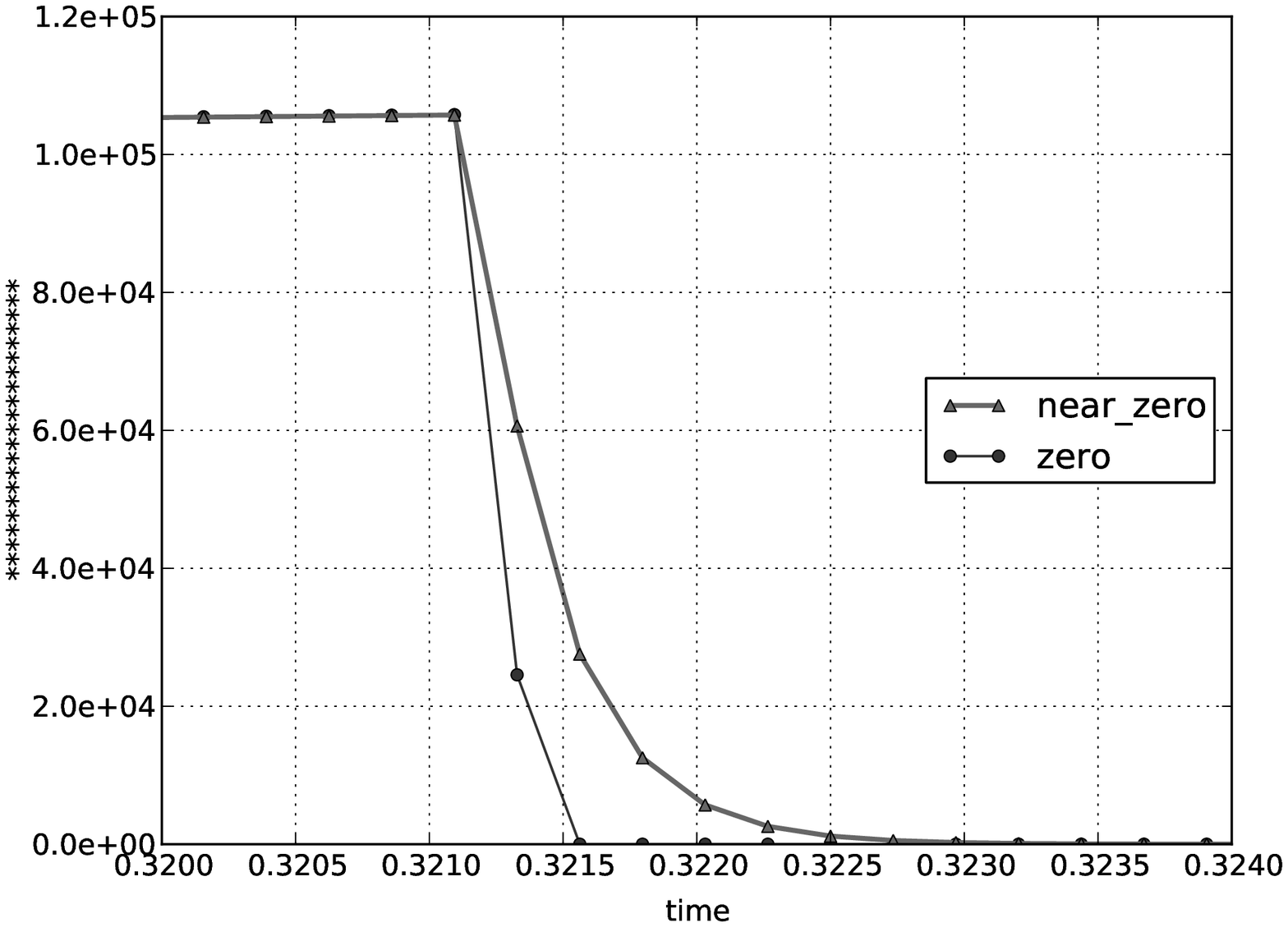}
\hspace*{3em}\vspace*{0em}\includegraphics[clip=true,width=.4\textwidth,height=.25\textwidth]{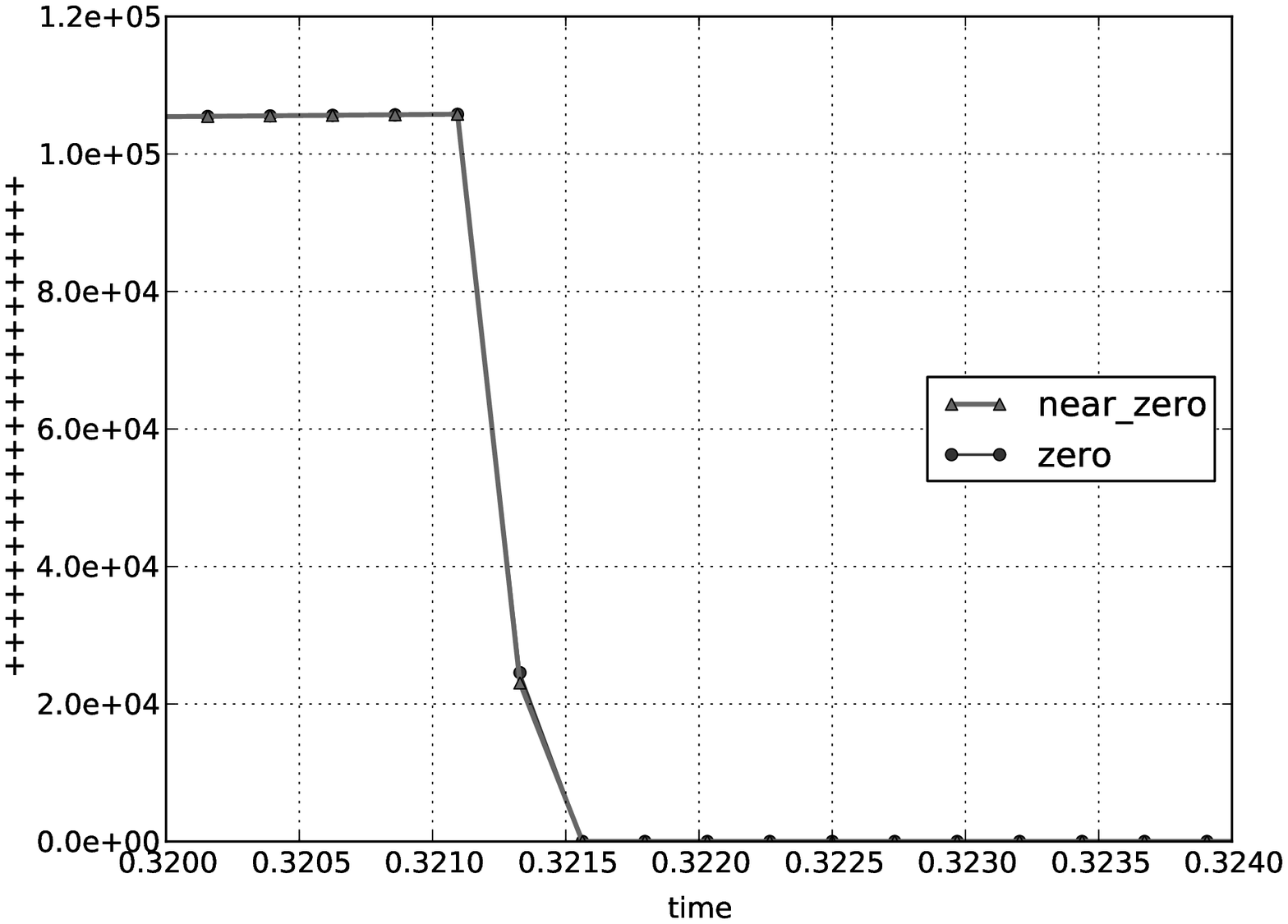}
\end{my-picture}
\\[-2.5em]
\begin{center}
{\small Fig.\,\ref{fig:1-D-violation}.\ }
\begin{minipage}[t]{.8\textwidth}\baselineskip=8pt
{\small A comparison of the strain (left) and stress (right) response
of a energetically justified small-viscosity solution with
an unphysical result without any viscosity obtained by a semi-implicit formula;
strongly zoomed in and depicted on a selected short time subinterval 
around rupture $[0.320\,,\,0.324]$:
a surprisingly good match is achieved although energy does not match at all (since
$\mu\equiv0$ without viscosity), cf.\ also Fig.\,\ref{fig:1-D-residuum+} for $\chi=0$.
}
\end{minipage}
\end{center}
\end{remark} 

\begin{remark}\label{rem-stress-vs-energy}
(Stress- versus energy-driven rupture.)
\upshape
The rupture of Kelvin-Voigt approximable solutions 
is essentially stress driven, while the energy-driven rupture 
(i.e.\ energy dissipated by delamination is compensated by 
the elastic energy got from the bulk and adhesive, being related
to so-called Levitas' maximum realizability principle and leading to
so-called energetic solutions, cf.\ \cite{MieThe04RIHM,MiThLe02VFRI}) 
occurs in general earlier, here in this simple example it would be at 
time $\sqrt{2\alpha(L\bK{+}\bC)/(v_\Dir^2\bK\bC)}$, as noted already in 
\cite{Roub??ACVE}, i.e.\ already at time 0.292\,s. The stress-driven rupture
seems to be much more natural (especially if a large bulk would lead 
to extremely early delamination) and is also preferred in engineering (where 
mostly the existence of solution and the calculations are not analytically 
justified, however), cf.\ \cite{Legu02STCC}, or also the discussion about 
energy versus stress or global versus local minimization in mathematical 
literature \cite{CFMT00RBFE,KnMiZa08ILMC,MiRoSa09MSJR,Stef09VCRI}. 
\end{remark}

\subsection{Computational experiments: a fully 2-D example}\label{sec-2D}

We now want to demonstrate applicability of the above developed
methodology and algorithms to nontrivial situations where the 
defect measure $\mu$ is not known and typically is inhomogeneous, i.e.\
not distributed uniformly in space. Although we keep correct 
energetics via tracking numerically the latter convergence in 
\eqref{joint-conv}, it should be emphasized that the calculations 
are not fully reliable because the former convergence in 
\eqref{joint-conv} cannot be checked. At this occasion, it should
be however also emphasized that all the previous studies about
defect measures have had been only purely theoretical
and analytically motivated  
(and being related, like here, with possible lack of 
regularity of weak solutions of various continuum-mechanical 
problems, exhibiting various concentration effects in contrast to 
regular weak solutions where the defect measure vanishes, cf.\ 
\cite{DipMaj87OCWS,Gera91MDM,Gren95DMVP,Feir03DCF,Naum06ETWS})
and, except \cite{Roub??ACVE}, existence of nontrivial defect measures 
had been rather only conjectured. In contrast to it, 
the simulations here represent, to our best knowledge, historically
the very first attempt to see particular nontrivial spatially non-homogeneous
defect measures.

For our computational experiment, we use a similar geometry as in 
Section~\ref{sec-0D} but, in contrast to Figure~\ref{fig6:delam-1D}, with
a delaminating surface $\GC$ on a different side and (in our 2nd experiment) 
also different direction of loading, both intentionally breaking the symmetry 
considered previously in Section~\ref{sec-0D}. We also consider more 
realistic Poisson ratio, cf.\ \eqref{bulk-mater}. The speed of loading was 
taken the same in both experiments: 
$w_\Dir(t)=v_\Dir t$ with $|v_\Dir|=333.3\,\mu$m/s; the direction of $v_\Dir$
was varied: horizontal in the 1st experiment and vertical in the second
experiment, cf.\ Figure\,\ref{fig_m1}. In addition to \eqref{bulk-mater}, we 
consider $\bbK={\rm diag}(K_{\rm n},K_{\rm t})$
with $K_{\rm n}=150\,$GPa/m and $K_{\rm t}=75\,$GPa/m,
and the fracture toughness $\alpha=187.5\,$J/m$^2$.
We use $\chi=0.01$s for all calculations

\begin{my-picture}{.7}{.17}{fig_m1}
\psfrag{GN}{\footnotesize $\GNeu$}
\psfrag{GD}{\footnotesize $\GDir$}
\psfrag{GC}{\footnotesize $\GC$}
\psfrag{elastic}{\footnotesize visco-elastic body $\Omega$}
\psfrag{obstacle}{\footnotesize rigid obstacle}
\psfrag{adhesive}{\footnotesize adhesive}
\psfrag{L}{\footnotesize 250\,mm}
\psfrag{L1}{\footnotesize 225\,mm}
\psfrag{H}{\footnotesize 25\,mm}
\psfrag{loading-I}{\footnotesize 1st loading}
\psfrag{experiment}{\footnotesize experiment}
\psfrag{loading-II}{\footnotesize 2nd loading}
\psfrag{experiment}{\footnotesize experiment}
\hspace*{4.3em}\vspace*{-.1em}{\includegraphics[width=.75\textwidth]{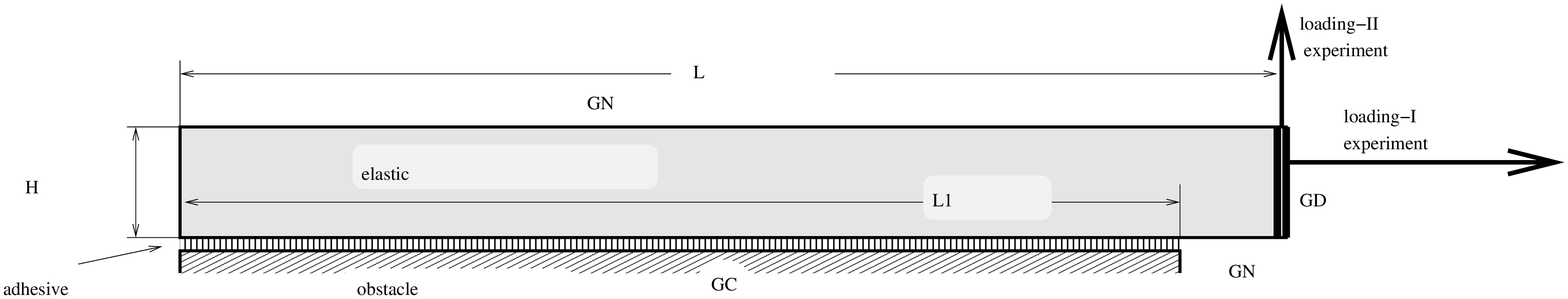}}
\end{my-picture}
\\[-2.5em]
\begin{center}
{\small Fig.\,\ref{fig_m1}.\ }
\begin{minipage}[t]{.57\textwidth}\baselineskip=8pt
{\small Geometry and boundary conditions of the 2-D problem considered.}
\end{minipage}
\end{center}

\noindent
Calculations with $\tau=5\!\times\!10^{-3}$s and $9.33\!\times\!10^{-3}$s have 
been performed 300 time steps, up to $T=1,5\,$s and $T=2.8\,$s for the 1st and 
the 2nd experiment, respectively. Such $T$ was big enough to achieve a 
complete delamination of the whole contact surface.

\subsubsection{Horizontal-loading experiment}\label{sect-2D-1st-exp}

In contrast to the example from Section~\ref{sec-0D}, we have now the traction 
force on $\GC$ nonhomogeneous, and it is interesting to see its evolution in 
time. This is depicted for 6 selected snapshots, starting from 
$t=0.21\,$s with an equidistant step $0.025\,$s, on Fig.\,\ref{fig-2D-I-movie} 
(upper and middle rows) which shows the delamination gradually propagating on 
$\GC$ from right to left. The displacement of the whole boundary has to be 
reconstructed by the Poincar\'e-Steklov operators which is a conventional 
procedure in BEM and is depicted on Fig.\,\ref{fig-2D-I-movie} (lower row).

\begin{my-picture}{.9}{.2}{fig-2D-I-movie}
\psfrag{time_step: 42}{}
\psfrag{time_step: 47}{}
\psfrag{time_step: 52}{}
\psfrag{time_step: 57}{}
\psfrag{time_step: 62}{}
\psfrag{time_step: 67}{}
\psfrag{time_step: 72}{}
\psfrag{normal_traction_x_(m)}{\tiny \hspace{-2em}$_{x\mbox{-coordinate}}$}
\psfrag{tangent_traction_x_(m)}{\tiny \hspace{-2em}$_{x\mbox{-coordinate}}$}
\psfrag{coordinate_y_(m)}{\tiny }
\psfrag{1000000}{\footnotesize 1}
\psfrag{0.10}{}
\begin{tabular}{ccccccc} 
\hspace*{-1.5em}\footnotesize $t=0.21$ & 
\hspace*{-1.5em}\footnotesize $t=0.235$ & 
\hspace*{-1.5em}\footnotesize $t=0.26$ & 
\hspace*{-1.5em}\footnotesize $t=0.285$ & 
\hspace*{-1.5em}\footnotesize $t=0.31$& 
\hspace*{-1.5em}\footnotesize $t=0.335$& 
\hspace*{-1.5em}\footnotesize $t=0.36$\\[-.2em]
\hspace*{-1.5em}\includegraphics[width=.15\textwidth]{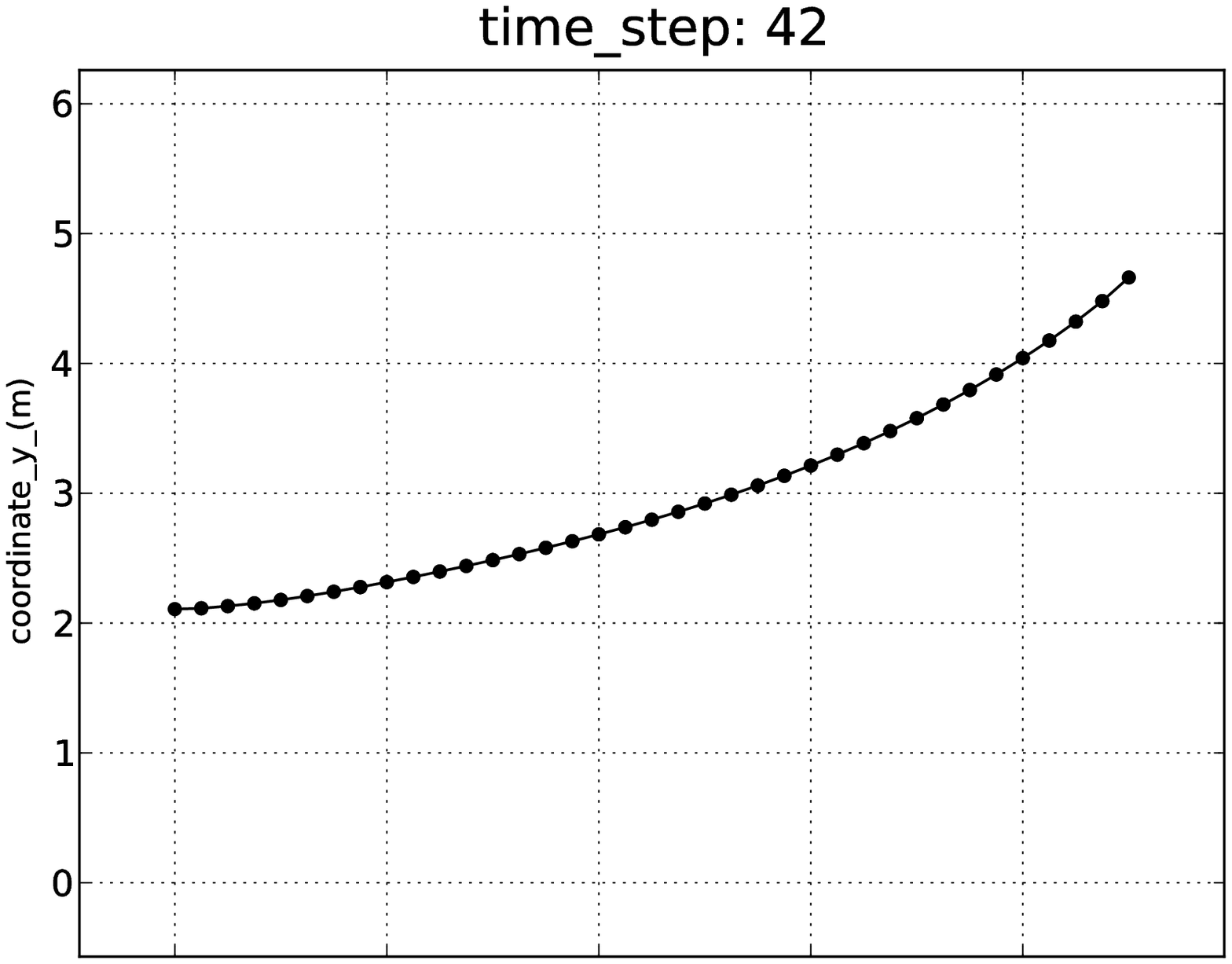}&
\hspace*{-1.5em}\vspace*{-0em}{\includegraphics[width=.15\textwidth]{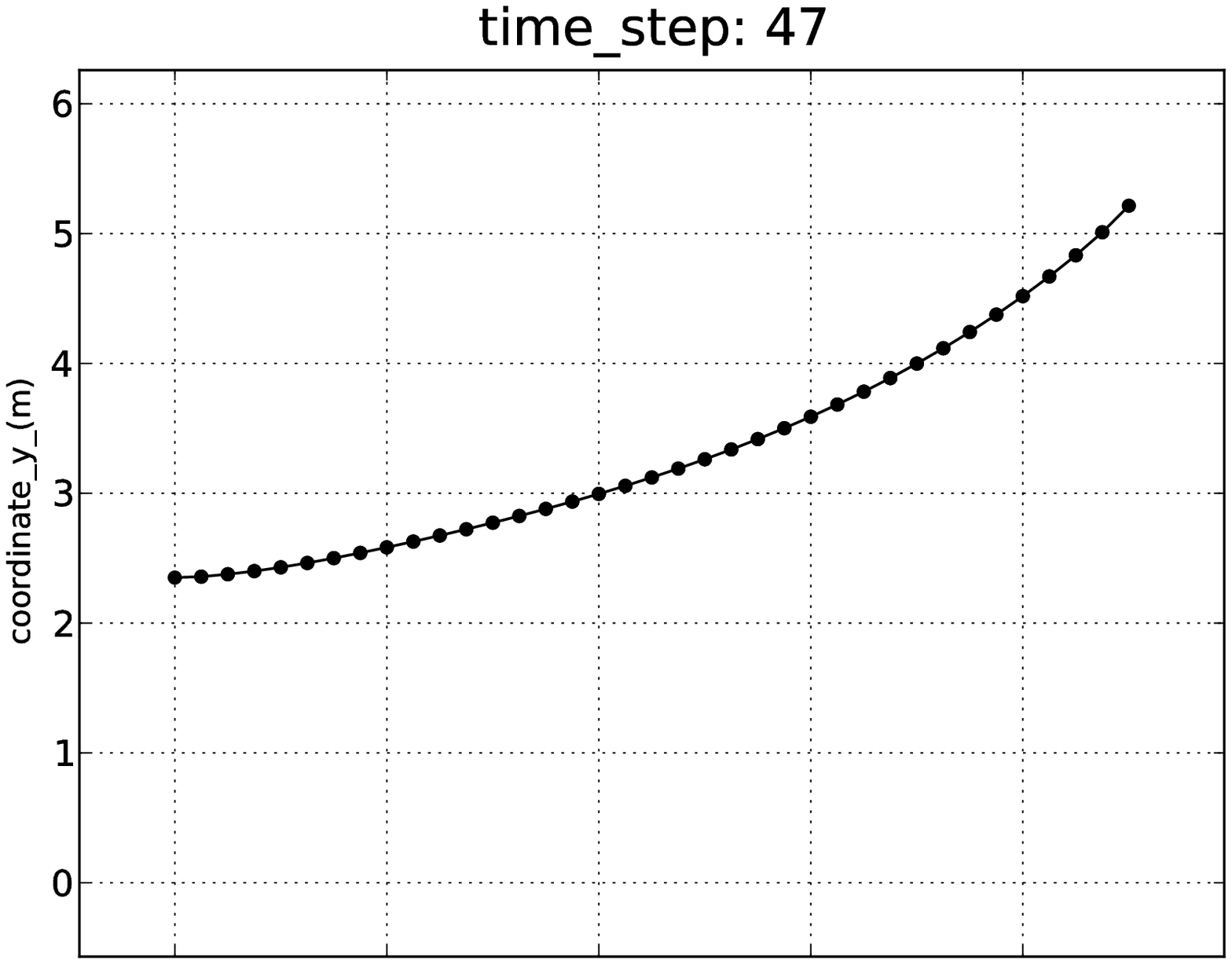}}&
\hspace*{-1.5em}\vspace*{-0em}{\includegraphics[width=.15\textwidth]{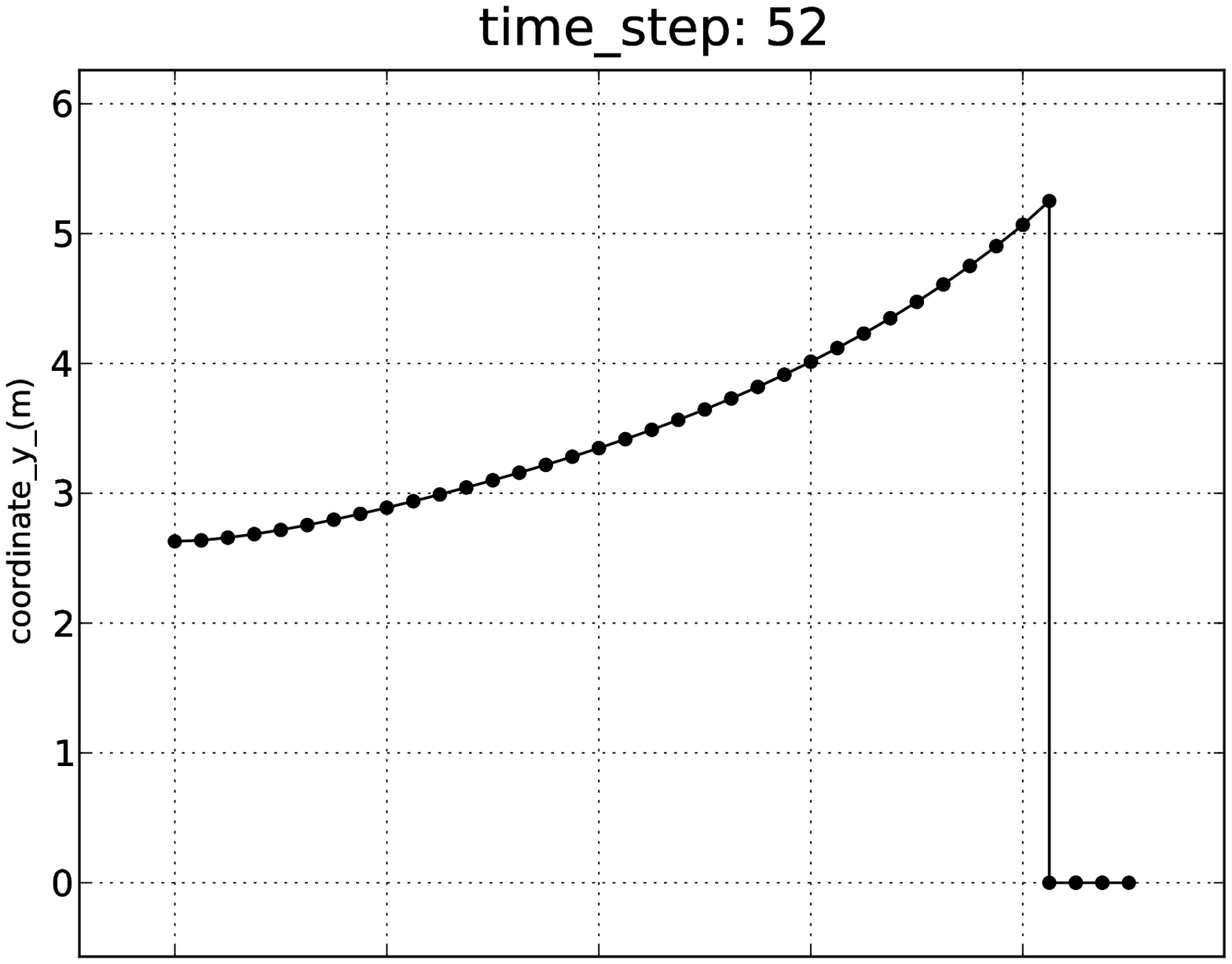}}&
\hspace*{-1.5em}\vspace*{-0em}{\includegraphics[width=.15\textwidth]{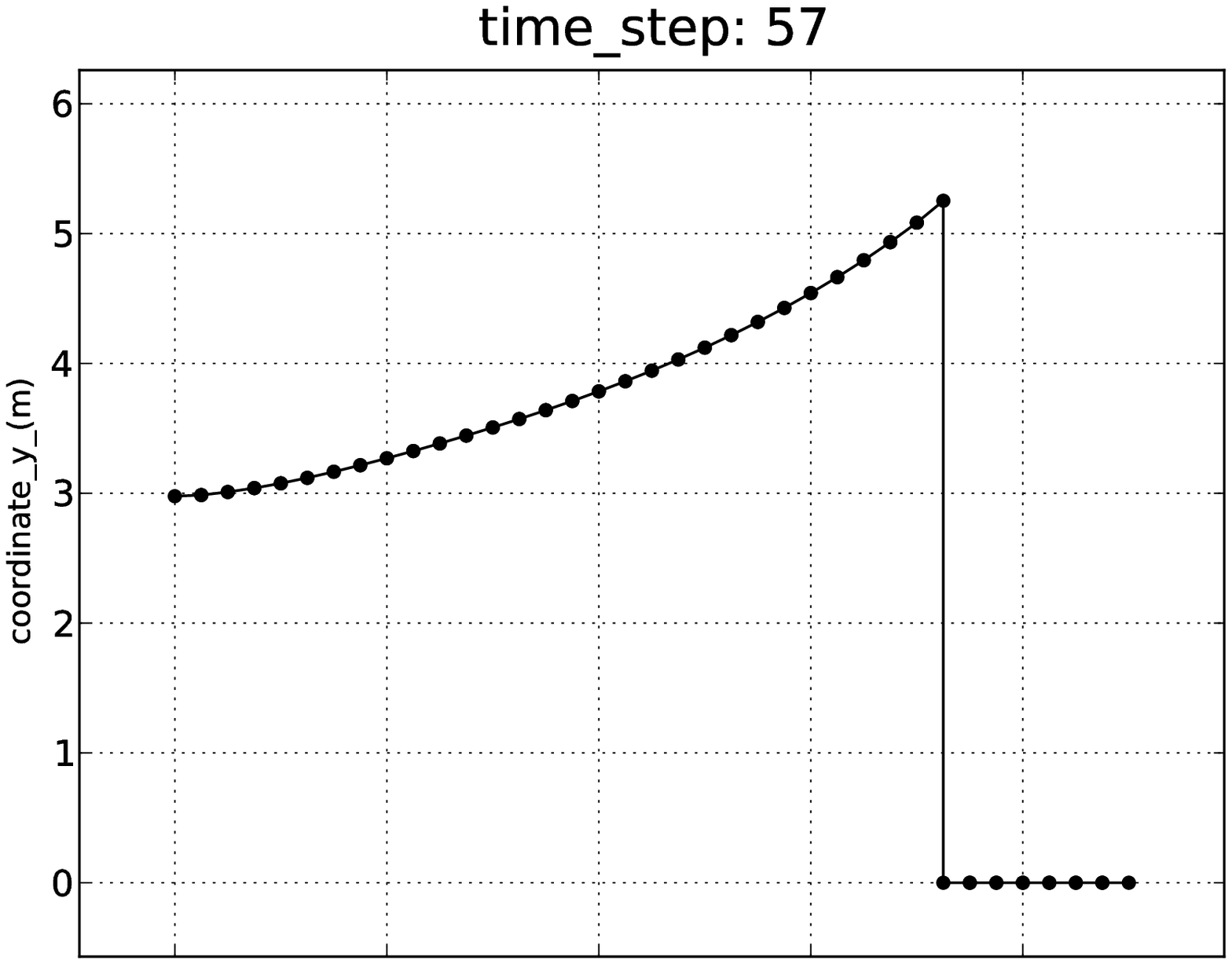}}&
\hspace*{-1.5em}\vspace*{-0em}{\includegraphics[width=.15\textwidth]{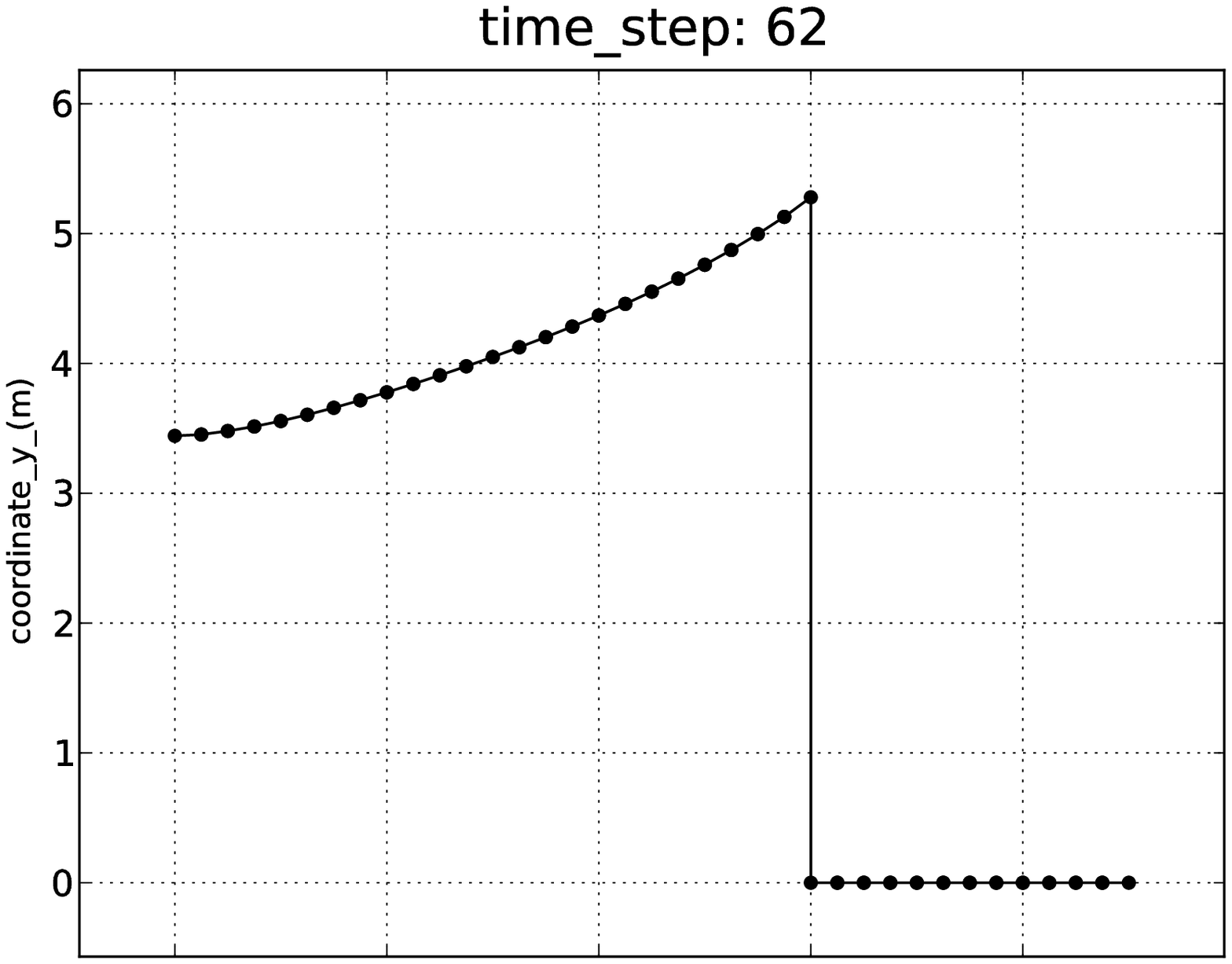}}&
\hspace*{-1.5em}\vspace*{-0em}{\includegraphics[width=.15\textwidth]{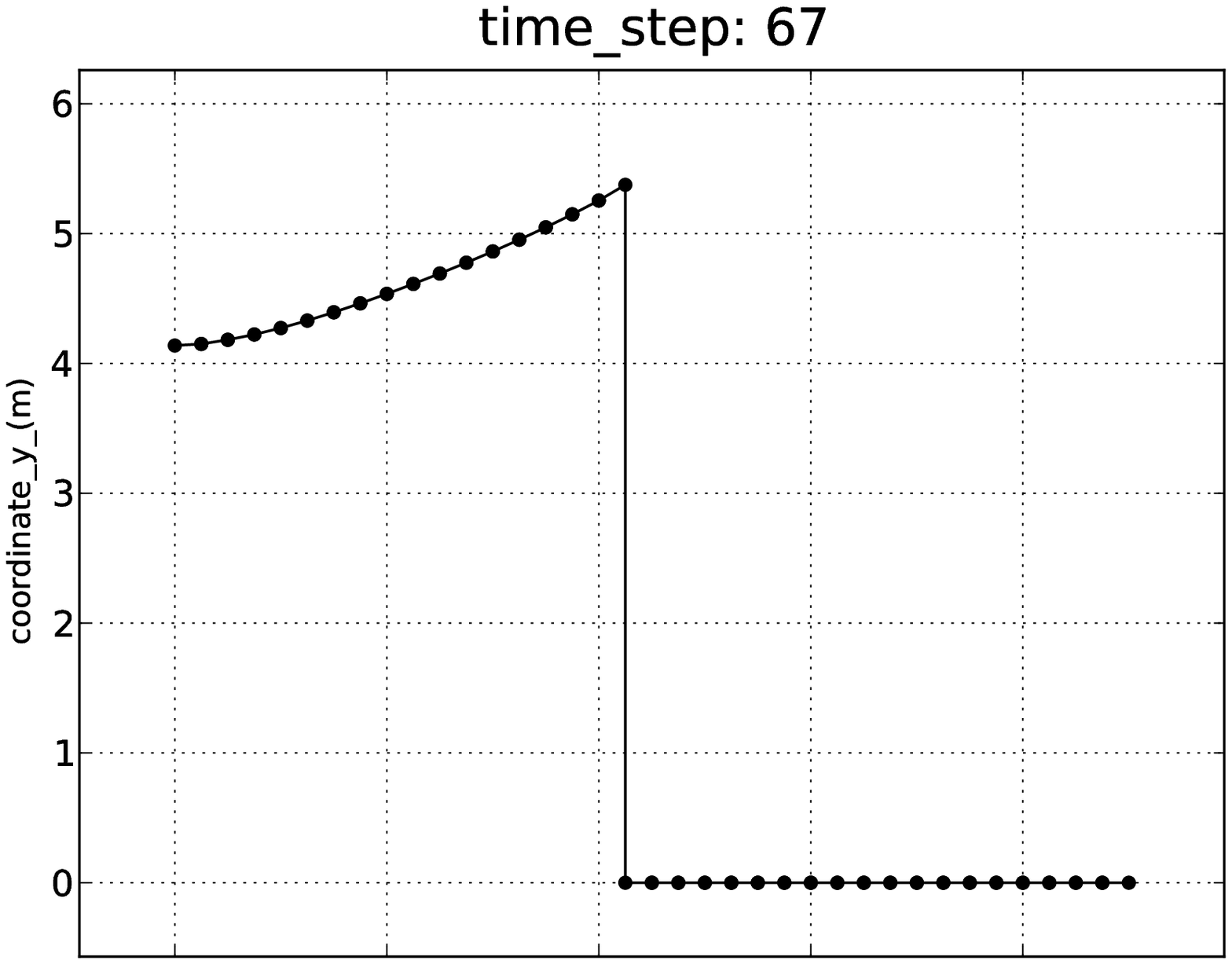}}&
\hspace*{-1.5em}\vspace*{-0em}{\includegraphics[width=.15\textwidth]{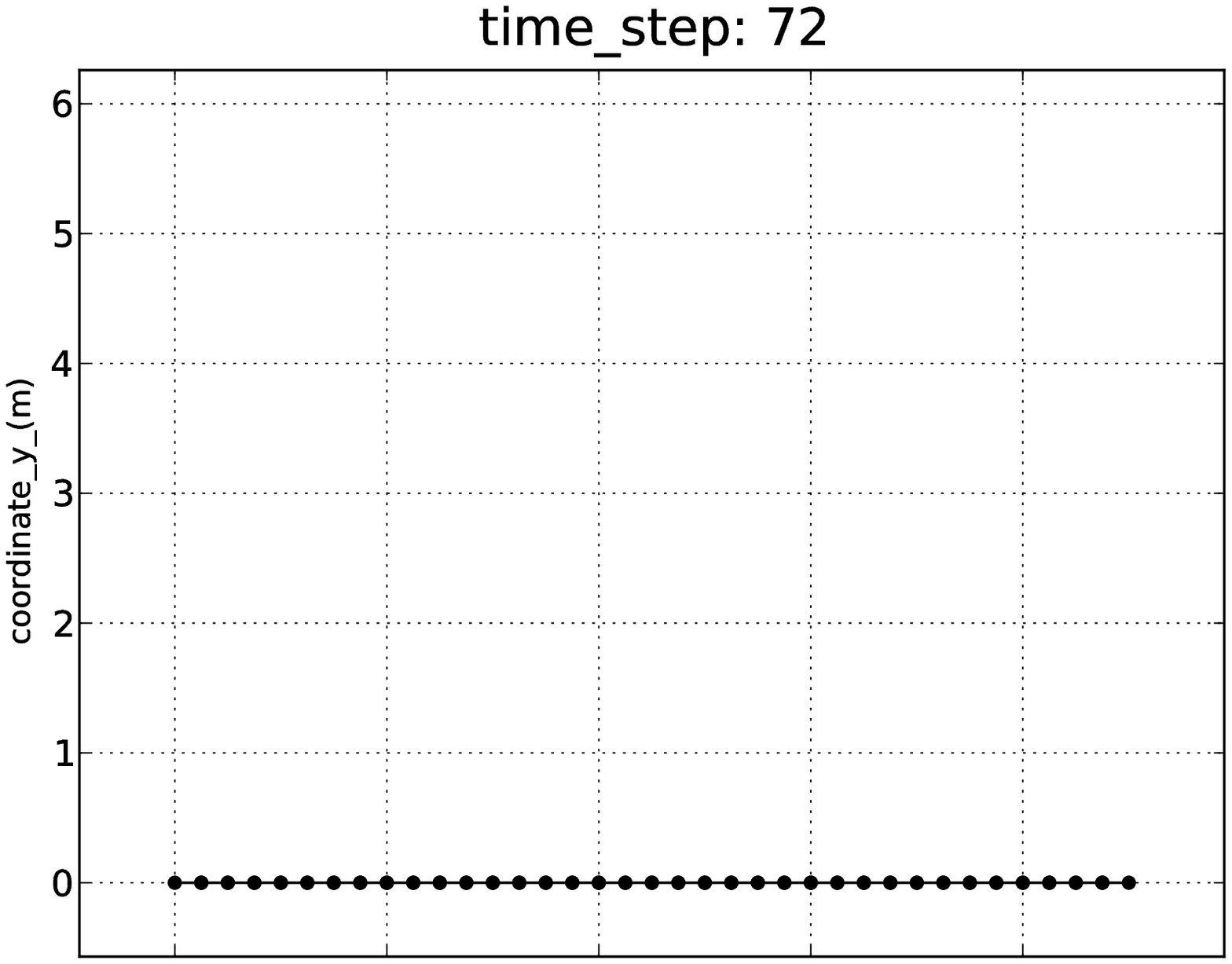}}
\\[-0em]
\hspace*{-1.5em}\includegraphics[width=.15\textwidth]{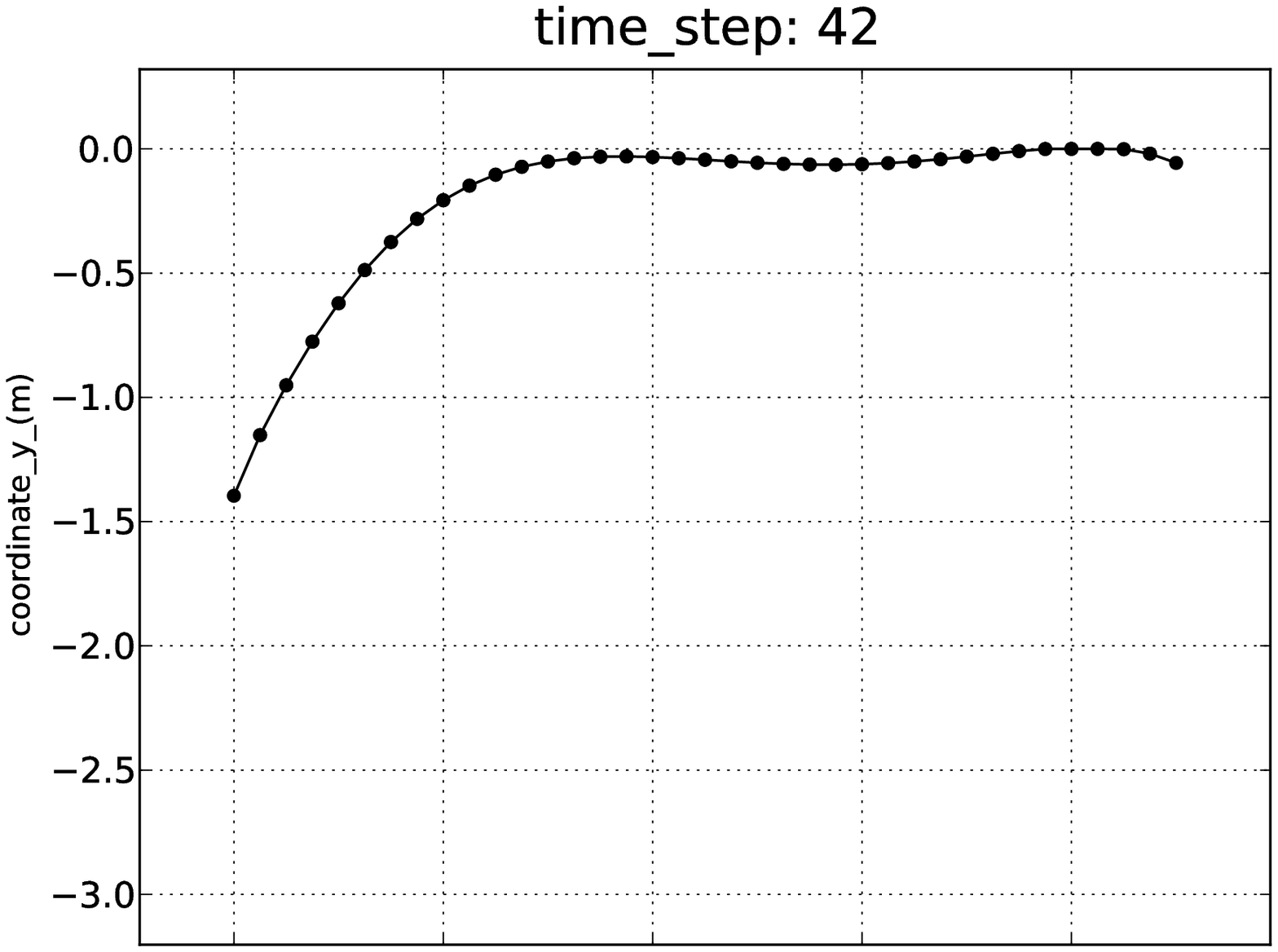}&
\hspace*{-1.5em}\vspace*{-0em}{\includegraphics[width=.15\textwidth]{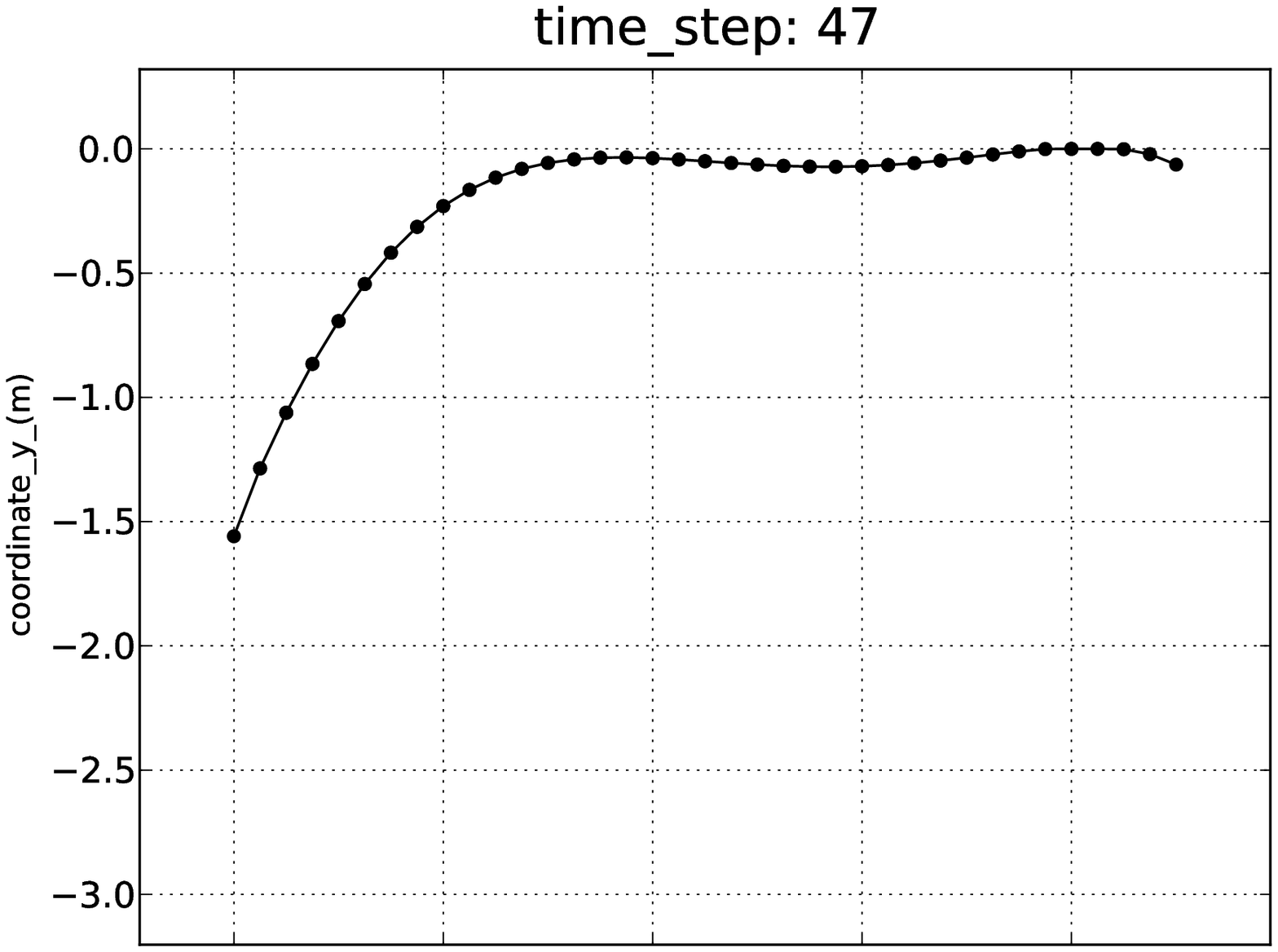}}&
\hspace*{-1.5em}\vspace*{-0em}{\includegraphics[width=.15\textwidth]{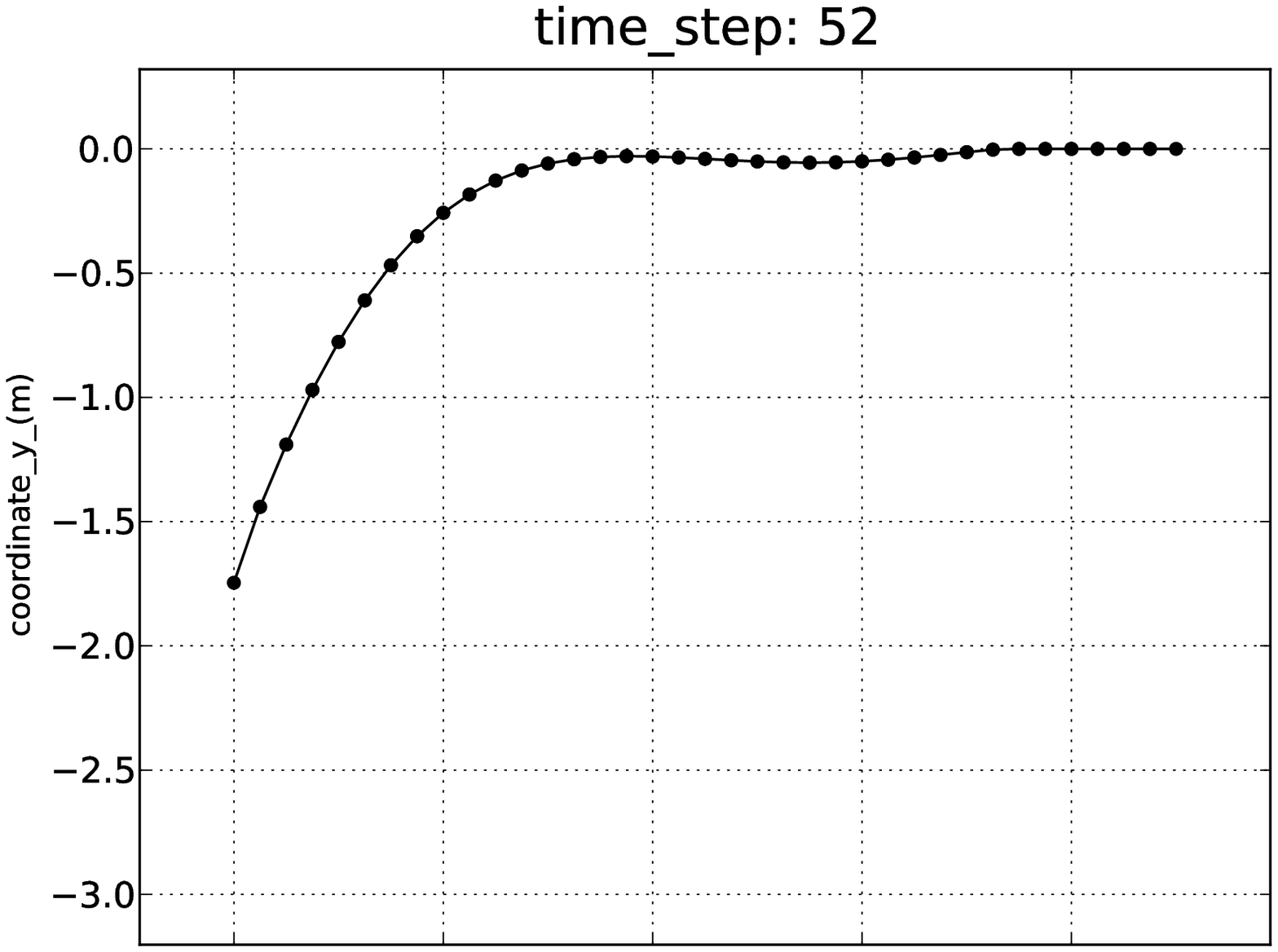}}&
\hspace*{-1.5em}\vspace*{-0em}{\includegraphics[width=.15\textwidth]{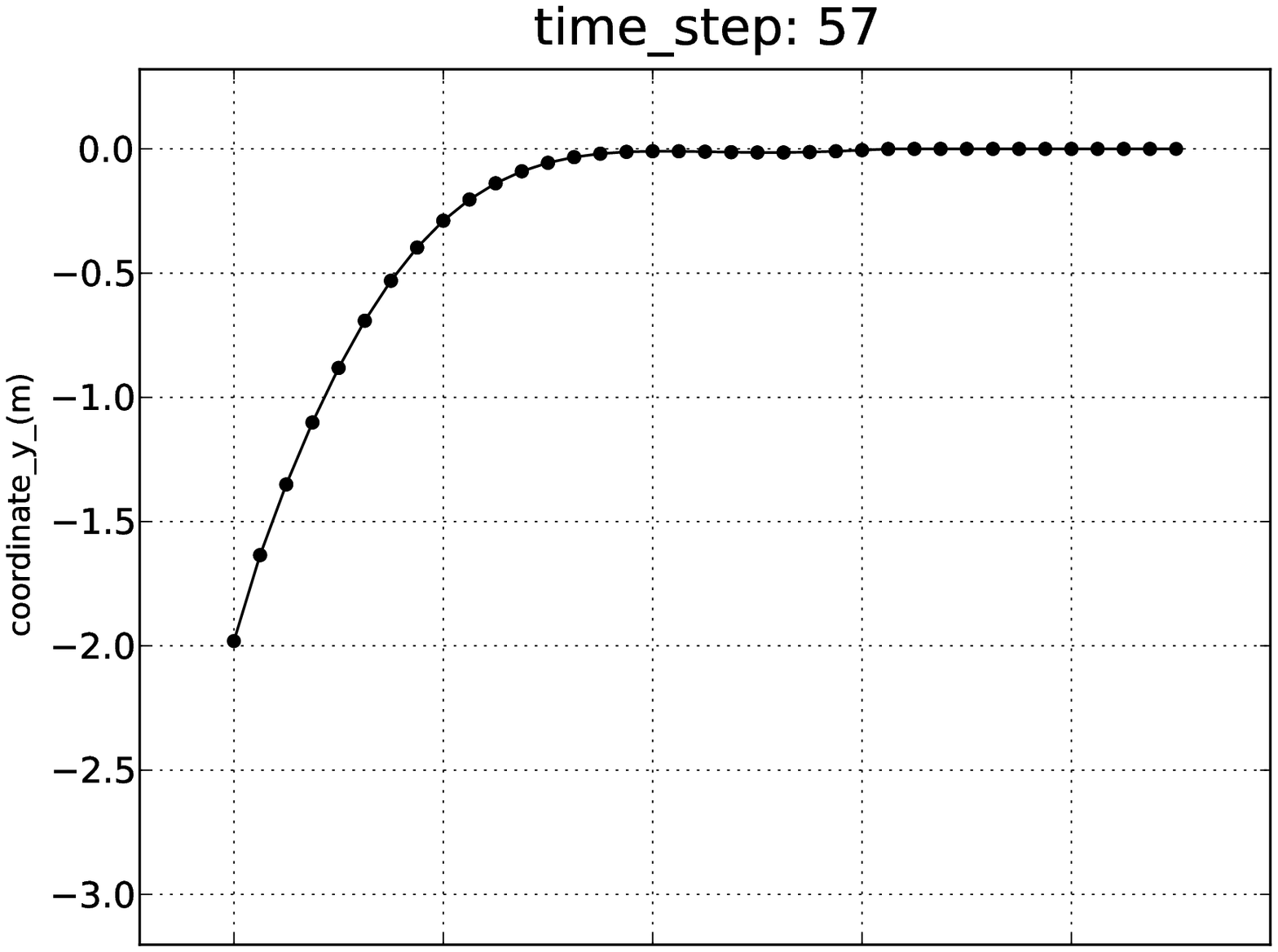}}&
\hspace*{-1.5em}\vspace*{-0em}{\includegraphics[width=.15\textwidth]{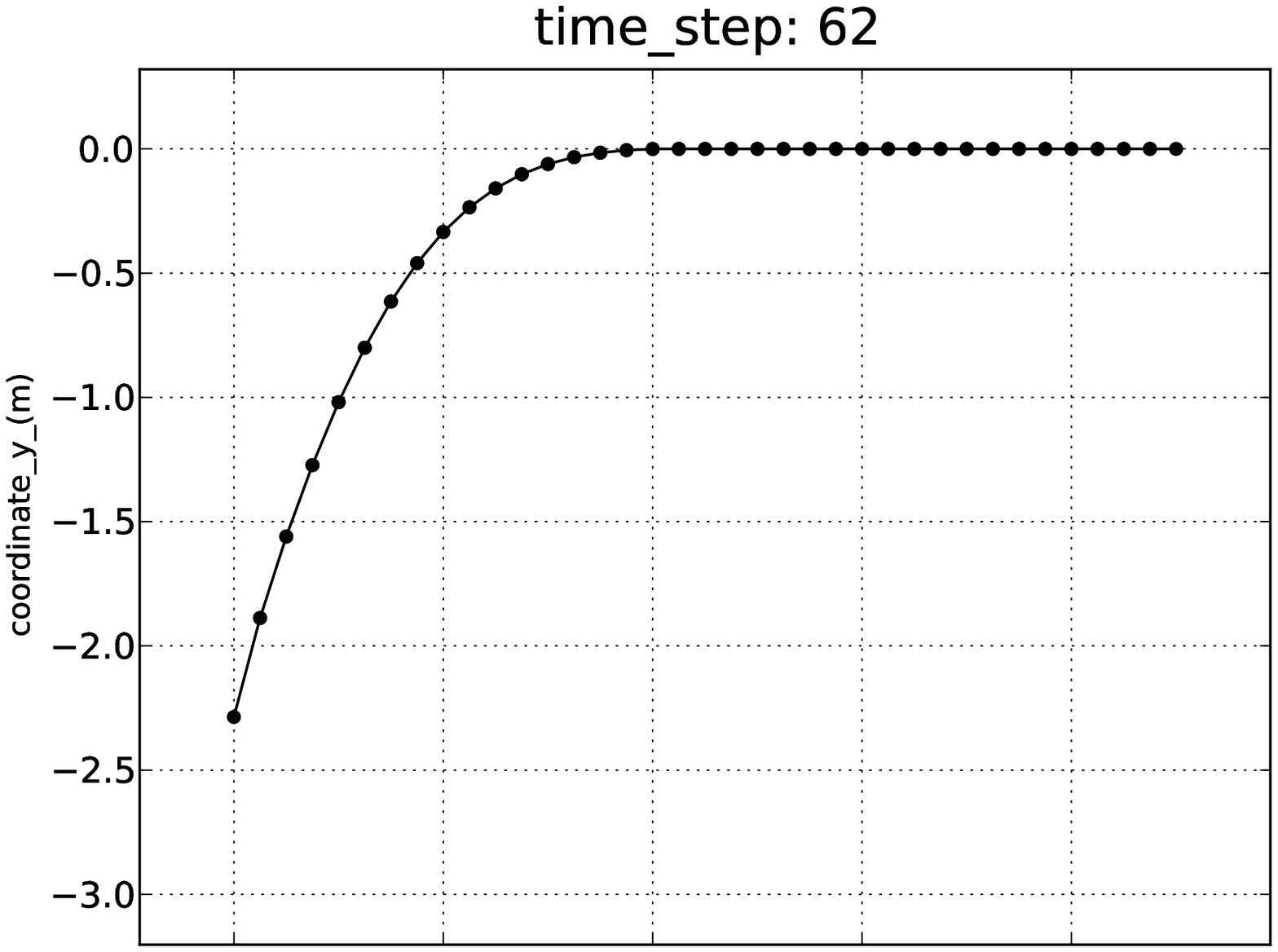}}&
\hspace*{-1.5em}\vspace*{-0em}{\includegraphics[width=.15\textwidth]{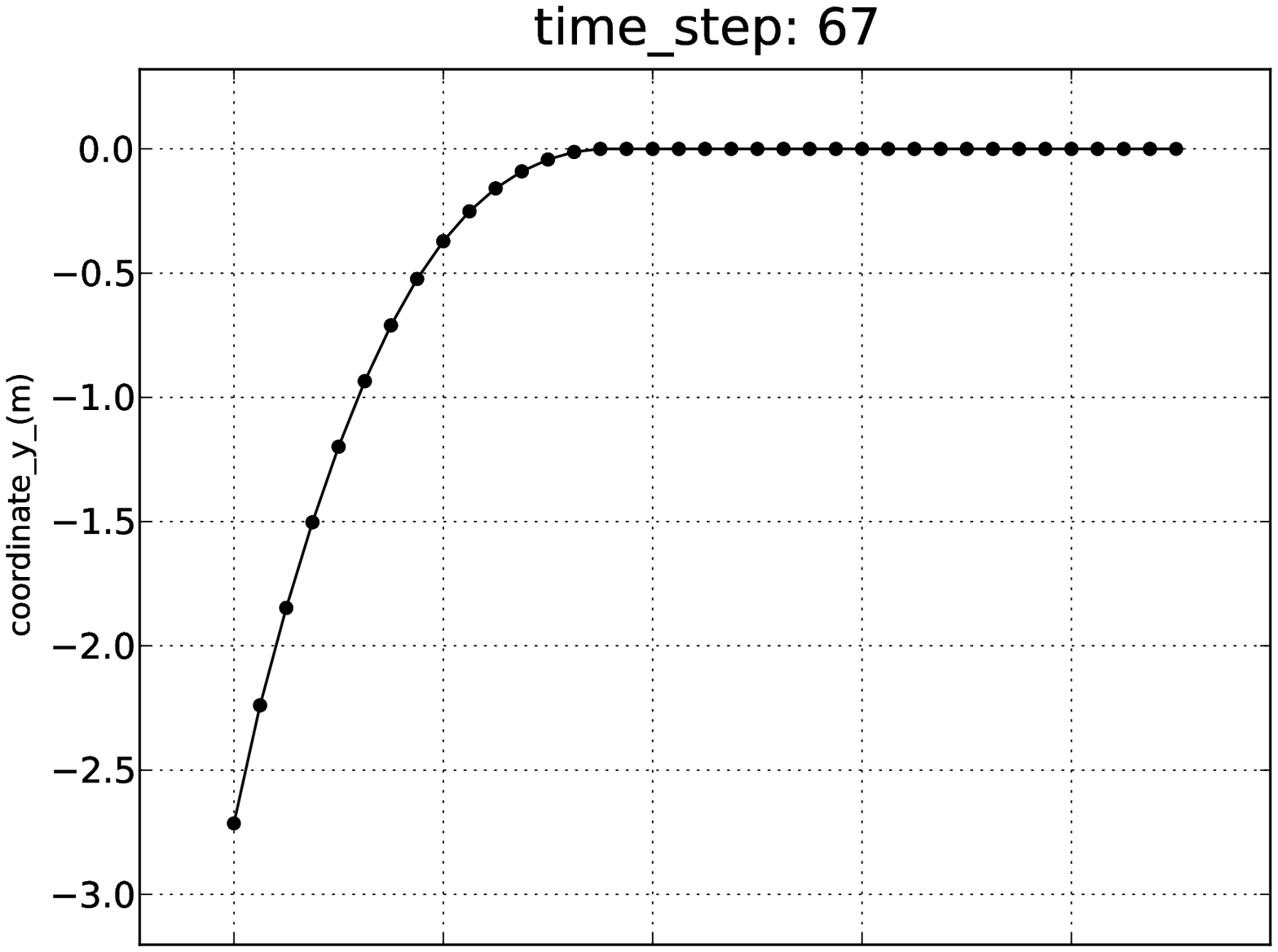}}&
\hspace*{-1.5em}\vspace*{-0em}{\includegraphics[width=.15\textwidth]{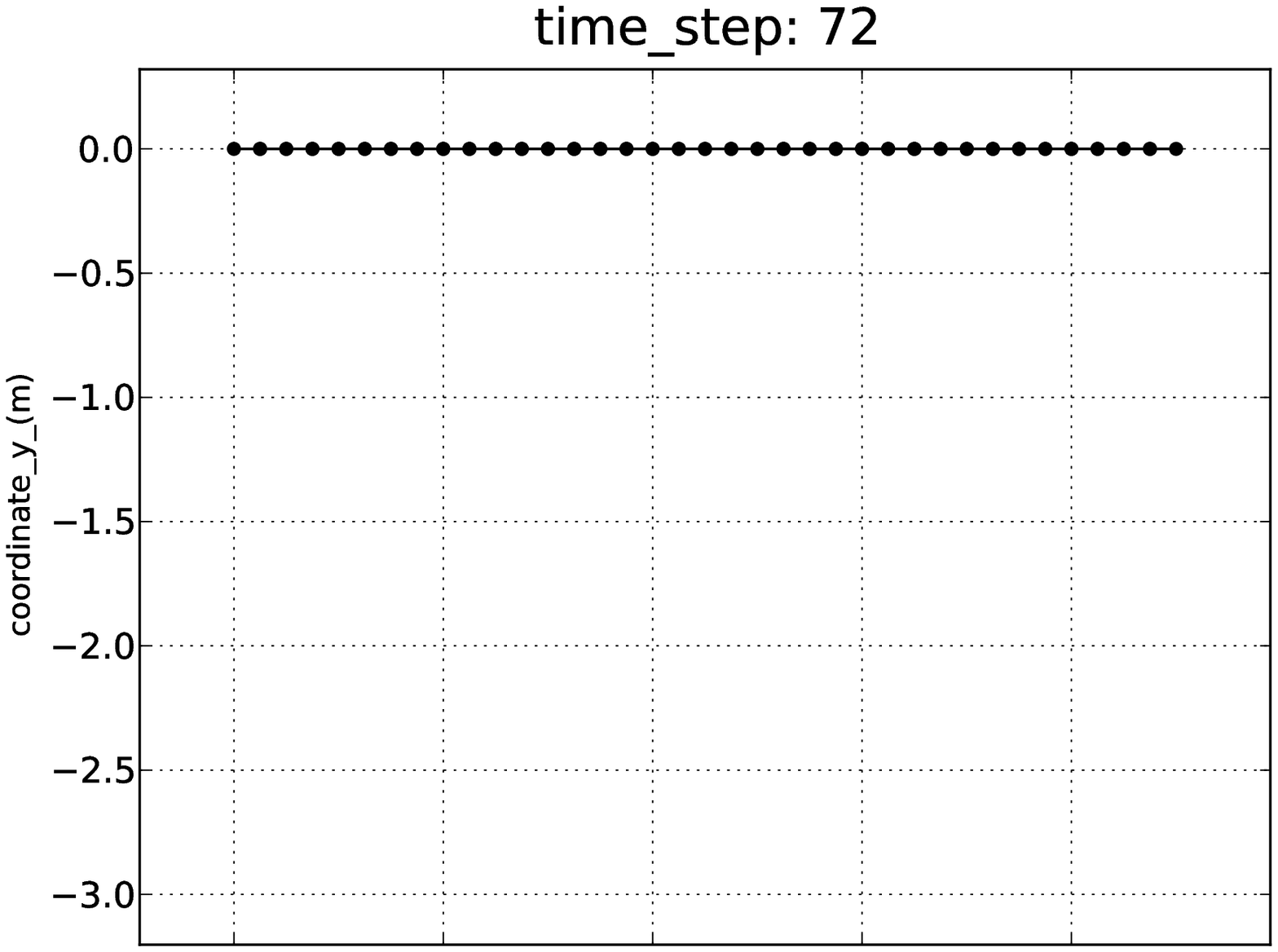}}
\\[.2em]
\hspace*{-1.5em}\vspace*{-0em}{\includegraphics[width=.13\textwidth,height=.07\textwidth]{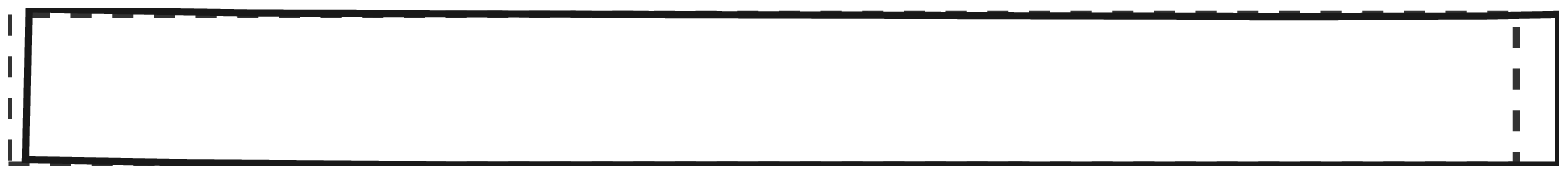}}&
\hspace*{-1.3em}\vspace*{-0em}{\includegraphics[width=.13\textwidth,height=.07\textwidth]{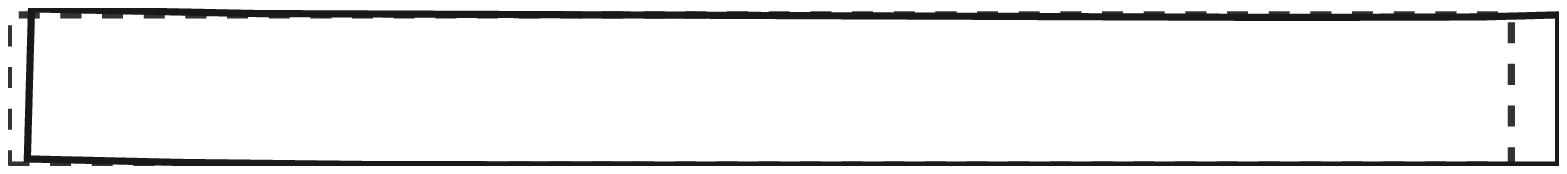}}&
\hspace*{-1.3em}\vspace*{-0em}{\includegraphics[width=.13\textwidth,height=.07\textwidth]{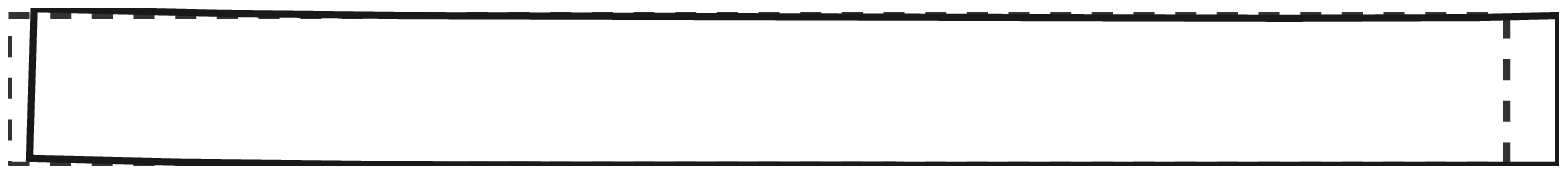}}&
\hspace*{-1.3em}\vspace*{-0em}{\includegraphics[width=.13\textwidth,height=.07\textwidth]{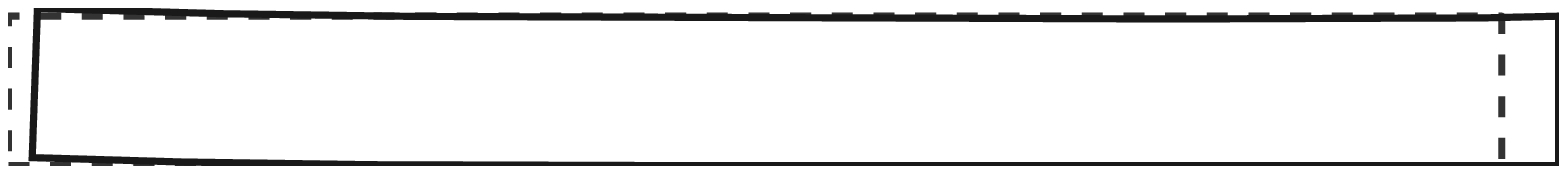}}&
\hspace*{-1.3em}\vspace*{-0em}{\includegraphics[width=.13\textwidth,height=.07\textwidth]{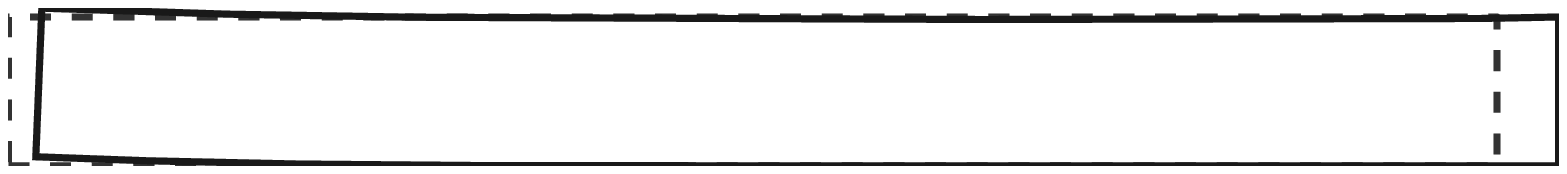}}&
\hspace*{-1.3em}\vspace*{-0em}{\includegraphics[width=.13\textwidth,height=.07\textwidth]{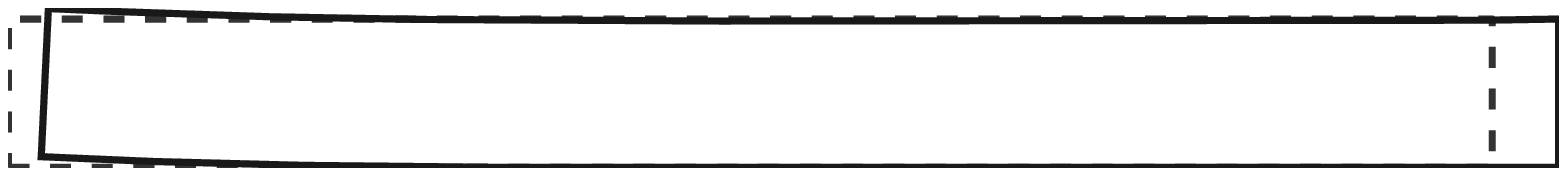}}&
\hspace*{-1.3em}\vspace*{-0em}{\includegraphics[width=.13\textwidth,height=.07\textwidth]{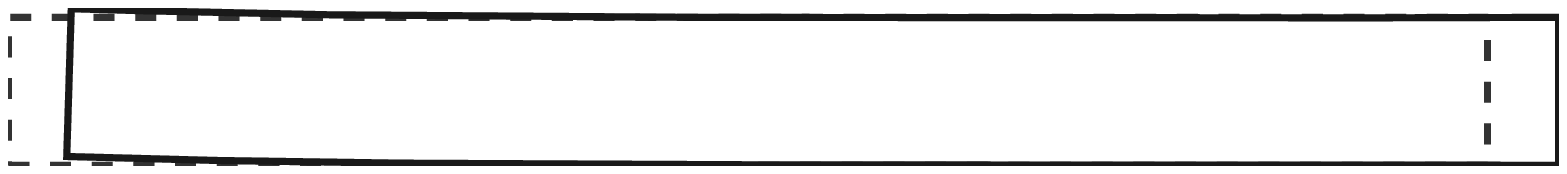}}
\end{tabular}
\end{my-picture}
\\[6em]
\begin{center}
{\small Fig.\,\ref{fig-2D-I-movie}.\ }
\begin{minipage}[t]{.9\textwidth}\baselineskip=8pt
{\small 
Upper row: \,distribution of the tangential traction force in the adhesive
along $\GC$.\\
Middle row: distribution of the normal traction force in the adhesive
along $\GC$.\\
Lower row: \,deformed configuration of gradually delaminating specimen under 
loading (1st experiment) from Fig.\,\ref{fig_m1}; the displacement depicted 
magnified 100$\,\times$ horizontally and 500$\,\times$ vertically
to make the vertical deformation more visible}.
\end{minipage}
\end{center}

\noindent
To present spatial distribution of the defect measure, one must reconstruct 
the strain inside the domain $\Omega$. This is a delicate (but anyhow 
executable) issue in BEM. To visualize the rate of viscous dissipation (which 
approximates the defect measure $\mu$, cf.~\eqref{ch6:delam-slow-conv-c}, and 
may exhibit time oscillations which would make visualization difficult), we 
display rather the overall dissipation on the interval $[0,t]$, cf.\ 
Fig.\,\ref{fig-2D-I-movie-2}, which approximates the total variation of the 
defect measure $\int_0^t[\mu(\cdot,x)](\d t)$ as a function of $x$: 

\begin{my-picture}{.7}{.55}{fig-2D-I-movie-2}
\psfrag{energy density rate, step: 6}{}
\psfrag{energy density rate, step: 9}{}
\psfrag{energy density rate, step: 12}{}
\psfrag{energy density rate, step: 15}{}
\psfrag{energy density rate, step: 18}{}
\psfrag{energy density rate, step: 21}{}
\begin{tabular}{lc} 
\hspace*{3em}\LARGE $^{^{^{^{\mbox{\footnotesize $t$=0.21}}}}}$  & 
\hspace*{1em}{\includegraphics[clip=true,width=.6\textwidth]{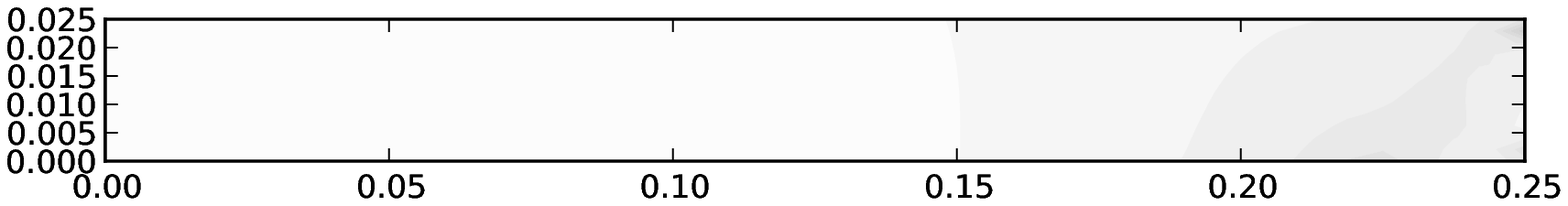}} 
  \\[-1.25em]
\hspace*{3em}\LARGE $^{^{^{^{\mbox{\footnotesize $t$=0.235}}}}}$ &
\hspace*{1em}{\includegraphics[clip=true,width=.6\textwidth]{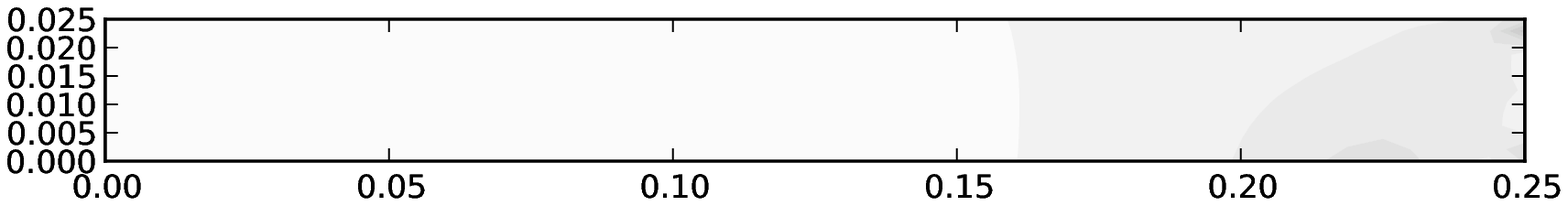}} 
\\[-1.25em]
\hspace*{3em}\LARGE $^{^{^{^{\mbox{\footnotesize $t$=0.26}}}}}$ &
\hspace*{1em}{\includegraphics[clip=true,width=.6\textwidth]{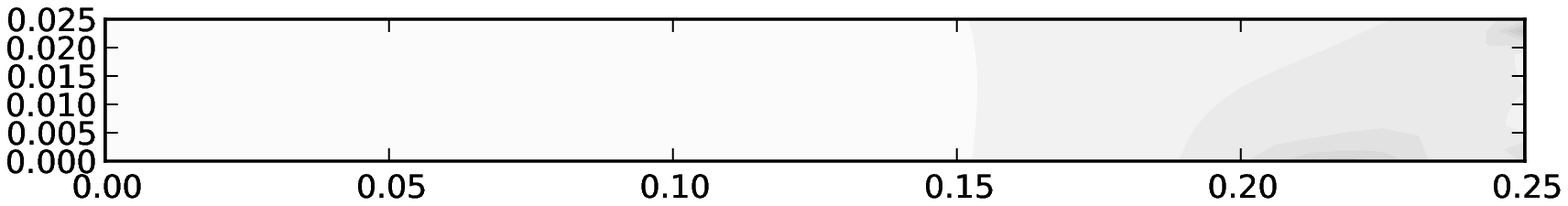}} 
\\[-1.25em]
\hspace*{3em}\LARGE $^{^{^{^{\mbox{\footnotesize $t$=0.285}}}}}$ &
\hspace*{1em}{\includegraphics[clip=true,width=.6\textwidth]{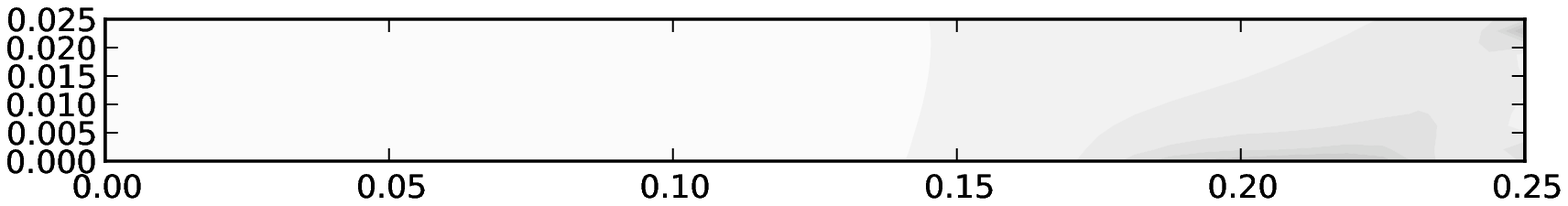}} 
\\[-1.25em]
\hspace*{3em}\LARGE $^{^{^{^{\mbox{\footnotesize $t$=0.31}}}}}$ &
\hspace*{1em}{\includegraphics[clip=true,width=.6\textwidth]{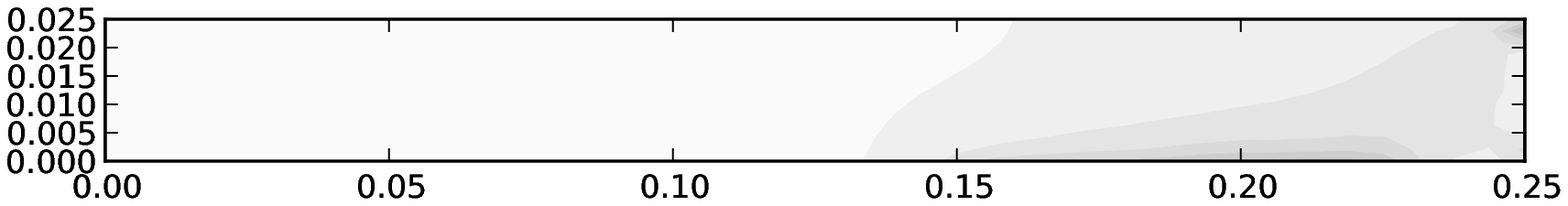}} 
\\[-1.25em]
\hspace*{3em}\LARGE $^{^{^{^{\mbox{\footnotesize $t$=0.335}}}}}$ &
\hspace*{1em}{\includegraphics[clip=true,width=.6\textwidth]{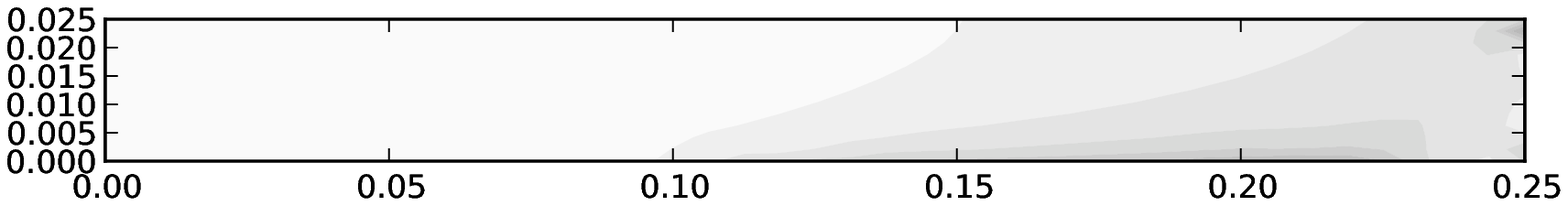}} 
\\[-1.25em]
\hspace*{3em}\LARGE $^{^{^{^{\mbox{\footnotesize $t$=0.36}}}}}$ &
\hspace*{1em}{\includegraphics[clip=true,width=.6\textwidth]{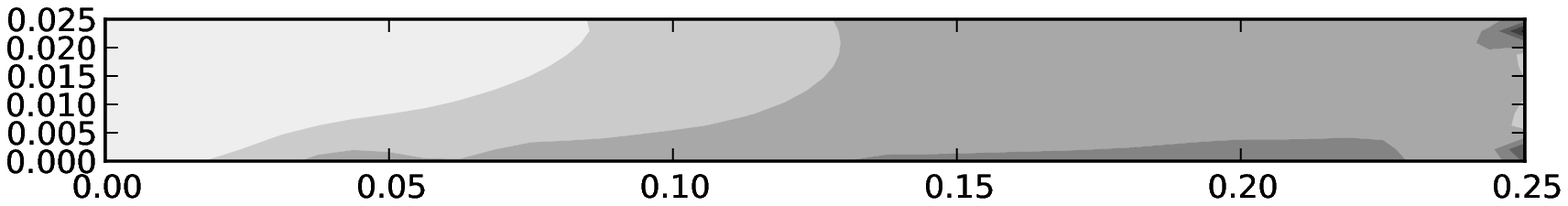}}
\\[-.5em]&\hspace*{0em}{\includegraphics[clip=true,width=.35\textwidth]{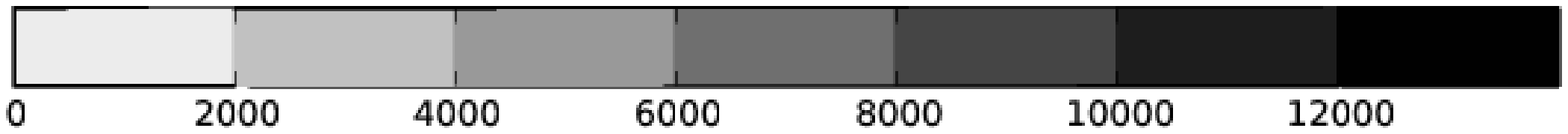}} 
\\[0em]
\end{tabular}
\end{my-picture}
\\[.0em]
\nopagebreak
\begin{center}
{\small Fig.\,\ref{fig-2D-I-movie-2}.\ }
\begin{minipage}[t]{.8\textwidth}\baselineskip=8pt
{\small The spatial distribution of the energy dissipated by (even very small) 
viscosity over the time interval $[0,t]$, i.e.\ 
$\int_0^t\chi\bbC e(\DT u_{\chi,\tau}){:}e(\DT u_{\chi,\tau})\,\d t$ 
depicted in a gray scale at 6 selected time instances as also used on 
Fig.\,\ref{fig-2D-I-movie}.}
\end{minipage}
\end{center}

\noindent
It is not surprising that the dissipated viscous energy is bigger in the 
right-hand part of the specimen which is particularly stretched during the 
delamination. Perhaps noteworthy phenomenon is that this energy is not 
localized along the delaminating surface; we saw this effect already in 
the example in Section~\ref{sec-0D} where it was distributed over the whole 
volume uniformly.

An analog of Figure~\ref{fig:1-D-stress-strain}(right) displaying the force 
response $t\mapsto\int_{\GDir}\!\mathfrak{t}(\epsilon(u,\DT u))(t,x)\,\d S$ is 
on Figure~\ref{fig-2D-engr}(left) together with a comparison with results 
obtained by the simplified inviscid algorithm from Remark~\ref{rem-inviscid}:

\begin{my-picture}{.9}{.25}{fig-2D-engr}
\psfrag{time}{\footnotesize\hspace*{3.8em}time\hspace{2em}$t$}
\psfrag{resultant_force}{\footnotesize\hspace{1em}reaction force}
\psfrag{vertical}{\tiny\hspace{-.6em}vertical}
\psfrag{horizontal}{\tiny\hspace{-.6em}horizon.}
\hspace*{3em}\vspace*{-.1em}{\includegraphics[width=.4\textwidth,height=.25\textwidth]{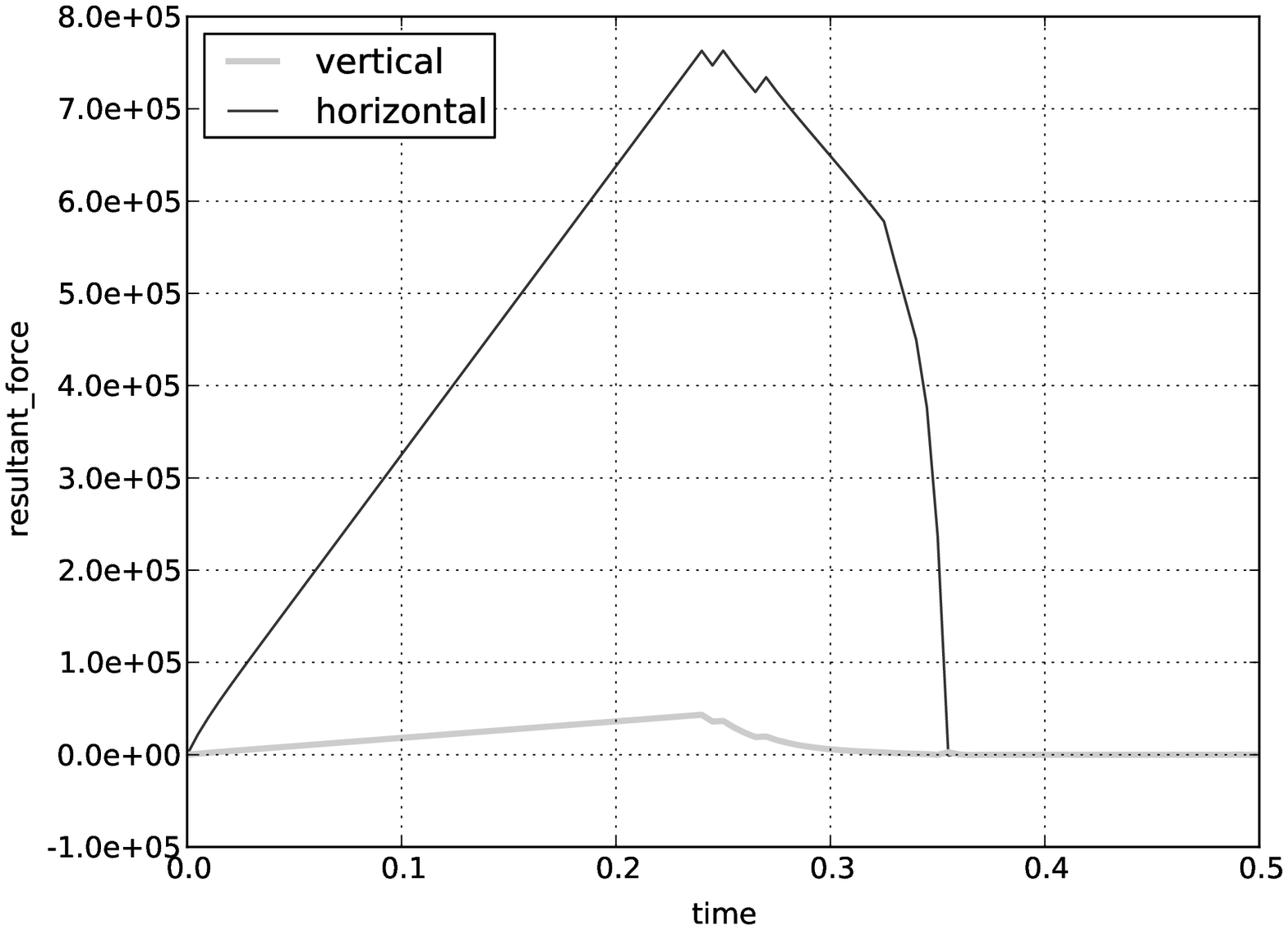}}
\hspace*{3em}\vspace*{-.1em}{\includegraphics[width=.4\textwidth,height=.25\textwidth]{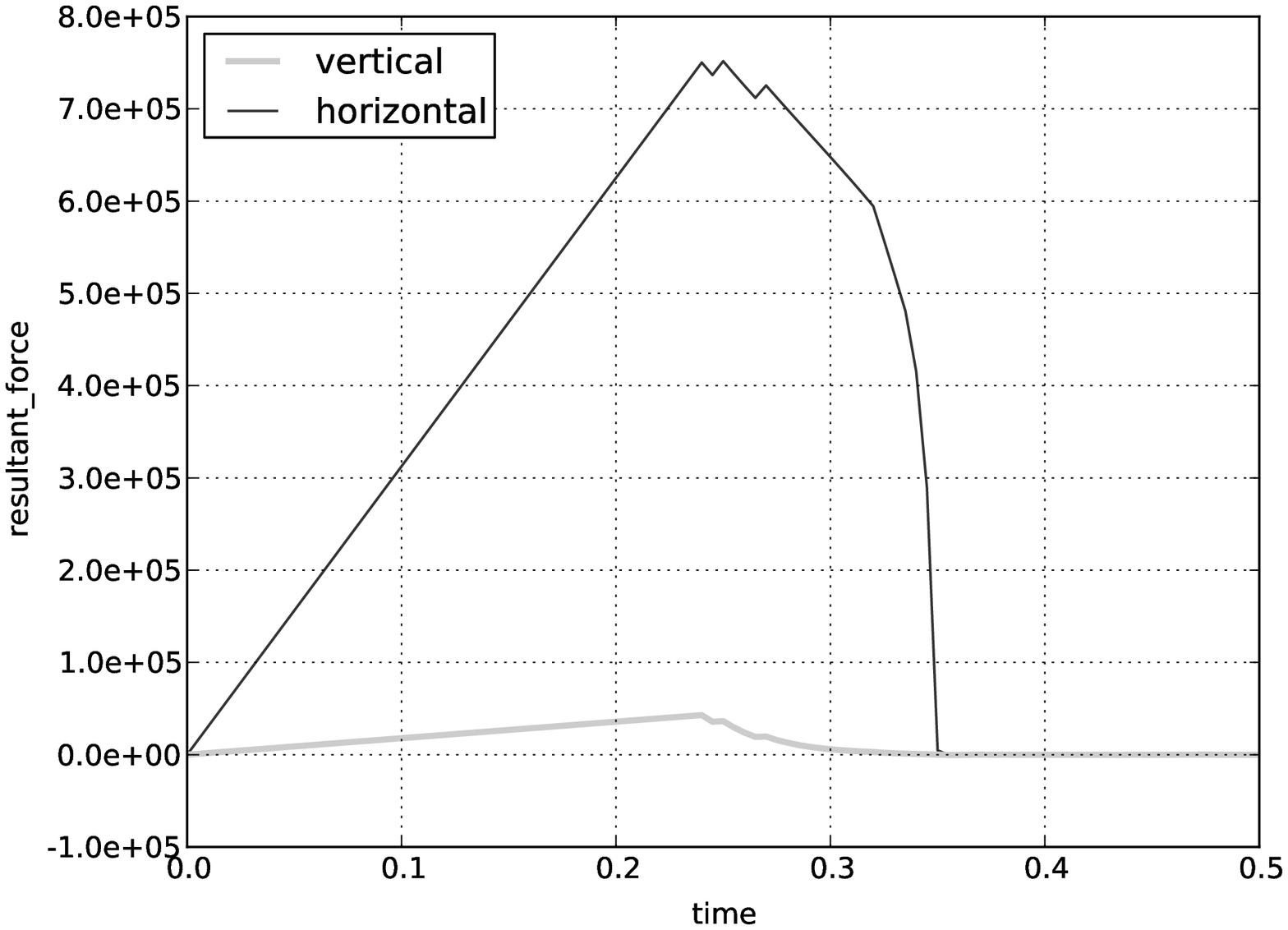}}
\end{my-picture}
\\[-2em]
\nopagebreak
\begin{center}
{\small Fig.\,\ref{fig-2D-engr}.\ }
\begin{minipage}[t]{.9\textwidth}\baselineskip=8pt
{\small
Vertical and horizontal components of the reaction force on the Dirichlet 
loading for small viscosity $\chi=0.01$s and energy well preserved (left), 
compared with the inviscid solution calculated by semi-implicit method but 
energy balance completely violated as in Remark~\ref{rem-inviscid} (right). 
As on Fig.~\ref{fig:1-D-violation}, a surprisingly good match of this force 
response can be observed.
}
\end{minipage}
\end{center}

\subsubsection{Vertical-loading experiment}

Eventually, we briefly present most of the responses from 
Section~\ref{sect-2D-1st-exp} for another loading as indicated on 
Figure~\ref{fig_m1}. Of course, the response is considerably
different in some aspects, although the phenomena commented already for the 
horizontal loading are again observed. Now, the delamination propagates more 
slowly and we depict it with an equidistant step $0.45\,$s (instead of 
$0.025\,$s used in the 1st loading experiment) starting from $t=0.05\,$s:
 
\begin{my-picture}{.9}{.2}{fig-movie-1}
\psfrag{time_step: 10}{}
\psfrag{time_step: 58}{}
\psfrag{time_step: 106}{}
\psfrag{time_step: 154}{}
\psfrag{time_step: 202}{}
\psfrag{time_step: 250}{}
\psfrag{time_step: 298}{}
\psfrag{normal_traction_x_(m)}{\tiny \hspace{-2em}$_{x\mbox{-coordinate}}$}
\psfrag{tangent_traction_x_(m)}{\tiny \hspace{-2em}$_{x\mbox{-coordinate}}$}
\psfrag{coordinate_y_(m)}{\tiny }
\psfrag{4}{\small 4}
\begin{tabular}{ccccccc} 
\hspace*{-1.5em}\footnotesize $t=0.05$ & 
\hspace*{-1.5em}\footnotesize $t=0.5$ & 
\hspace*{-1.5em}\footnotesize $t=0.95$ & 
\hspace*{-1.5em}\footnotesize $t=1.4$& 
\hspace*{-1.5em}\footnotesize $t=1.85$& 
\hspace*{-1.5em}\footnotesize $t=2.3$& 
\hspace*{-1.5em}\footnotesize $t=2.75$\\[-0em]
\hspace*{-2.em}\includegraphics[width=.15\textwidth]{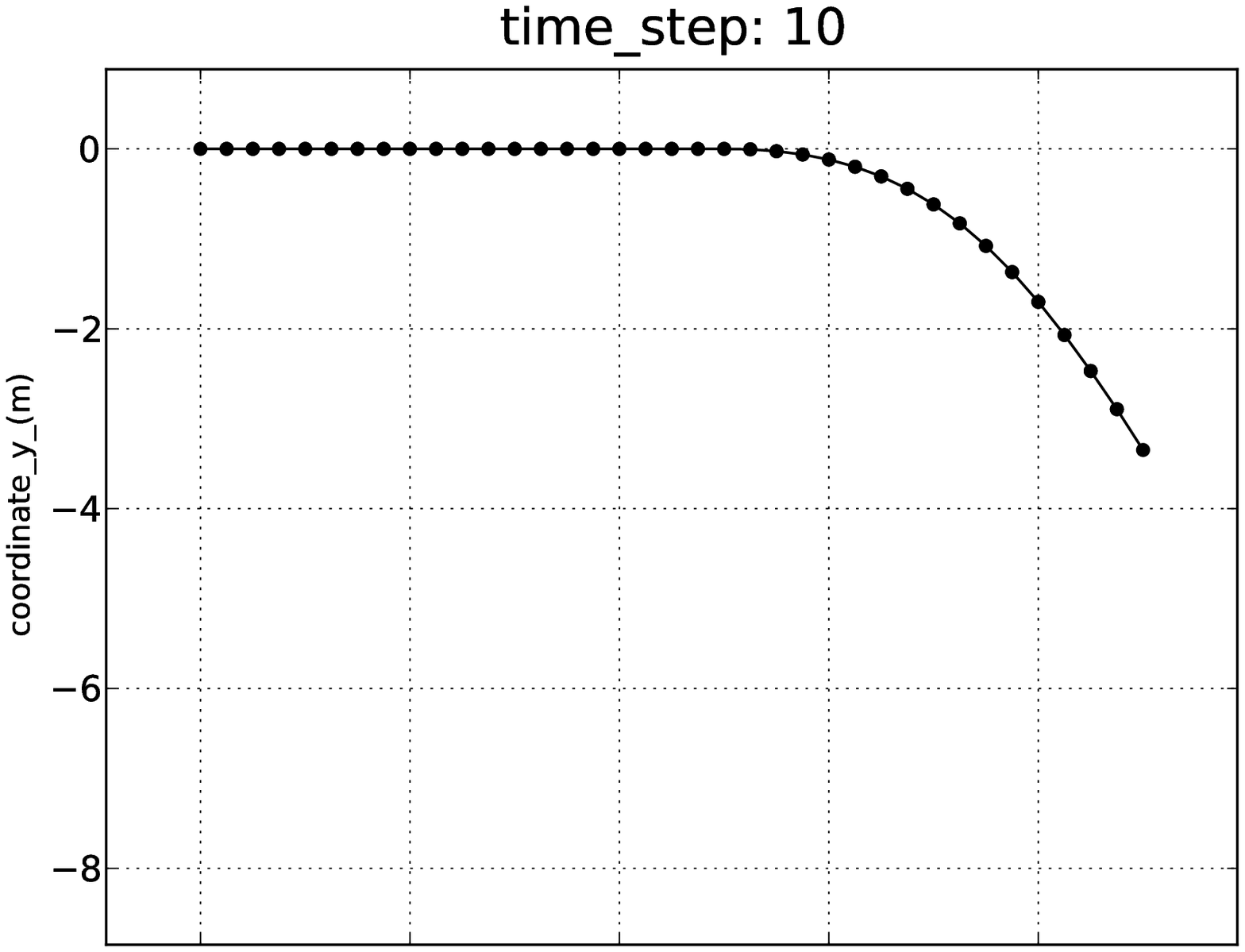}&
\hspace*{-1.5em}\vspace*{-0em}{\includegraphics[width=.15\textwidth]{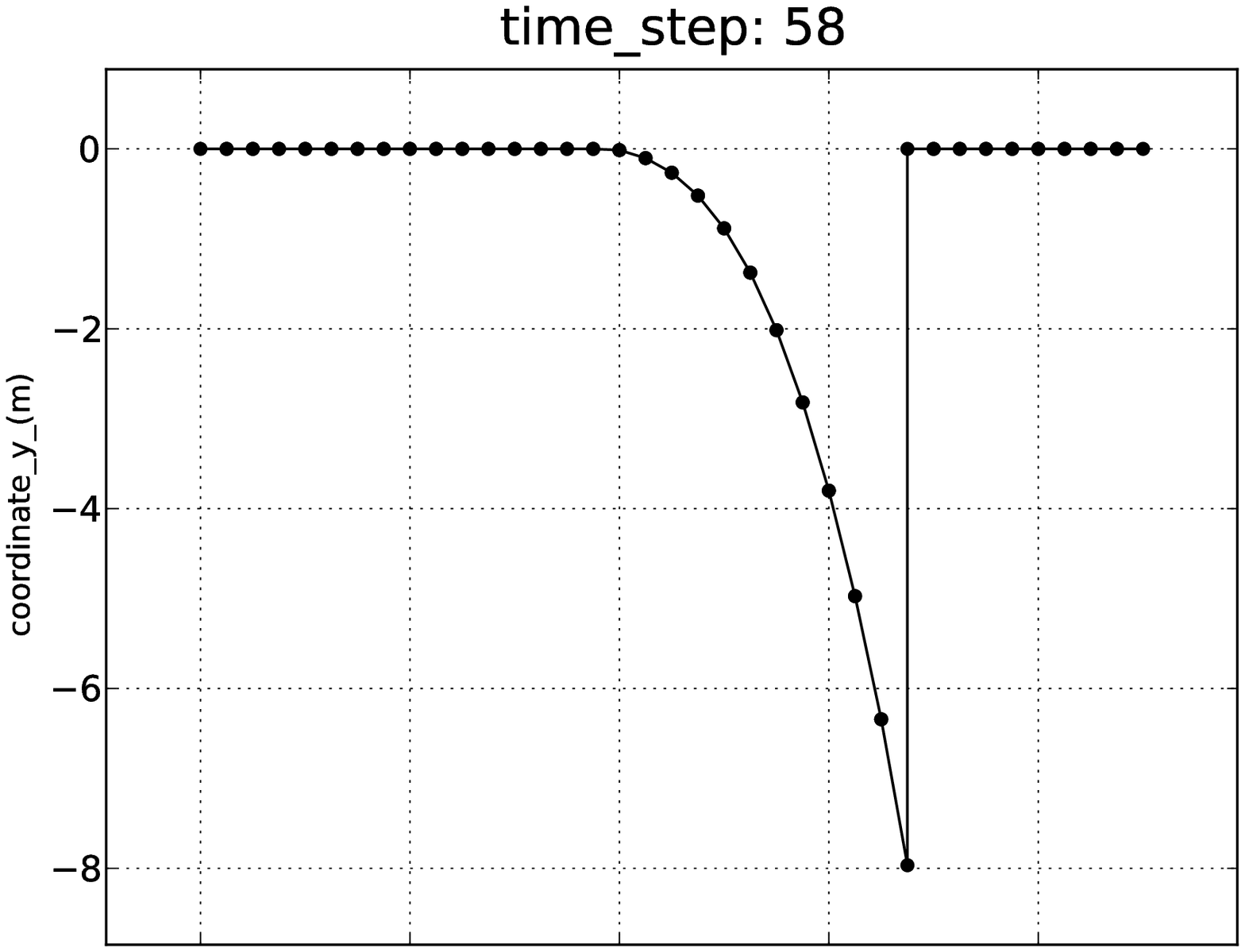}}&
\hspace*{-1.5em}\vspace*{-0em}{\includegraphics[width=.15\textwidth]{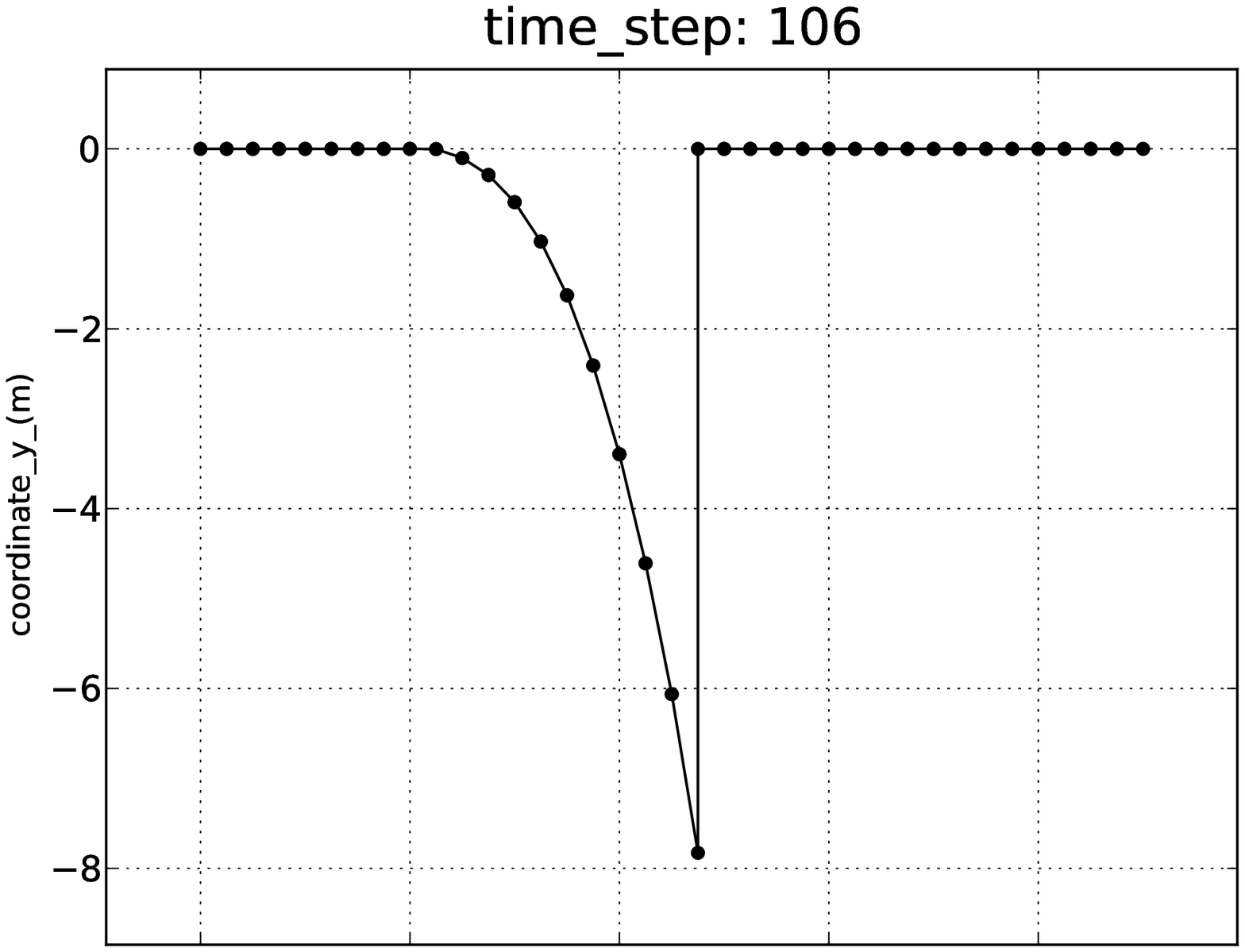}}&
\hspace*{-1.5em}\vspace*{-0em}{\includegraphics[width=.15\textwidth]{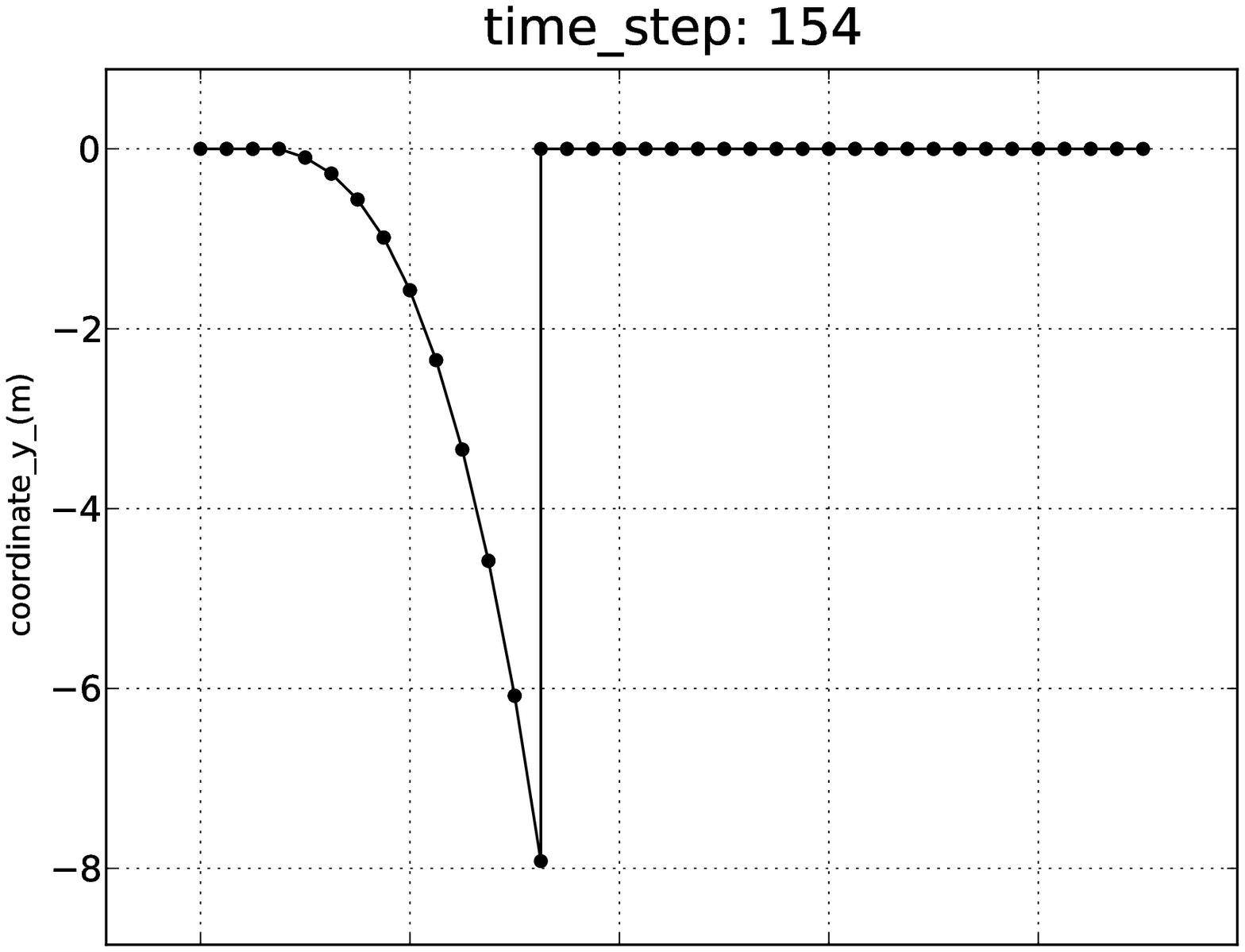}}&
\hspace*{-1.5em}\vspace*{-0em}{\includegraphics[width=.15\textwidth]{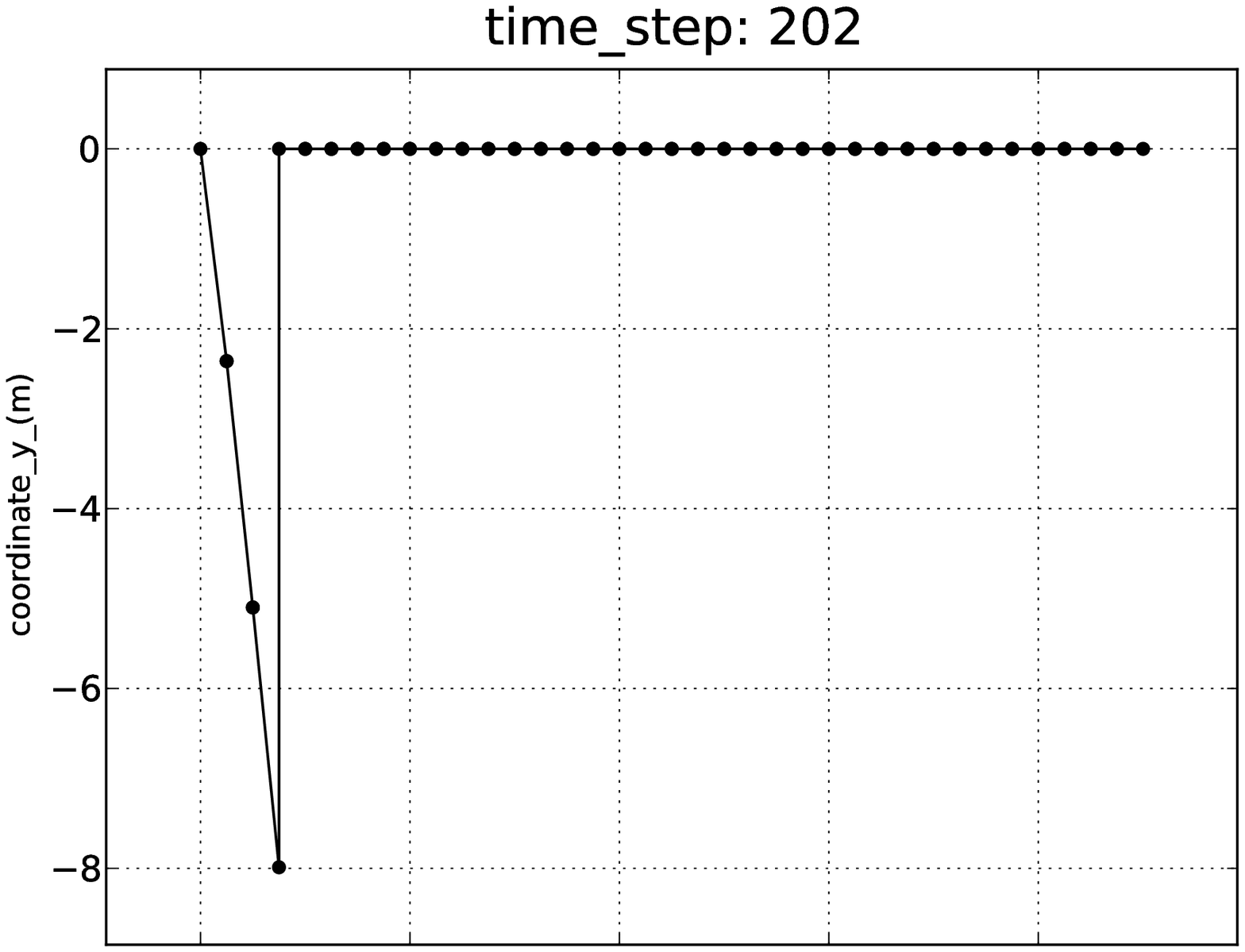}}&
\hspace*{-1.5em}\vspace*{-0em}{\includegraphics[width=.15\textwidth]{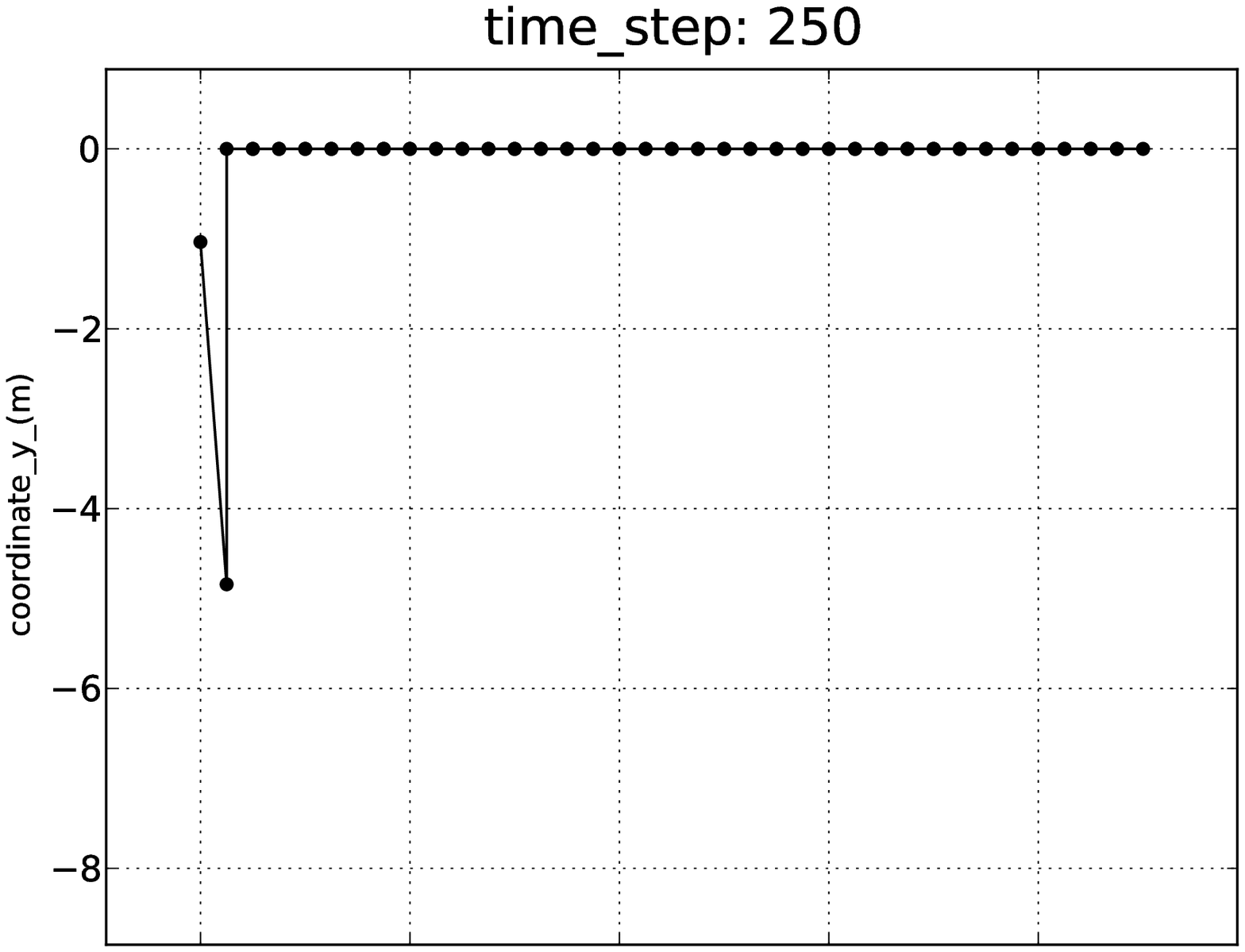}}&
\hspace*{-1.5em}\vspace*{-0em}{\includegraphics[width=.15\textwidth]{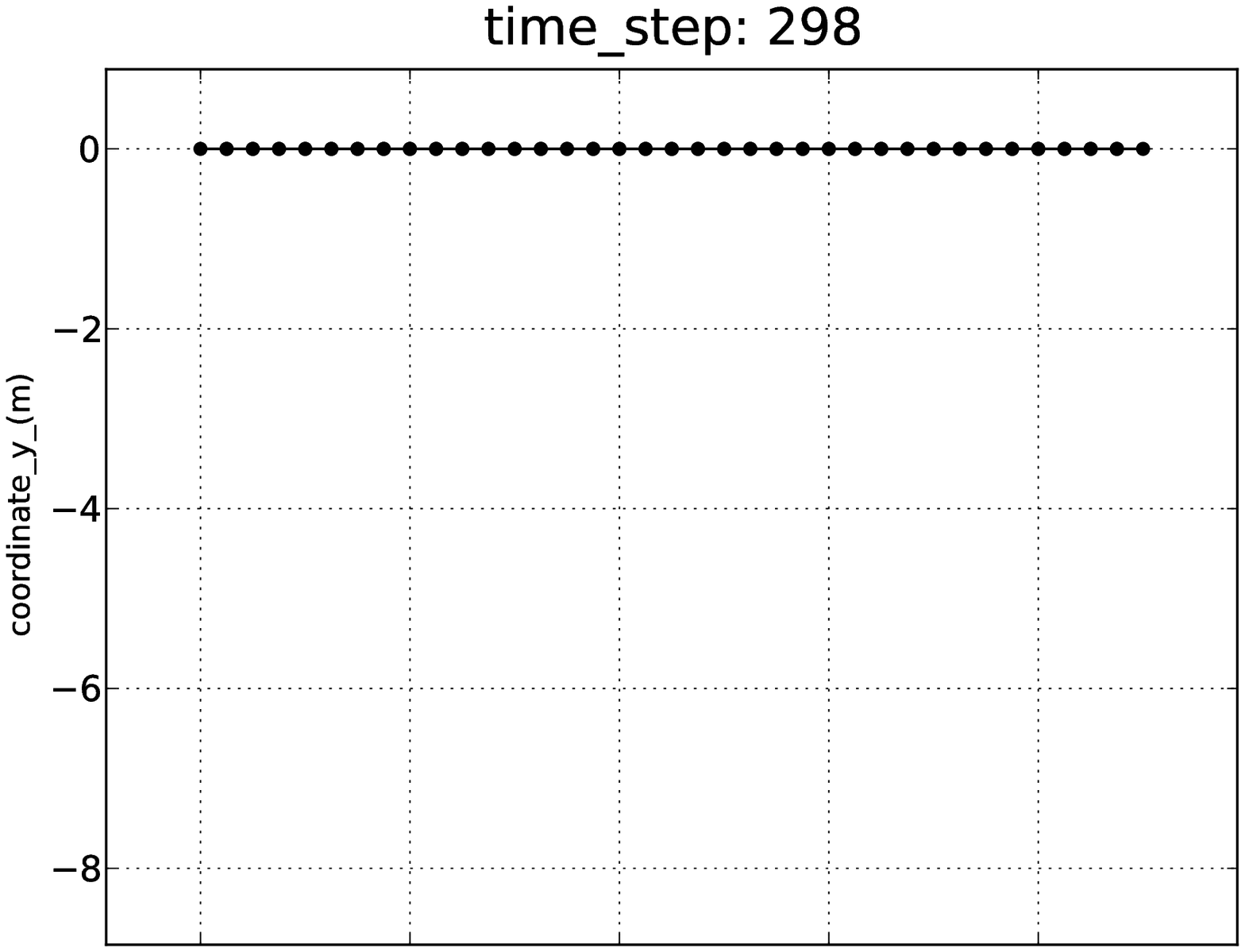}}
\\[-0em]
\hspace*{-2.em}\includegraphics[width=.15\textwidth]{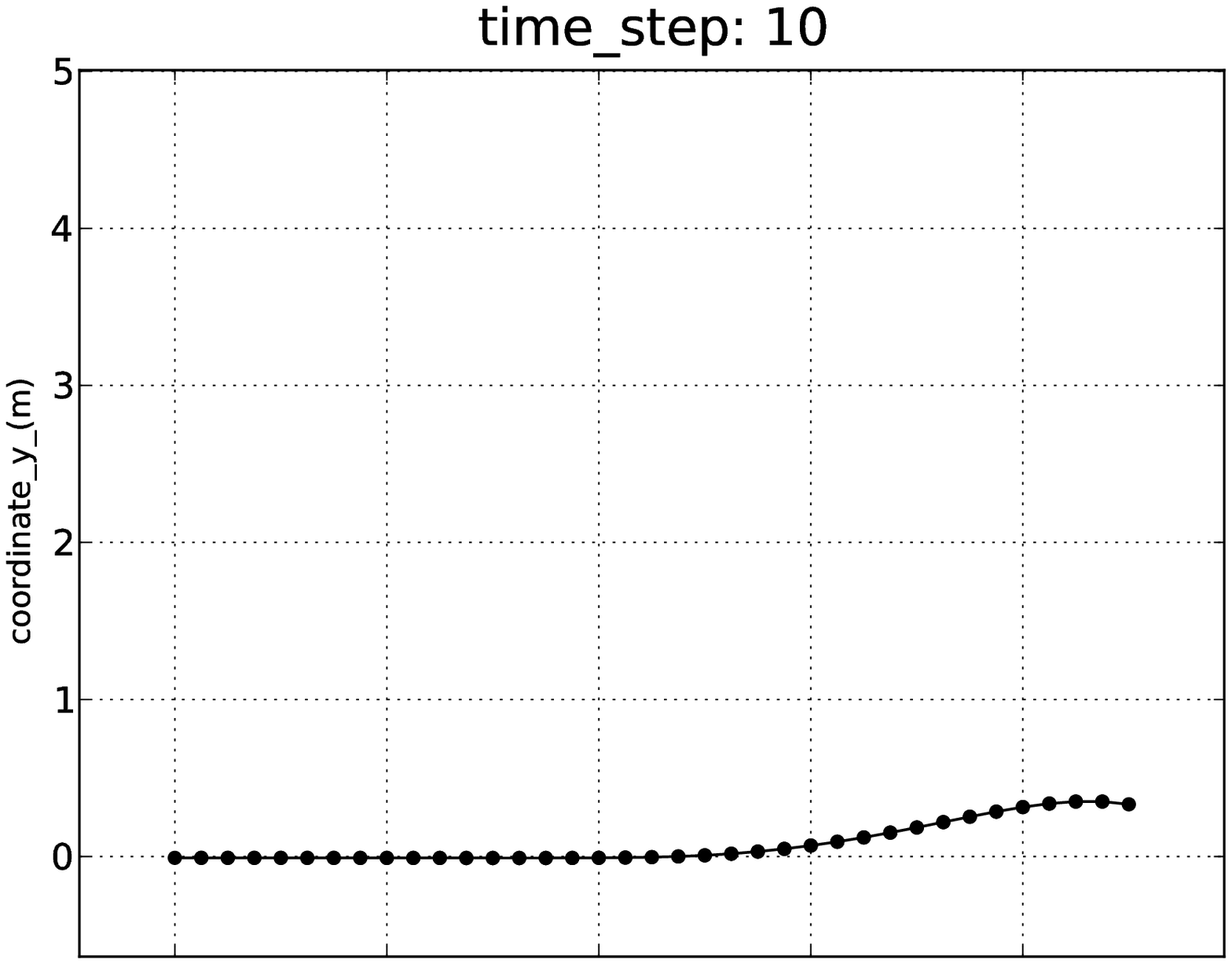}&
\hspace*{-1.5em}\vspace*{-0em}{\includegraphics[width=.15\textwidth]{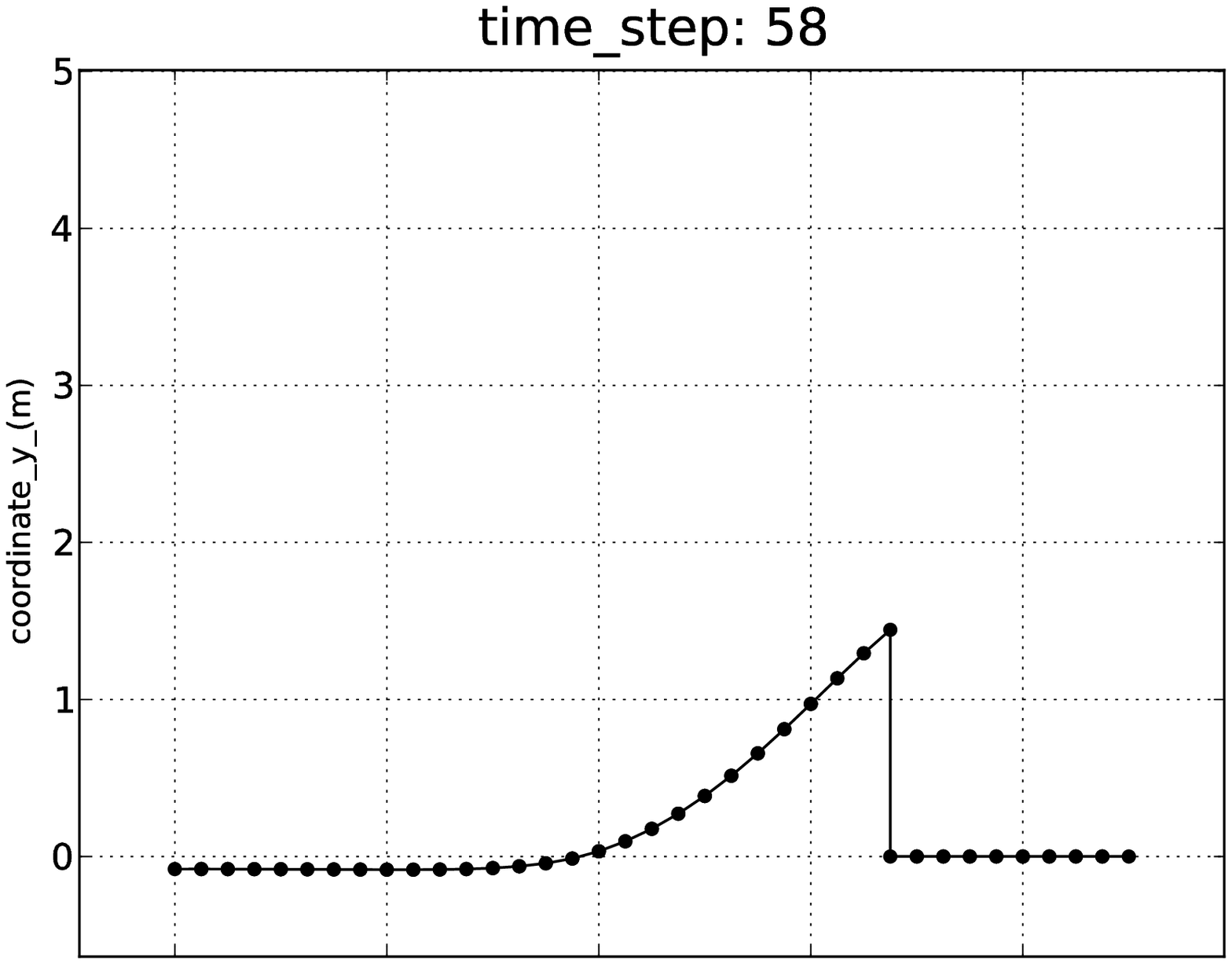}}&
\hspace*{-1.5em}\vspace*{-0em}{\includegraphics[width=.15\textwidth]{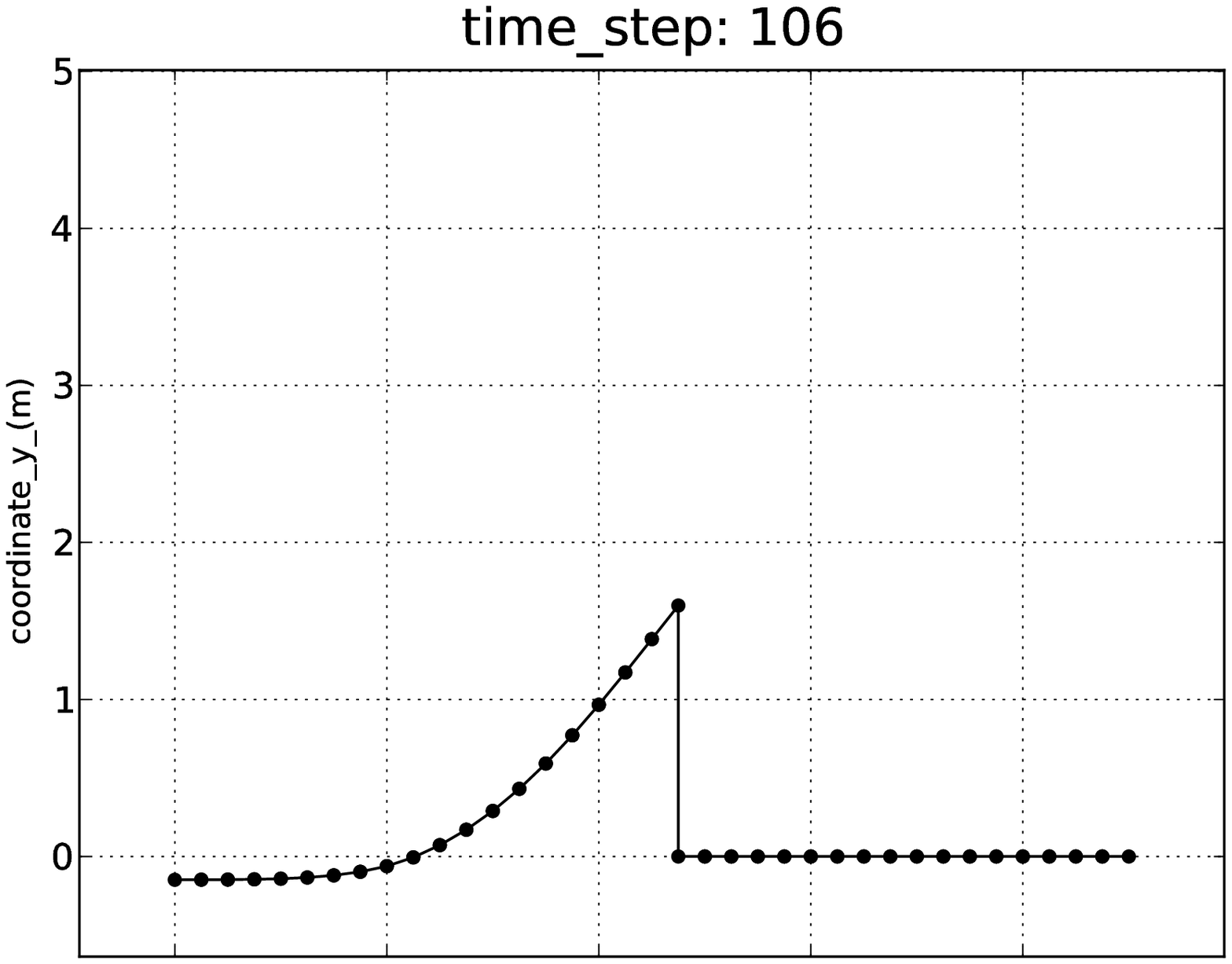}}&
\hspace*{-1.5em}\vspace*{-0em}{\includegraphics[width=.15\textwidth]{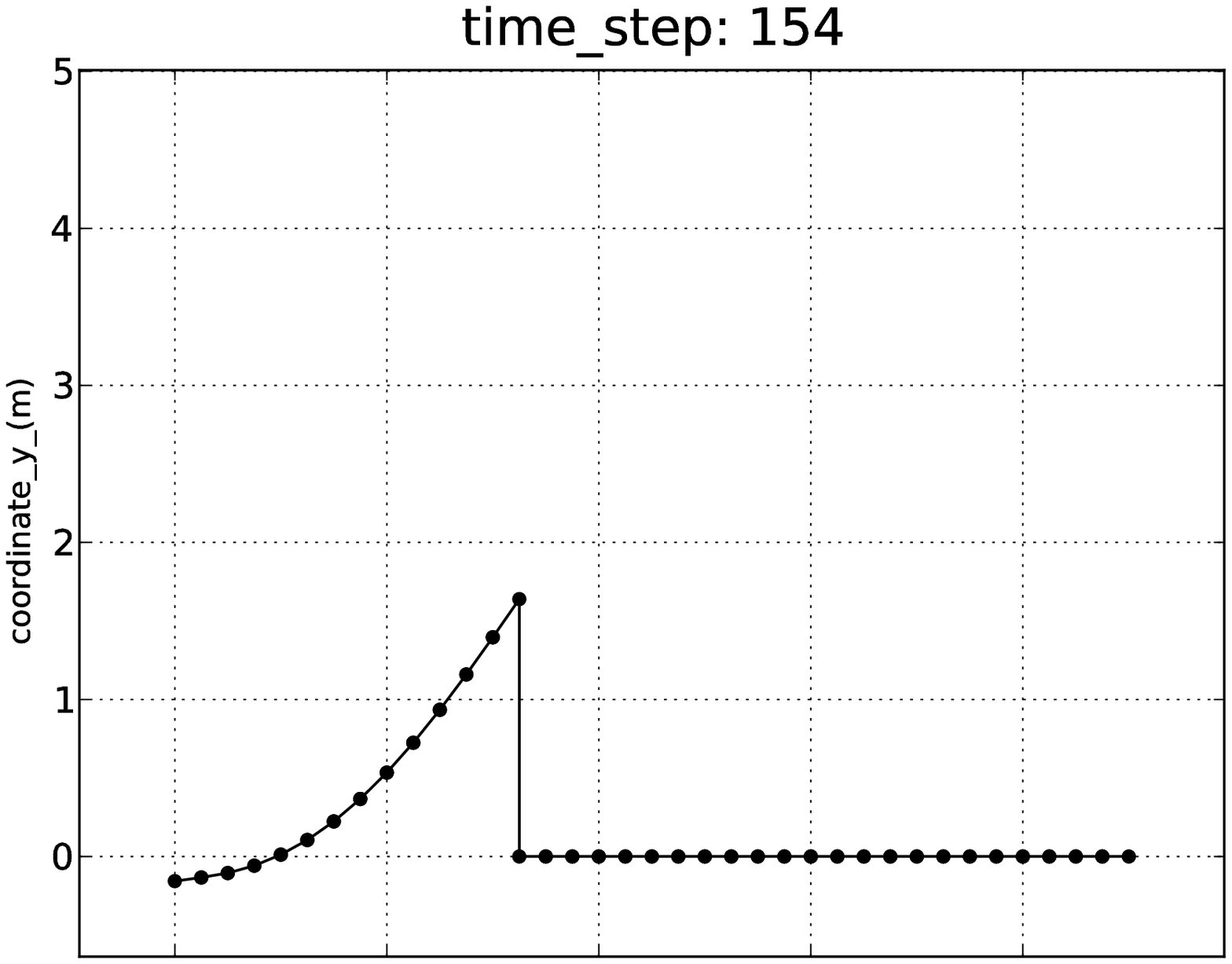}}&
\hspace*{-1.5em}\vspace*{-0em}{\includegraphics[width=.15\textwidth]{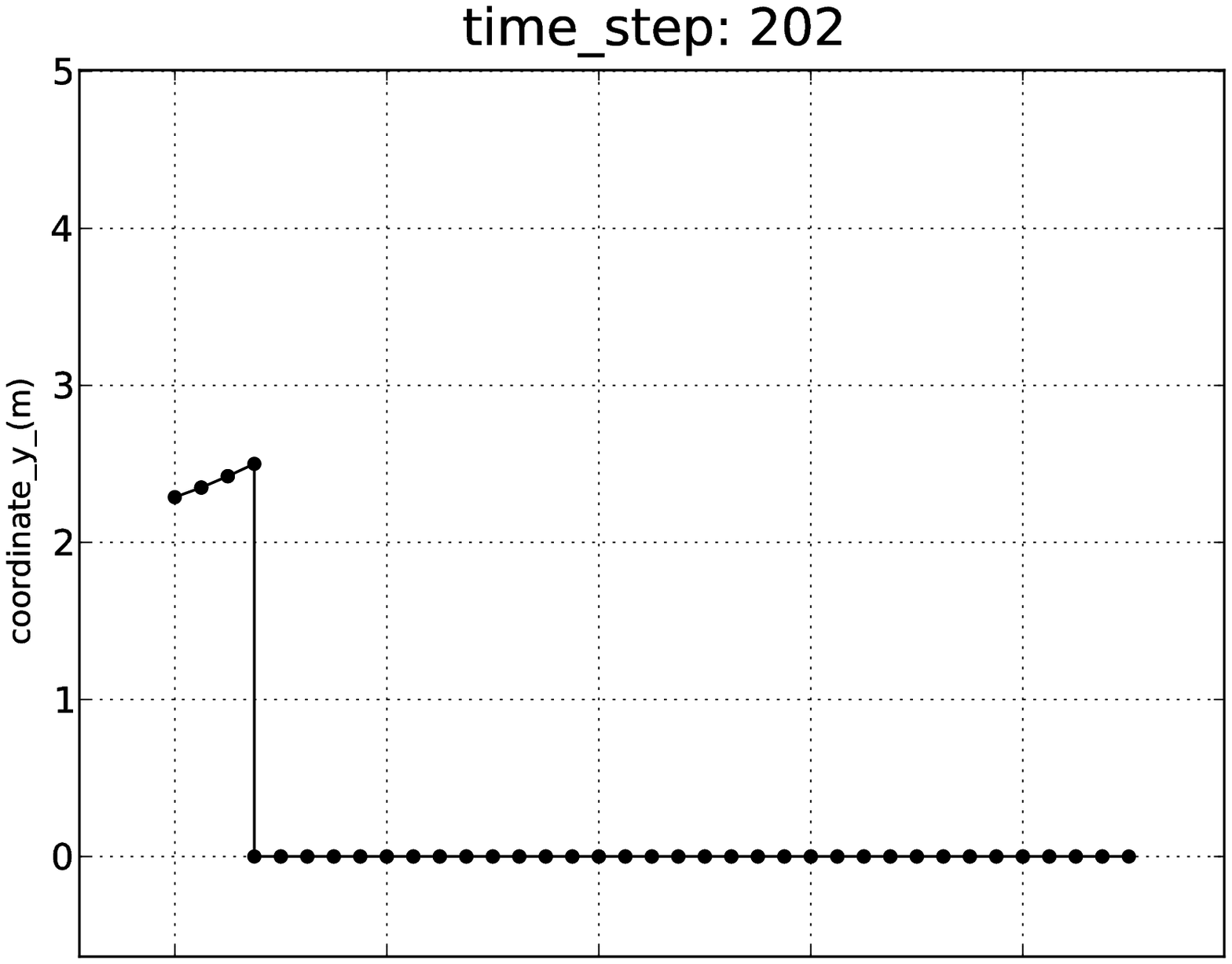}}&
\hspace*{-1.5em}\vspace*{-0em}{\includegraphics[width=.15\textwidth]{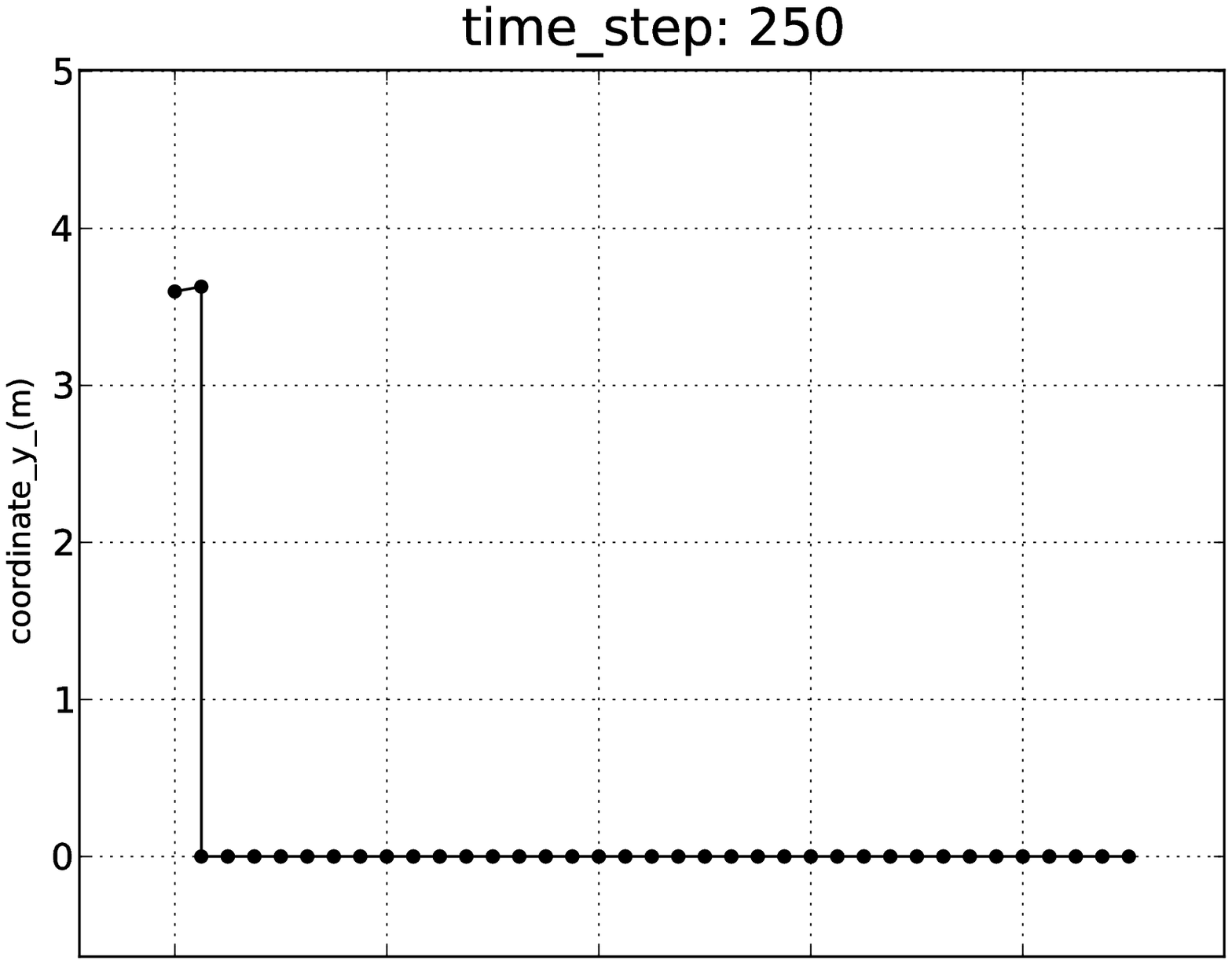}}&
\hspace*{-1.5em}\vspace*{-0em}{\includegraphics[width=.15\textwidth]{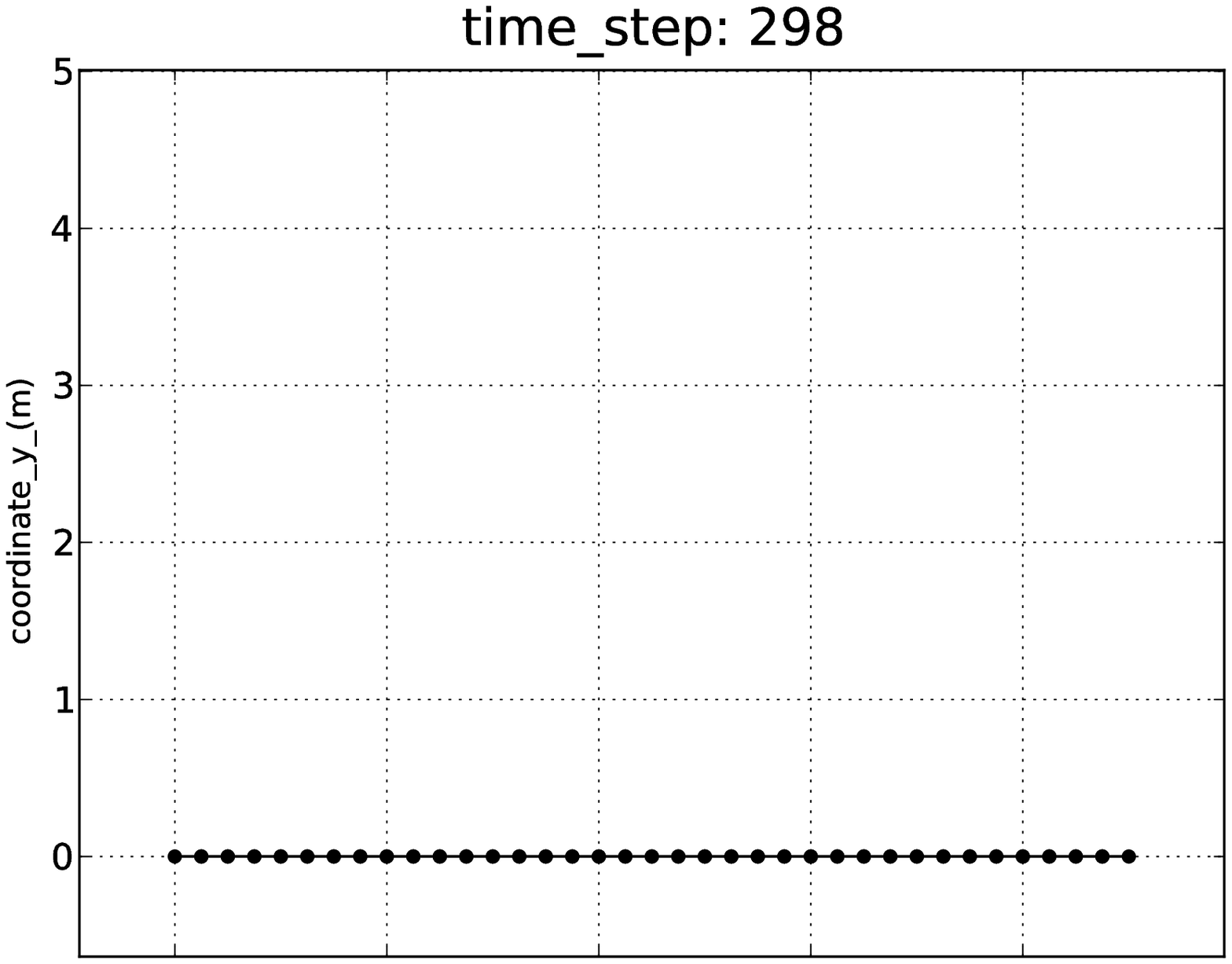}}
\\[-0em]
\hspace*{-1.5em}\vspace*{-0em}{\includegraphics[width=.13\textwidth]{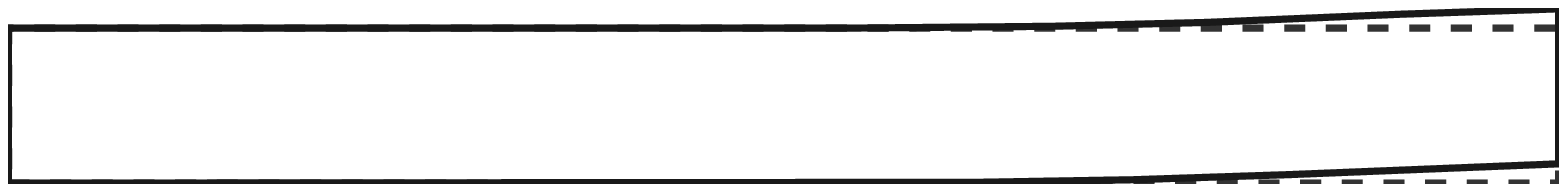}}&
\hspace*{-1.3em}\vspace*{-0em}{\includegraphics[width=.13\textwidth]{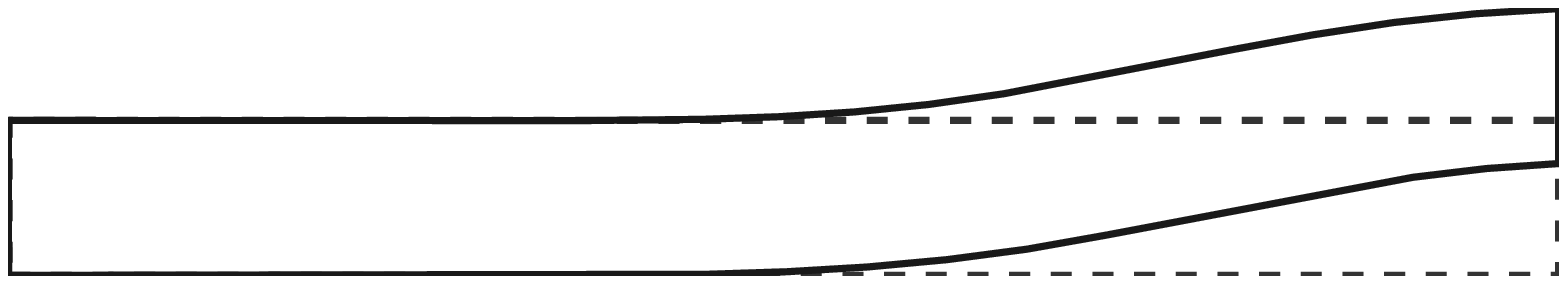}}&
\hspace*{-1.3em}\vspace*{-0em}{\includegraphics[width=.13\textwidth]{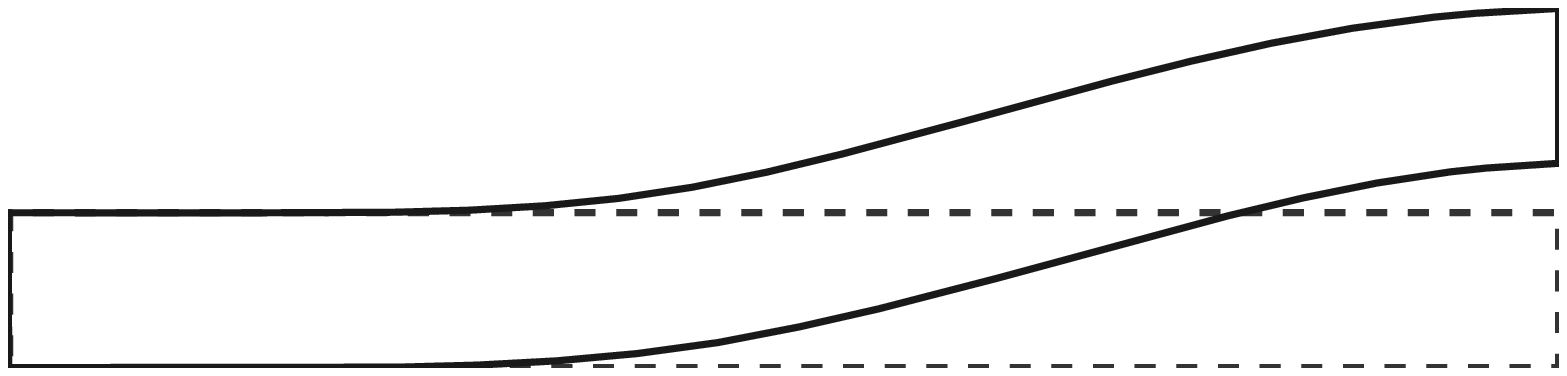}}&
\hspace*{-1.3em}\vspace*{-0em}{\includegraphics[width=.13\textwidth]{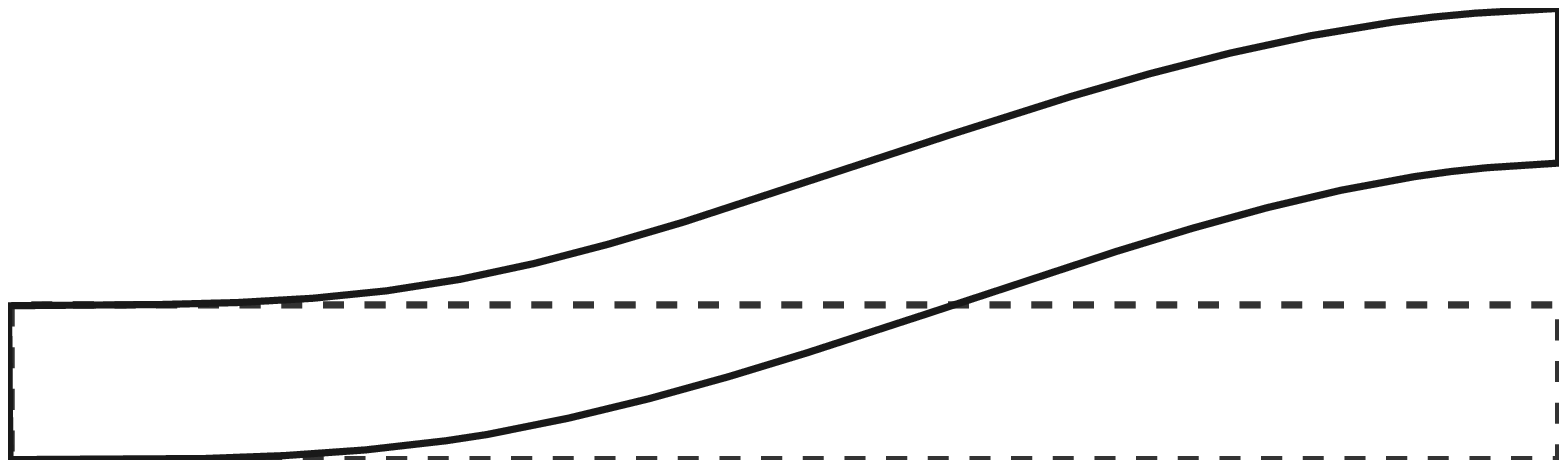}}&
\hspace*{-1.3em}\vspace*{-0em}{\includegraphics[width=.13\textwidth]{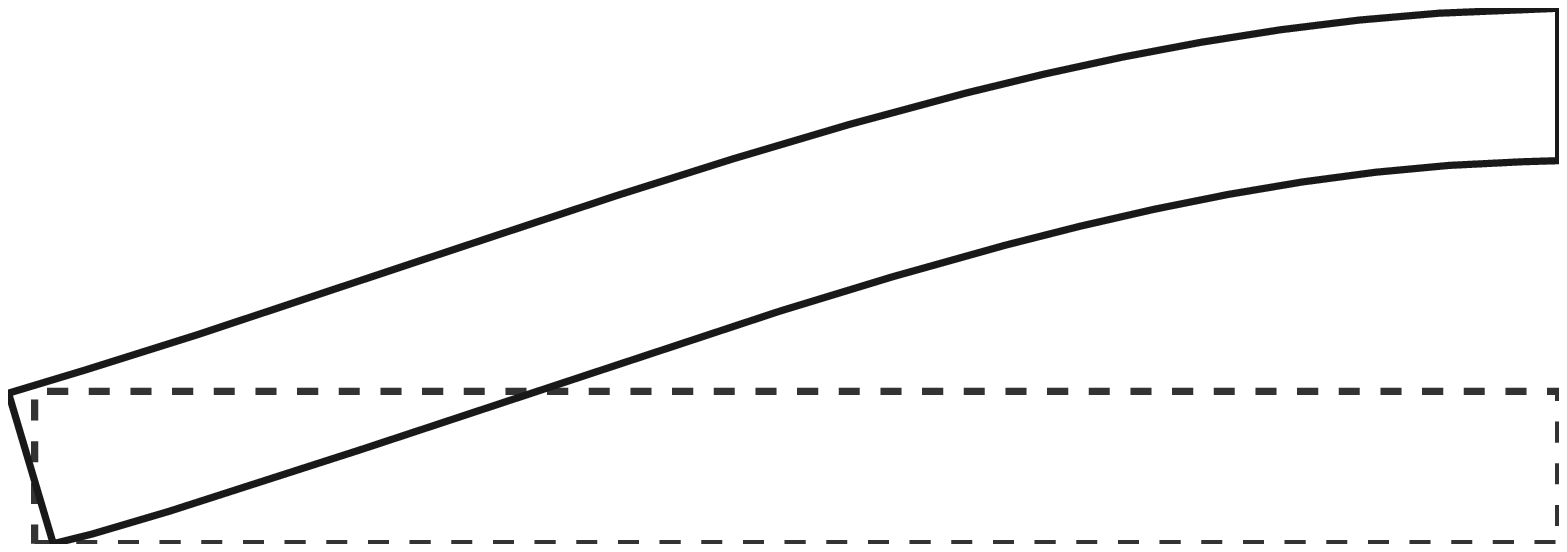}}&
\hspace*{-1.3em}\vspace*{-0em}{\includegraphics[width=.13\textwidth]{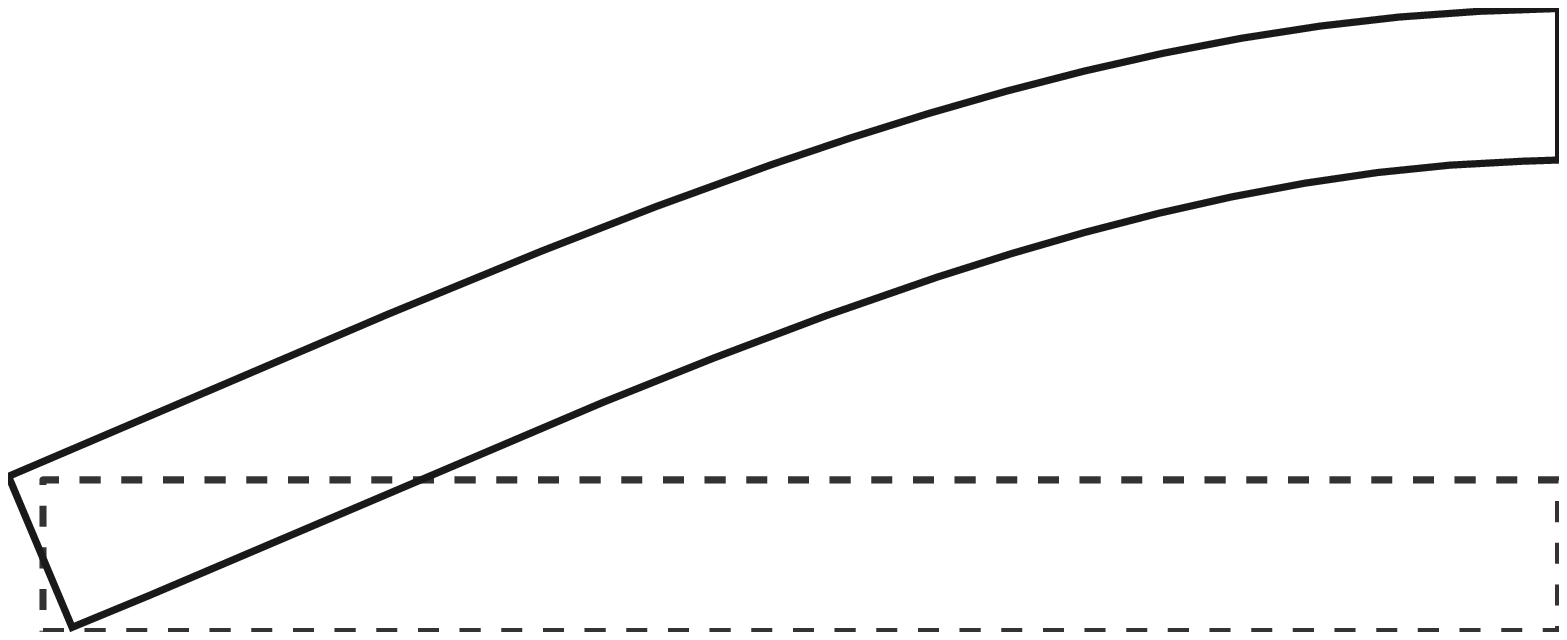}}&
\hspace*{-1.3em}\vspace*{-0em}{\includegraphics[width=.13\textwidth]{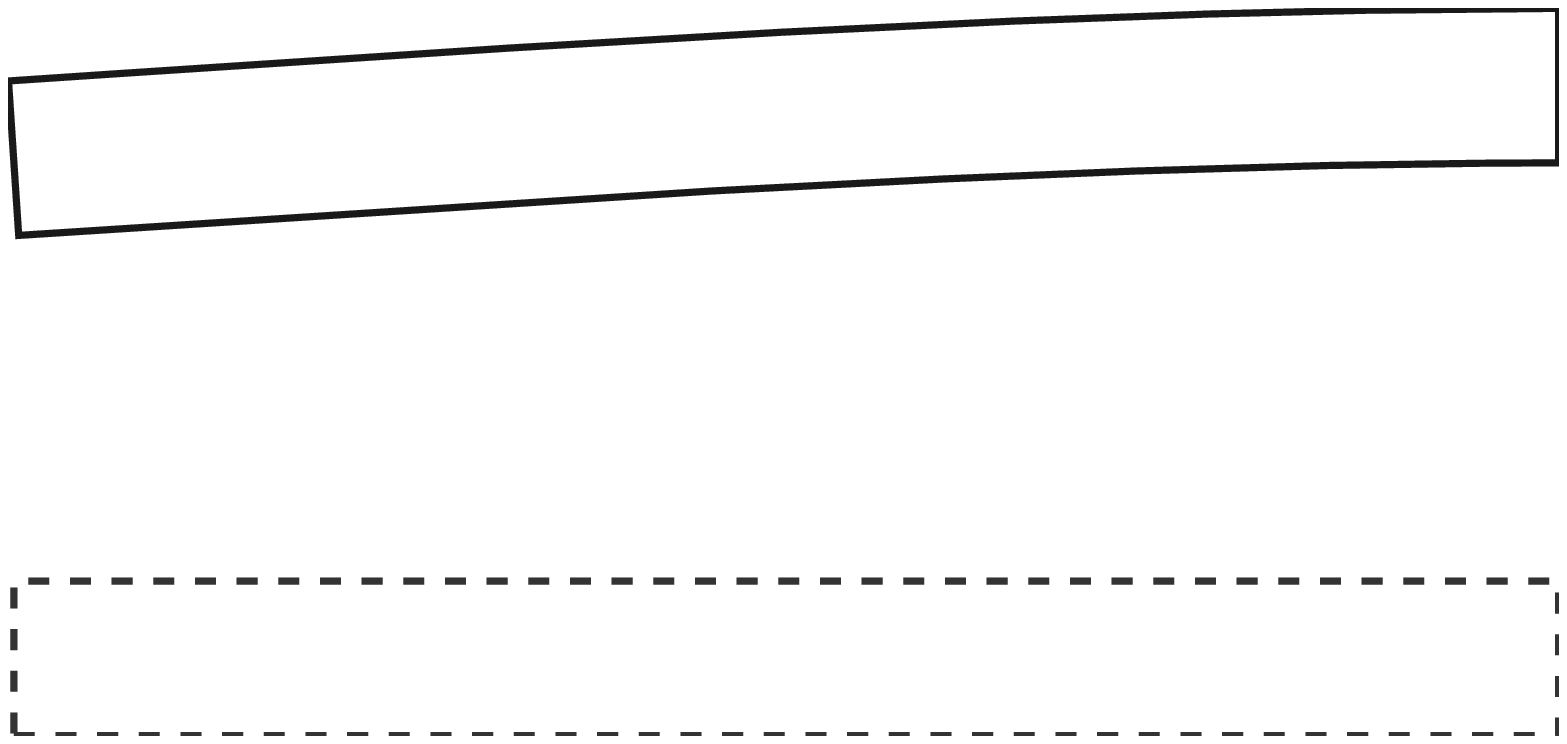}}
\end{tabular}
\end{my-picture}
\nopagebreak
\\[5em]
\begin{center}
{\small Fig.\,\ref{fig-movie-1}.\ }
\begin{minipage}[t]{.87\textwidth}\baselineskip=8pt
{\small 
Upper row: distribution of the normal traction force in the adhesive
along $\GC$.\\
Middle row: distribution of the tangential traction force in the adhesive
along $\GC$.\\
 Lower row:
deformed configuration of gradually delaminating specimen under loading
(2nd experiment) from Fig.\,\ref{fig_m1}; 
the displacement depicted 100$\,\times$  magnified. \\
}
\end{minipage}
\end{center}

\noindent
The analog of Fig.\,\ref{fig-2D-I-movie-2} is on Fig.\,\ref{fig-movie-2++}, 
showing again that the viscous energy (and also the defect measure) can be 
supported in the bulk far away from the delaminating surface $\GC$ and here 
even a tendency to surprising symmetry in spite of nonsymmetry of the boundary 
conditions: 

\begin{my-picture}{.6}{.55}{fig-movie-2++}
\psfrag{energy density rate, step: 6}{}
\psfrag{energy density rate, step: 9}{}
\psfrag{energy density rate, step: 12}{}
\psfrag{energy density rate, step: 15}{}
\psfrag{energy density rate, step: 18}{}
\psfrag{energy density rate, step: 21}{}
\begin{tabular}{lc} 
\hspace*{4em}\LARGE $^{^{^{^{\mbox{\footnotesize $t=\,$0.05}}}}}$  & 
\hspace*{1em}{\includegraphics[clip=true,width=.6\textwidth]{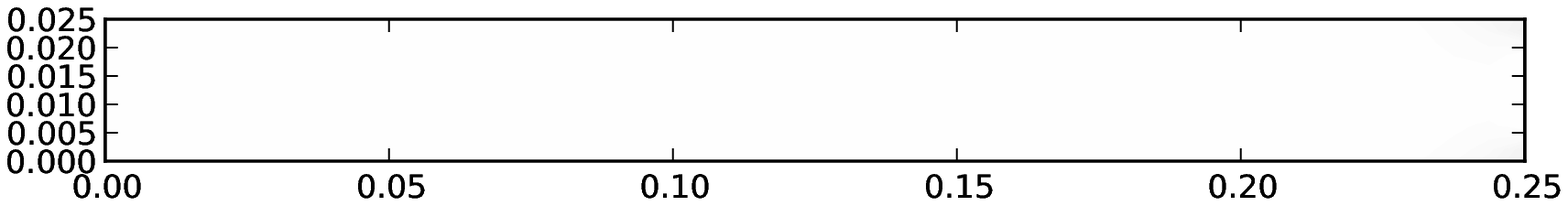}} 
\\[-1.25em]\hspace*{4em}\LARGE $^{^{^{^{\mbox{\footnotesize $t=\,$0.5}}}}}$ &
\hspace*{1em}{\includegraphics[clip=true,width=.6\textwidth]{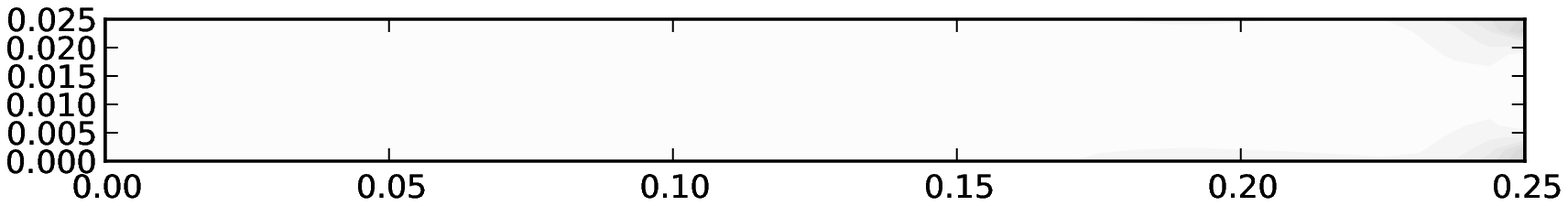}} 
\\[-1.25em]\hspace*{4em}\LARGE $^{^{^{^{\mbox{\footnotesize $t=\,$0.95}}}}}$ &
\hspace*{1em}{\includegraphics[clip=true,width=.6\textwidth]{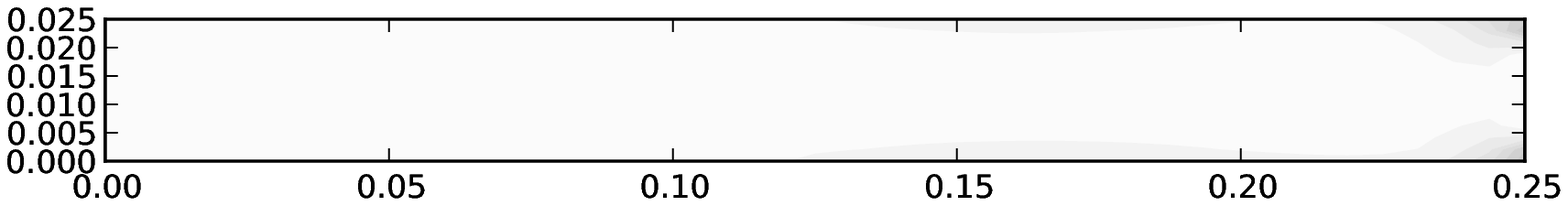}} 
\\[-1.25em]\hspace*{4em}\LARGE $^{^{^{^{\mbox{\footnotesize $t=\,$1.4}}}}}$ &
\hspace*{1em}{\includegraphics[clip=true,width=.6\textwidth]{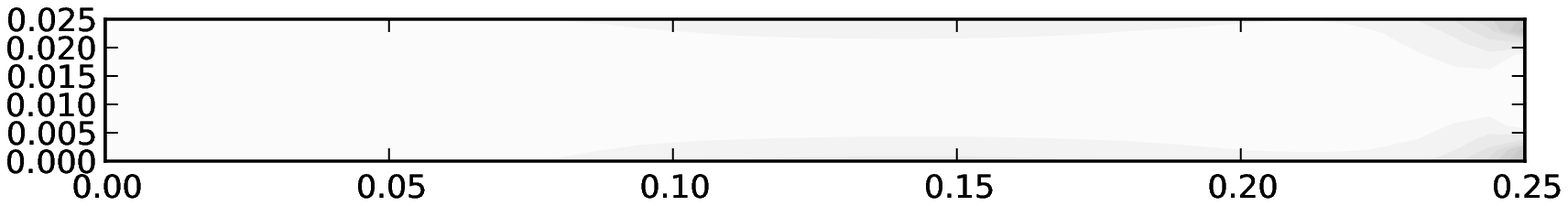}} 
\\[-1.25em]\hspace*{4em}\LARGE $^{^{^{^{\mbox{\footnotesize $t=\,$1.85}}}}}$ &
\hspace*{1em}{\includegraphics[clip=true,width=.6\textwidth]{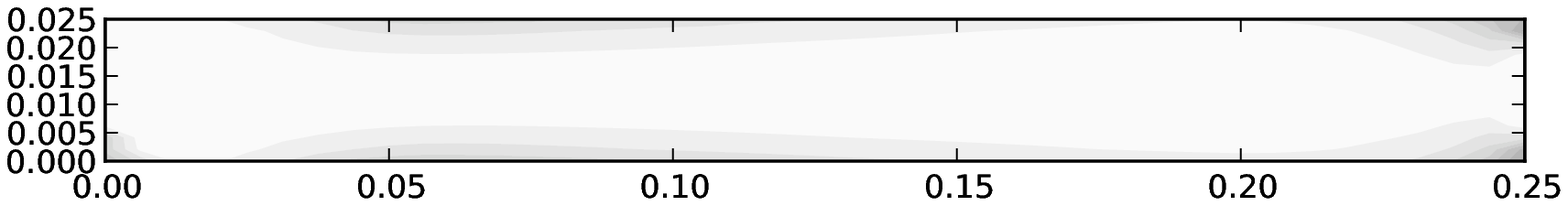}} 
\\[-1.25em]\hspace*{4em}\LARGE $^{^{^{^{\mbox{\footnotesize $t=\,$2.3}}}}}$ &
\hspace*{1em}{\includegraphics[clip=true,width=.6\textwidth]{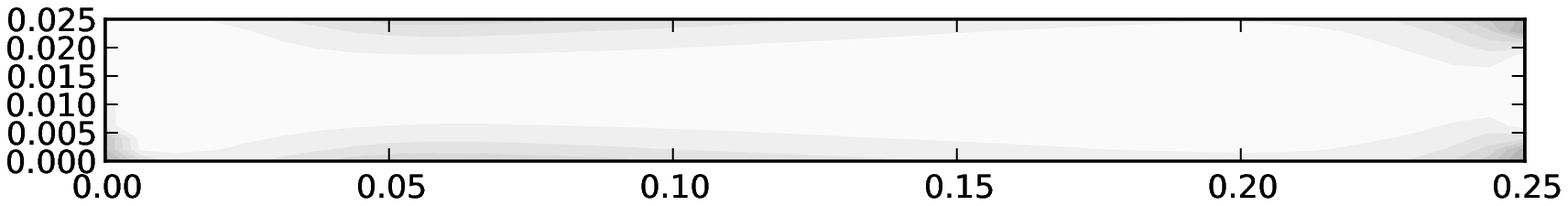}} 
\\[-1.25em]\hspace*{4em}\LARGE $^{^{^{^{\mbox{\footnotesize $t=\,$2.75}}}}}$&
\hspace*{1em}{\includegraphics[clip=true,width=.6\textwidth]{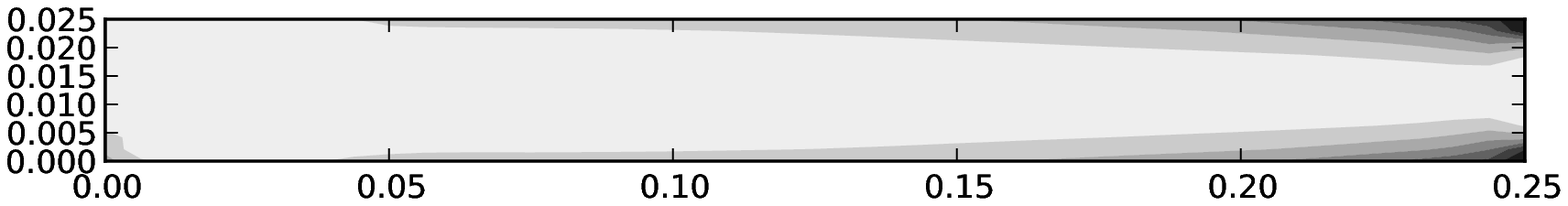}} 
\\[-.5em]&
\hspace*{0em}{\includegraphics[clip=true,width=.35\textwidth]{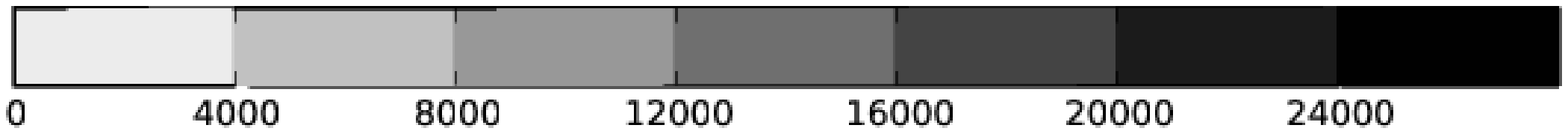}}
\\[-0em]
\end{tabular}
\end{my-picture}
\nopagebreak
\\[-1em]
\begin{center}
{\small Fig.\,\ref{fig-movie-2++}.\ }
\begin{minipage}[t]{.65\textwidth}\baselineskip=8pt
{\small The spatial distribution of the energy dissipated by viscosity over 
$[0,t]$, i.e.\ 
$\int_0^t\chi\bbC e(\DT u_{\chi,\tau}){:}e(\DT u_{\chi,\tau})\,\d t$ depicted 
at 6 selected time instances as on Fig.\,\ref{fig-movie-1}. Surprising 
tendency to a symmetry even under nonsymmetry loading can be observed.}
\end{minipage}
\end{center}

\noindent
Eventually, the force response corresponding to the previous 
Figures~\ref{fig-movie-1}--\ref{fig-movie-2++} is on 
Fig.\,\ref{fig-2D-engr-2}(left). Like on Fig.\,\ref{fig-2D-engr}, there is 
again a surprisingly good match if calculated by the simplified algorithm 
from Remark~\ref{rem-inviscid} as depicted  Fig.\,\ref{fig-2D-engr-2}(right),
although there is no theoretical guaranty of such force-response match
and obviously there is no match of energy.
\newpage
\begin{my-picture}{.9}{.25}{fig-2D-engr-2}
\psfrag{time}{\footnotesize\hspace*{4.5em}time\hspace{2em}$t$}
\psfrag{resultant_force}{\footnotesize\hspace{.5em}resultant force}
\psfrag{0.00005}{}
\hspace*{3em}\vspace*{-.1em}{\includegraphics[width=.4\textwidth,height=.25\textwidth]{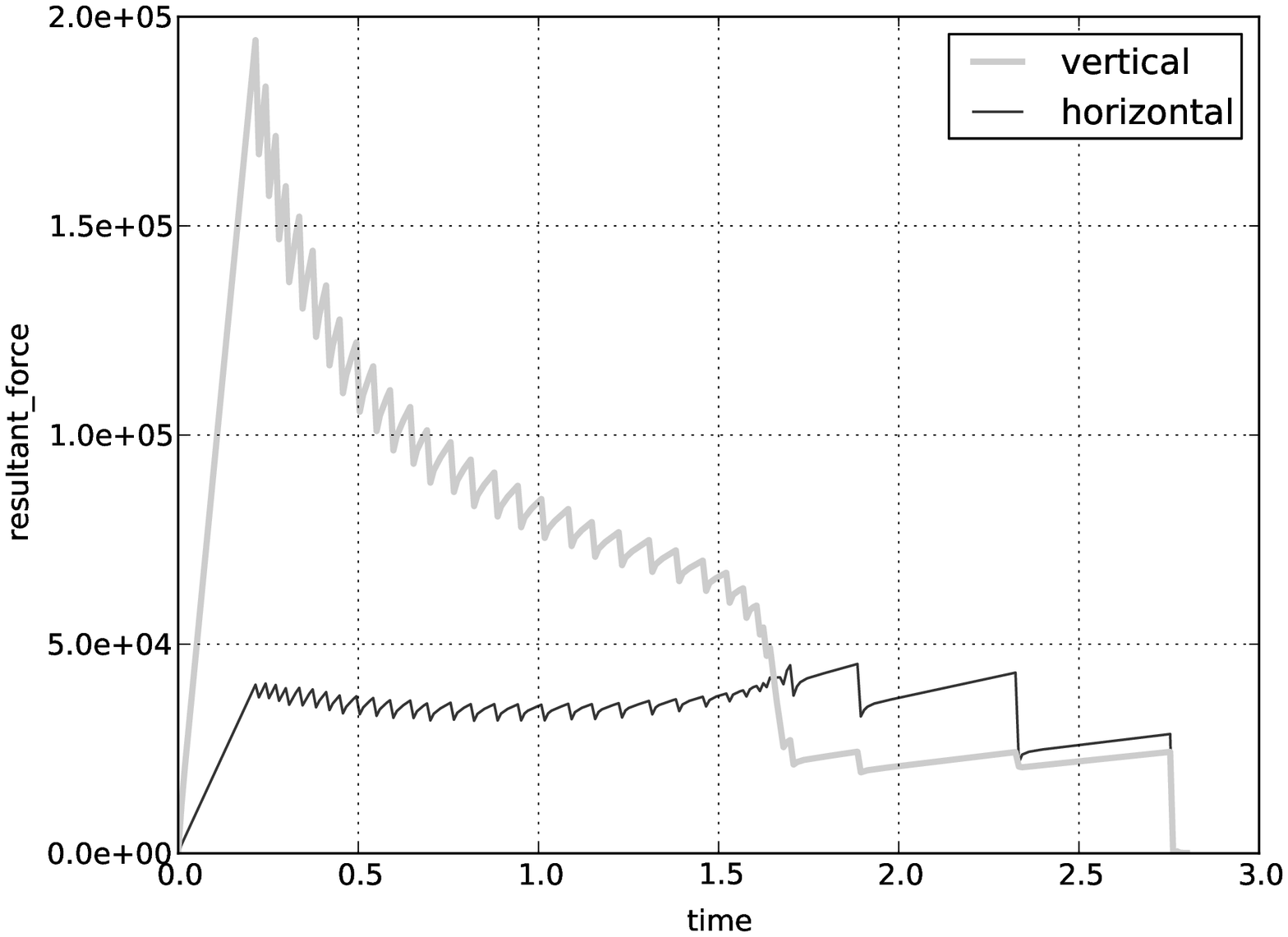}}
\hspace*{3em}\vspace*{-.1em}{\includegraphics[width=.4\textwidth,height=.25\textwidth]{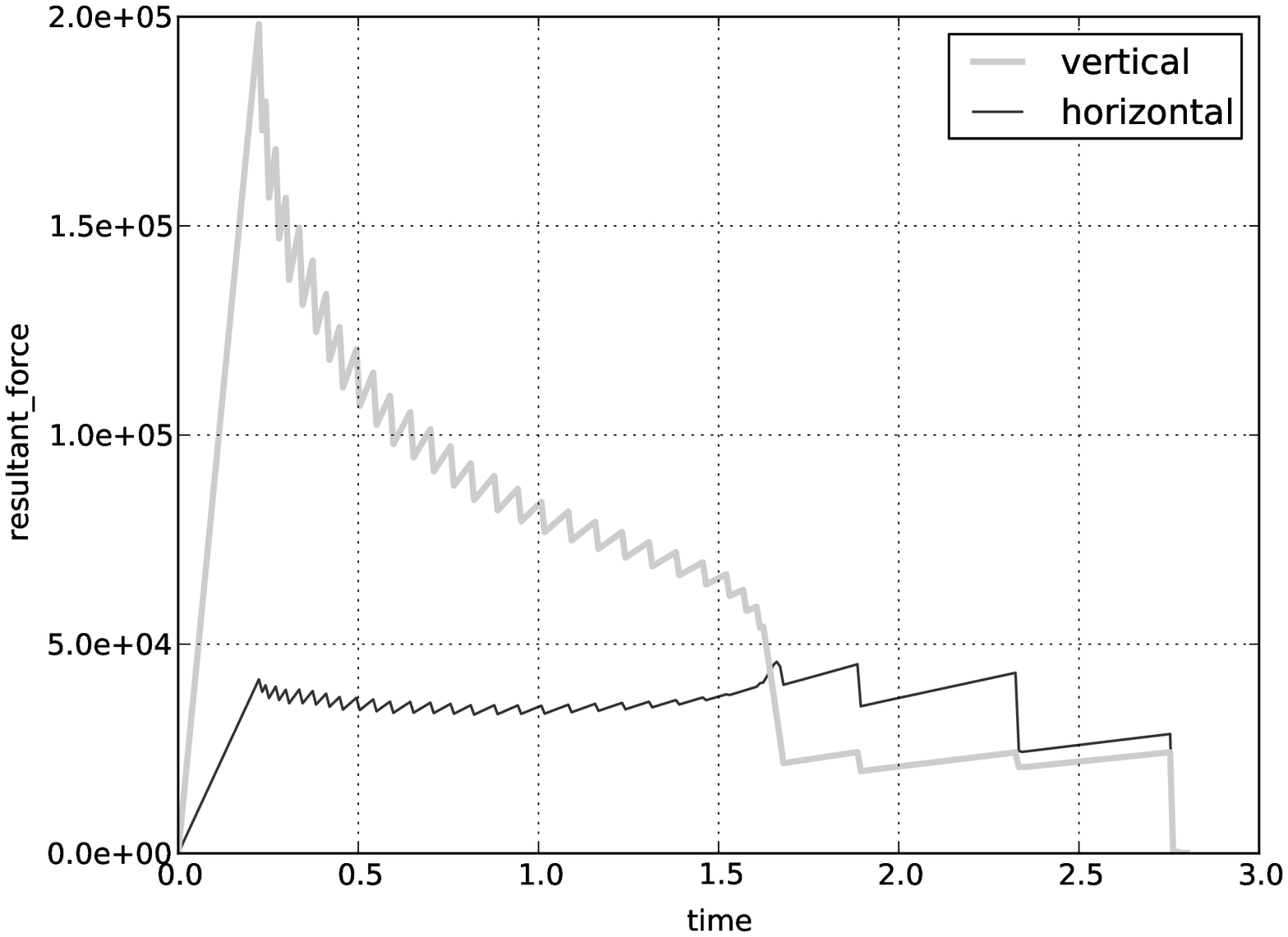}}
\end{my-picture}
\nopagebreak
\\[-2em]
\begin{center}
{\small Fig.\,\ref{fig-2D-engr-2}.\ }
\begin{minipage}[t]{.85\textwidth}\baselineskip=8pt
{\small
Vertical and horizontal components of the reaction force on the Dirichlet 
loading (left) and its comparison with the simplified inviscid algorithm from
Remark~\ref{rem-inviscid} (right), again showing a surprising match
as on Figures~\ref{fig:1-D-violation} and \ref{fig-2D-engr}.
}
\end{minipage}
\end{center}

\section{Conclusion}\label{sect-concl}

We devised and tested a numerical strategy to approximate 
the natural notion of solution to delamination of 
purely elastic material (as usually considered in engineering
applications). These solutions involve certain ``defect measures'' 
and follow the asymptotics 
arising from vanishing Kelvin-Voigt visco-elastic rheology 
which in the limit gives mere elastic material with the 
delamination driven naturally by stress rather than energy,
as devised purely theoretically in \cite{Roub??ACVE} without 
any time discretisation.

We showed a delicate interaction between vanishing viscosity
and time discretisation and difficulty to calculate physically
relevant solutions. After testing the algorithm on a 0-dimensional 
example where exact solution is known, we calculated a couple of 
nontrivial 2-dimensional examples by using BEM. 
Beside, we also compared the results with those obtained by a simplified 
inviscid algorithm ignoring defect measures and thus violating the 
energy balance, and showed a very good match of stress-strain responses
in all investigated particular cases, 
cf.\ Figs.\,\ref{fig:1-D-violation}, \ref{fig-2D-engr}, and 
\ref{fig-2D-engr-2}. Such an algorithm, called a Griffith
model, was advocated already in \cite{NegOrt08QSCP} although 
the desctruction of energy conservation was already pointed out there, 
and an investigation of some dynamical model leading to a correct limit
in the quasistatic evolution advised, cf.\ \cite[Sect.\,3.2]{NegOrt08QSCP}.
We conjecture that this simplified algorithm may converge 
to local solutions in the sense as introduced for a special crack 
problem in \cite{ToaZan09AVAQ} and further generally investigated in 
\cite{Miel11DEMF}, but the relation (and the phenomenon of good match) 
with the vanishing-viscosity approach remains still not justified.

The importance of the above presented methodology for calculation such defect 
measures would be pronounced 
in the full thermodynamical context like \cite{Eck02ESTV,RosRou11TARI},
cf.\ also \cite[Sect.\,5.4]{EcJaKr05UCP},
where the defect measure would naturally occur in the heat-transfer 
equation as a heat source and thus would influence temperature distribution
inside the body and then backward the
mechanics e.g.\ through thermal expansion or temperature dependence
of mechanical properties of the adhesive.
Interesting observation from Figures \ref{fig-2D-I-movie-2} and 
\ref{fig-movie-2++} is that the defect measure (and, the
possible heat production) may occur even in spots which are quite 
distant from the surface undergoing inelastic dissipative process of
delamination and it is certainly difficult (or rather impossible)
to guess its distribution by intuition.

\bigskip

\noindent
{\it Acknowledgments}: The authors thank Professor Alexander Mielke for 
discussion about the local-solution concept.

\bibliographystyle{abbrv}

\bibliography{trsevilla4-arXiv-preprint}

\begin{thebibliography}{10}

\bibitem{BarRou11TVER}
S.~Bartels and T.~Roub\'{\i}\v{c}ek.
\newblock Thermo-visco-elasticity with rate-independent plasticity in isotropic
  materials undergoing thermal expansion.
\newblock {\em Math. Model. Numer. Anal.}, 45:477--504, 2011.

\bibitem{Cagn08VVAF}
F.~Cagnetti.
\newblock A vanishing viscosity approach to fracture growth in a cohesive zone
  model with prescribed crack path.
\newblock {\em Math. Models Meth. Appl. Sci}, 18:1027--1071, 2009.

\bibitem{CFMT00RBFE}
M.~Charlotte, G.~Francfort, J.-J. Marigo, and L.~Truskinovsky.
\newblock Revisting brittle fracture as an energy minimization problem:
  comparison of {G}riffith and {B}arenblatt surface energy models.
\newblock In A.~Cachan, editor, {\em Continuous Damage and Fracture}, pages
  7--12, Paris, 2000. Elsevier.

\bibitem{DDMM08VVAQ}
G.~{Dal Maso}, A.~DeSimone, M.~G. Mora, and M.~Morini.
\newblock A vanishing viscosity approach to quasistatic evolution in plasticity
  with softening.
\newblock {\em Arch. Rational Mech. Anal.}, 189:469--544, 2008.

\bibitem{DipMaj87OCWS}
R.~J. DiPerna and A.~J. Majda.
\newblock Oscillations and concentrations in weak solutions of the
  incompressible fluid equations.
\newblock {\em Comm. Math. Phys.}, 108:667--689, 1987.

\bibitem{Eck02ESTV}
C.~Eck.
\newblock Existence of solutions to a thermo-viscoelastic contact problem with
  {C}oulomb friction.
\newblock {\em Math. Models Methods Appl. Sci.}, 12:1491--1511, 2002.

\bibitem{EcJaKr05UCP}
C.~Eck, J.~Jaru\v{s}ek, and M.~Krbec.
\newblock {\em Unilateral Contact Problems}.
\newblock Chapman \& Hall/CRC, Boca Raton, 2005.

\bibitem{EfeMie06RILS}
M.~Efendiev and A.~Mielke.
\newblock On the rate-independent limit of systems with dry friction and small
  viscosity.
\newblock {\em J. Convex Anal.}, 13:151--167, 2006.

\bibitem{Feir03DCF}
E.~Feireisl.
\newblock {\em Dynamics of {C}ompressible {F}low}.
\newblock Clarendon Press, Oxford, 2003.

\bibitem{Fias??VVAQ}
A.~Fiaschi.
\newblock A vanishing viscosity approach to a quasistatic evolution problem
  with nonconvex energy.
\newblock {\em Ann. Inst. H. Poincar{\'{e}}, Anal. Nonlin.}, 26:1055--1080,
  2009.

\bibitem{Frem85DAS}
M.~Fr\'emond.
\newblock Dissipation dans l'adh\'{e}rence des solides.
\newblock {\em C.R. Acad. Sci., Paris, S\'{e}r.{I}{I}}, 300:709--714, 1985.

\bibitem{Gera91MDM}
P.~G{\'e}rard.
\newblock Microlocal defect measures.
\newblock {\em Comm. Partial Differential Equations}, 16:1761--1794, 1991.

\bibitem{Gren95DMVP}
E.~Grenier.
\newblock Defect measures of the {V}lasov-{P}oisson system in the quasineutral
  regime.
\newblock {\em Comm. Partial Diff. Equations}, 20:1189--1215, 1995.

\bibitem{HsiWen08BIE}
G.~Hsiao and W.~Wendland.
\newblock {\em Boundary Integral Equations}.
\newblock Springer, Berlin, 2008.

\bibitem{KnMiZa08ILMC}
D.~Knees, A.~Mielke, and C.~Zanini.
\newblock On the inviscid limit of a model for crack propagation.
\newblock {\em Math. Models Meth. Appl. Sci.}, 18:1529--1569, 2008.

\bibitem{LazToa??MCPB}
G.~Lazzaroni and R.~Toader.
\newblock A model for crack propagation based on viscous approximation.
\newblock {\em Math. Models Meth. Appl. Sci.}, 21:2019--2047, 2011.

\bibitem{Legu02STCC}
D.~Leguillon.
\newblock Strength or toughness? {A} criterion for crack onset at a notch.
\newblock {\em European J. of Mechanics A/Solids}, 21:61–72, 2002.

\bibitem{Miel05ERIS}
A.~Mielke.
\newblock Evolution in rate-independent systems ({C}h.~6).
\newblock In C.~Dafermos and E.~Feireisl, editors, {\em Handbook of
  Differential Equations, Evolutionary Equations, vol.~2}, pages 461--559.
  Elsevier B.V., Amsterdam, 2005.

\bibitem{Miel11DEMF}
A.~Mielke.
\newblock Differential, energetic and metric formulations for rate-independent
  processes.
\newblock In L.~Ambrosio and G.~Savar\'e, editors, {\em Nonlinear PDE's and
  Applications}, pages 87--170. Springer, Berlin, 2011.

\bibitem{MiRoSa09MSJR}
A.~Mielke, R.~Rossi, and G.~Savar\'{e}.
\newblock Modeling solutions with jumps for rate-independent systems on metric
  spaces.
\newblock {\em Discr. Cont. Dynam. Systems Ser.~A}, 25:585--615, 2009.

\bibitem{MiRoSa09?BVSV}
A.~Mielke, R.~Rossi, and G.~Savar\'{e}.
\newblock {B}{V} solutions and viscosity approximations of rate-independent
  systems.
\newblock {\em ESAIM Control Optim. Calc. Var.}, 18(1):36--80, 2012.

\bibitem{MieThe04RIHM}
A.~Mielke and F.~Theil.
\newblock On rate-independent hysteresis models.
\newblock {\em Nonl. Diff. Eqns. Appl.}, 11:151--189, 2004.

\bibitem{MiThLe02VFRI}
A.~Mielke, F.~Theil, and V.~I. Levitas.
\newblock A variational formulation of rate--independent phase transformations
  using an extremum principle.
\newblock {\em Arch. Rat. Mech. Anal.}, 162:137--177, 2002.

\bibitem{Naum06ETWS}
J.~Naumann.
\newblock An existence theorem for weak solutions to the equations of
  non-stationary motion of heat-conducting incompressible viscous fluids.
\newblock {\em Math. Meth. Appl. Sci.}, 29:1883--1906, 2006.

\bibitem{NegOrt08QSCP}
M.~Negri and C.~Ortner.
\newblock Quasi-static crack propagation by {G}riffith's criterion.
\newblock {\em Math. Models Methods Appl. Sci.}, 18(11):1895--1925, 2008.

\bibitem{RosRou11TARI}
R.~Rossi and T.~Roub{\'\i}{\v{c}}ek.
\newblock Thermodynamics and analysis of rate-independent adhesive contact at
  small strains.
\newblock {\em Nonlinear Anal.}, 74:3159--3190, 2011.

\bibitem{Roub09RIPV}
T.~Roub{\'\i}{\v{c}}ek.
\newblock Rate independent processes in viscous solids at small strains.
\newblock {\em Math. Methods Appl. Sci.}, 32:825--862, 2009.

\bibitem{Roub12NPDE}
T.~Roub{\'\i}{\v{c}}ek.
\newblock {\em Nonlinear Partial Differential Equations with Applications}.
\newblock 2nd Ed., Birk\-h\"auser, Basel, 2013.

\bibitem{Roub??ACVE}
T.~Roub{\'\i}{\v{c}}ek.
\newblock Adhesive contact of visco-elastic bodies and defect measures arising
  by vanishing viscosity.
\newblock {\em SIAM J. Math. Anal.}, in print, DOI. 10.1137/12088286X.

\bibitem{Stef09VCRI}
U.~Stefanelli.
\newblock A variational characterization of rate-independent evolution.
\newblock {\em Mathem. Nach.}, 282:1492--1512, 2009.

\bibitem{ToaZan09AVAQ}
R.~Toader and C.~Zanini.
\newblock An artificial viscosity approach to quasistatic crack growth.
\newblock {\em Boll. Unione Matem. Ital.}, 2:1--36, 2009.

\end{thebibliography}

\end{document}